\documentclass{amsart}

\usepackage[ansinew]{inputenc}

\usepackage[margin=3cm]{geometry}
\usepackage{changepage}

\usepackage{amsmath}
\usepackage{amsfonts}
\usepackage{amsthm}
\usepackage{amssymb}
\usepackage{bbold}
\usepackage{graphicx}
\usepackage{tikz}
\usepackage{enumerate}
\usepackage{multirow,longtable,float}
\usepackage{tikz-cd} 
\tikzset{
	labl/.style={anchor=south, rotate=90, inner sep=.5mm}
}

\makeindex
\usepackage[
pdftex,
pdfstartview=FitH,
bookmarks=true,
bookmarksdepth=section,
bookmarksopen=false,
bookmarksnumbered=true
]{hyperref}

\usepackage{aliascnt}

\usepackage{array} 
\usepackage[T1]{fontenc} 
\newcommand{\subtile}[1] 
{
	\vspace{-0.3cm}
	\begin{center}
		{{\textsc{#1}}}\\
	\end{center}
	\vspace{0.1cm}
}

\newcommand{\imageplain}[2] 
{
	{\includegraphics[scale=#1]{#2.png}}
}
\newcommand{\image}[2] 
{
	\begin{center} 
		\includegraphics[scale=#1]{#2.png}
	\end{center}
}
\newcommand{\hamburger}[4] 
{
	\thispagestyle{empty}
	\vspace*{-2cm}
	\begin{flushright}
		ZMP-HH / #2 \\
		Hamburger Beitr{\"a}ge zur Mathematik Nr. #3 \\
		#4 \\
	\end{flushright}
	\vspace{0.5cm}
	\begin{center}
		\Large \bf
		#1
	\end{center}
	\vspace{0.5cm}
	\begin{center}	
		Azat M. Gainutdinov${}^{\dagger *}$ and Simon Lentner${}^{*}$\footnote{Corresponding author, simon.lentner@uni-hamburg.de} and Tobias Ohrmann${}^{*}$\\
		\vspace{1.0cm}
${}^*$ Algebra and Number Theory, 
		University of Hamburg,\\
		Bundesstra{\ss}e 55, D-20146 Hamburg, Germany\\
		 $^{\dagger}$  Laboratoire de Math\'ematiques et Physique Th\'eorique CNRS/UMR 7350,\\
 F\'ed\'eration Denis Poisson FR2964,
Universit\'e de Tours,\\
Parc de Grammont, 37200 Tours, 
France\\
	\end{center}
	\vspace{0.5cm}
	
}

\makeatletter
\newcommand{\xleftrightarrow}[2][]{\ext@arrow 3359\leftrightarrowfill@{#1}{#2}}
\newcommand{\xdashrightarrow}[2][]{\ext@arrow 0359\rightarrowfill@@{#1}{#2}}
\newcommand{\xdashleftarrow}[2][]{\ext@arrow 3095\leftarrowfill@@{#1}{#2}}
\newcommand{\xdashleftrightarrow}[2][]{\ext@arrow
	3359\leftrightarrowfill@@{#1}{#2}}
\def\rightarrowfill@@{\arrowfill@@\relax\relbar\rightarrow}
\def\leftarrowfill@@{\arrowfill@@\leftarrow\relbar\relax}
\def\leftrightarrowfill@@{\arrowfill@@\leftarrow\relbar\rightarrow}
\def\arrowfill@@#1#2#3#4{%
	$\m@th\thickmuskip0mu\medmuskip\thickmuskip\thinmuskip\thickmuskip
	\relax#4#1
	\xleaders\hbox{$#4#2$}\hfill
	#3$%
}
\makeatother

\allowdisplaybreaks[1]

\newcommand{\Hom}{\mbox{Hom}}

\newcommand{\Cent}{\mbox{Cent}}

\newcommand{\id}{\mbox{id}}

\newcommand{\B}{\mathcal{B}} 	
\newcommand{\g}{\mathfrak{g}} 	
\renewcommand{\sl}{\mathfrak{sl}}	
\newcommand{\Z}{\mathbb{Z}}  	
\newcommand{\N}{\mathbb{N}}  	
\newcommand{\C}{\mathbb{C}}  	
\newcommand{\Q}{\mathbb{Q}}  	

\newcommand{\CC}{\mathcal{C}}   
\newcommand{\V}{\mathcal{V}}	
\newcommand{\W}{\mathcal{W}} 

\newcommand{\Beta}{\mathrm{B}}
\newcommand{\Vect}{\mathsf{Vect}}
\newcommand{\Rep}{\mathsf{Rep}}

\theoremstyle{plain}
\newtheorem{theorem}{Theorem}[section]

\newtheorem{conjecture}[theorem]{Conjecture}
\newtheorem{corollary}[theorem]{Corollary}

\newtheorem{definition}[theorem]{Definition}

\newtheorem{example}[theorem]{Example}

\newtheorem{lemma}[theorem]{Lemma}

\newtheorem{proposition}[theorem]{Proposition}
\newtheorem{remark}[theorem]{Remark}

\newtheorem{defprop}[theorem]{Definition-Proposition}
\newtheorem{convention}[theorem]{Convention}

\begin{document}
	
		\enlargethispage{8\baselineskip}
	\hamburger{Modularization of small quantum groups}
	{18-18}{748}{September 2018}
	
	\vspace{1.2cm}
	\begin{center}
		\begin{minipage}{0.9\textwidth}
			\textbf{Abstract. } We construct a large family of ribbon quasi-Hopf algebras related to small quantum groups, with a factorizable $R$-matrix. Our main purpose is to obtain non-semisimple modular tensor categories for quantum groups at even roots of unity, where typically the initial representation category is not even braided. 
Our quasi-Hopf algebras are built from modules over the twisted Drinfeld double via a universal construction, but we also work out explicit generators and relations, and we prove that these algebras are modularizations of the quantum group extensions with $R$-matrices listed in~\cite{LentnerOhrmann2017}.\\		
			As an application, we find one distinguished factorizable quasi-Hopf algebra for any finite root system and any root of unity of even order (resp. divisible by $4$ or $6$, depending on the root length). 
			Under  the same divisibility condition on a rescaled root lattice, a corresponding lattice Vertex-Operator Algebra contains a  VOA $\W$ defined as the kernel of screening operators. We then conjecture that $\W$ representation categories are braided equivalent to the representation categories  of the distinguished factorizable quasi-Hopf algebras.
For $A_1$ root system, our construction specializes to the quasi-Hopf algebras in \cite{GainutdinovRunkel2017,CreutzigGainutdinovRunkel2017}, where the answer is affirmative, similiary for $B_n$ at fourth root of unity in \cite{FarsadGainutdinovRunkel2017a, FlandoliLentner2017}.
		\end{minipage}
	\end{center}
	\vspace{0.5cm}
	
	\newpage
	
	\tableofcontents
	
	\newpage

	\section{Introduction}
	
	A finite-dimensional Hopf algebra (over a field $k$) with $R$-matrix is an algebra with additional structure, such that the (finite) representation category has a tensor product, dualities and a braiding. If the braiding is non-degenerate in the sense that double braiding with any non-trivial object is not the identity, then the Hopf algebra is called \emph{factorizable} \cite{ReshetikhinSemenov-Tian-Shansky1988} and its braided tensor category is also  factorizable~\cite{Shimizu2016}. When the Hopf algebra is additionally ribbon, the name \emph{modular tensor category} is also used. Note that we use this term also for non-semisimple categories \cite{KerlerLyubashenko2001,Shimizu2016}.
	
	Drinfeld and Jimbo \cite{Drinfeld1987,Jimbo1985} defined the \emph{quantum group} $U_q(\g)$ associated to a semisimple finite-dimensional Lie algebra $\g$ as a deformation of the (infinite-dimensional) universal enveloping algebra $U(\g)$ over the field of functions $\Q(q)$, and  the corresponding representation category is an (infinite) braided tensor category.  
	
	Lusztig \cite{Lusztig1990} defined a specialization to $q$ an $\ell$-th root of unity, and inside it discovered the \emph{small quantum group} $u_q(\g)$, which is a finite-dimensional non-semisimple Hopf algebra.

	For $q$ a root of unity of odd order $\ell$ the Hopf algebra $u_q(\g)$ has a factorizable $R$-matrix \cite{Lusztig1993} and its representation theory is closely related to the representation theory of $\g$ in characteristic $\ell$ (if prime). 
	
	If $\ell$ has common divisors with root lengths and/or the determinant of the Cartan matrix, the situation is more complicated -- some are not factorizable, some do not even have $R$-matrices (e.g. \cite{KondoSaito2011} for $\sl_2$ with $\ell$ even). In \cite{LentnerNett2015,LentnerOhrmann2017}, it is shown how this problem is resolved by extending the Cartan part of the quantum group and we have developed tools to compute a list of such quantum group extensions with different $R$-matrices and check which of them are factorizable (unfortunately, not many are).
	
	The idea of this article is to pass from these (non-factorizable) braided tensor categories to their respective \textsl{modularizations} (when possible) and thus to obtain a large family of modular tensor categories for quantum groups if $\ell$ has problematic divisibilities. We recall that in the semi-simple situation, a ribbon category is called modularizable if the ribbon twist is trivial on the subcategory of transparent objects. Such a category can be then modularized, i.e. it admits a dominant ribbon functor onto a modular category, and it was shown~\cite{Bruguieres2000} that such a functor is unique, up to an equivalence. It is however awaits a proper generalization to non-semisimple (finite) braided tensor categories, some steps in this direction were made in~\cite{FuchsGainutdinovSchweigert}. We nevertheless provide explicitly such a modularization functor in the case of our non-factorizable quantum groups.
	This resulted in constructing quasi-Hopf algebras, as we show below related again to small quantum groups but with another choice of the Cartan subalgebra. These quasi-Hopf algebras are factorizable and ribbon by the construction, and in the following we call them \emph{quasi-quantum groups}.
	
	We should note that the idea of constructing quasi-Hopf algebras as quotients of ordinary Hopf algebras appears first in \cite{EtingofGelaki2005}, along these lines \cite{HuangLiuYuYe2015} has classified Nichols algebras in the quasi-setting and \cite{AngionoGalindoPereira2014} has constructed (co)quasi Hopf algebras, without regard to $R$-matrices. 
	
	Here, we give a direct construction of \textsl{quasi-triangular} quasi-Hopf algebras from Nichols algebras in semi-simple braided categories with non-trivial associator. While our proofs rely on the universal properties of this construction, we work out for convenience also generators and relations. Then we prove that for a careful choice of parameters involved (an explicit list is provided) we obtain factorizable quasi-Hopf algebras, whose representation category is indeed the modularization of the respective quantum group extension in \cite{LentnerOhrmann2017} possessing $R$-matrices.\\ 
	
	The main application that drove our work concerns the braided tensor categories arising from Vertex Operator Algebras (VOAs) 
	 defined as kernel of screenings \cite{FuchsHwangSemikhatovEtAl2004,FeiginTipunin2010}. They are defined for any $\g$ and $\ell$ satisfying certain divisibility condition and denoted by $\W(\g,\ell)$. These VOAs are conjectured to be $C_2$-cofinite \cite[Conj.\,3.2]{AdamovicMilas2014} and thus having finite representation categories. For the $\mathfrak{sl}_2$-case, this is proven in \cite{CarquevilleFlohr2006}. Moreover, they posses a  braided tensor category structure~\cite{HuangLepowskyZhang}, which is also  believed to be modular. 
	 Important conjectures link them to the representation categories of the respective small quantum groups $u_q(\g)$. (This is discussed more in Section~\ref{sec:CFT}.) However, it is known that in general this can not be an equivalence of braided tensor categories (the quantum group here does not even have an $R$-matrix). Recently, a factorizable quasi-Hopf algebra variant of $u_q(\g)$ for $\g=\sl_2$  was proposed \cite{GainutdinovRunkel2017,CreutzigGainutdinovRunkel2017} (using ideas of VOA extensions)  as a correct replacement of $u_q(\g)$,   for which conjecturally there is an equivalence of ribbon categories with the corresponding VOA. 
	
	As our last result, we show that our construction indeed returns a surprisingly uniform family of factorizable ribbon quasi-Hopf algebras for every $\g$ with $\ell$ satisfying the divisibility constraints. In the  example $\g=\sl_2$, we get indeed the previously constructed ribbon quasi-Hopf algebra~\cite{CreutzigGainutdinovRunkel2017}.
	
			Another possible application of our construction of quasi-Hopf algebras via the modularization concerns their use in recent developments in low-dimensional topology. For example, a construction~\cite{GeerPatureau-Mirand2013}  of  topological invariants begins with quantum groups at even roots of unity (and uses so-called  modified traces). Here the non-existence of an $R$-matrix for these quantum groups is overcome by passing to the (infinite-dimensional) unrolled quantum group. We conjecture that this is not necessary if working directly with the ``correct'' quasi-quantum group\footnote{We should mention that a second effect of their definition (omitting the $K$-relation) returns a very large graded category and here we just talk about the $0$-grade part. We would assume that these two effects are in fact unrelated, and one could work with a Kac-DeConcini-Procesi version of our quasi-quantum group to get the full category.}.			
			Also, a recent construction~\cite{BeliakovaBlanchetGeer2017} of invariants of $3$-manifolds from the representation category of $u_q(\mathfrak{sl}_2)$ was based on the existence of a monodromy matrix rather than an $R$-matrix. And here again, the use of the corresponding quasi-quantum group modification $\bar{u}_q(\mathfrak{sl}_2)$ with an $R$-matrix could make the construction more straightforward.


	\medskip
	We now summarize the results of this article in more detail (as some parts are very technical, we decided to describe here our construction step by step):

		\medskip
	In Section \ref{sec: preliminaries} we give a brief introduction to quasi-Hopf algebras. The main example underlying our constructions is the quasi-Hopf algebra $k^{G}_{\omega}$ with $R$-matrix associated to an abelian group $G$ and an abelian $3$-cocycle $(\sigma,\omega)$ \cite{MacLane1952}. The corresponding braided tensor category of representations are $G$-graded vector spaces $\Vect_G^{\omega,\sigma}$ with non-trivial associator $\omega$ and braiding $\sigma$. We should mention that we have an additional freedom $\eta: G\to \{\pm1\}$ in the choice of our ribbon structure.
	
	Abelian $3$-cocycles up to abelian coboundary are most conveniently given by a quadratic form $Q$ on $G$, and the associated symmetric bilinear form $B$ on $G$ is non-degenerate iff the braiding is nondegenerate and thus the category is a modular tensor category. Otherwise the radical of $B$ determines the objects with trivial double-braiding (transparent objects).\\
	
	In Section \ref{sec: modularization of Vec} we prove for later use the nice (and rather obvious) fact, that the braided tensor category $\Vect_G^{(\omega,\sigma,\eta)}$ has a modularization iff $Q$ and $\eta$ are trivial on the radical of $Rad(B)$. The corresponding modular tensor category is then $\Vect_{\bar{G}}^{(\bar{\omega},\bar{\sigma},\bar{\eta})}$ with $\bar{G}=G/Rad(B)$ and $(\omega,\sigma,\eta)$ factorized to the quotient. In some sense this is the seed for our modularization of the quantum group extensions, where the group algebra $kG$ plays the role of the Cartan subalgebra.\\
	
	In Section \ref{Quantum groups uq(g,Lambda) and R-matrices} we collect the data and results for our initial small quantum group $u_q(\g,\Lambda)$ and its $R$-matrices: Let $\g$ be a finite-dimensional semisimple complex Lie algebra $\g$, let $q$ be an $\ell$-th root of unity $q$ and let $\Lambda$ be some lattice between the root-lattice $\Lambda_R$ and the weight-lattice $\Lambda_W$ of $\g$. Note that this corresponds to some choice of the Lie group between adjoint and simply-connected form, where $\Lambda/\Lambda_R$ is the corresponding center, realized as a subgroup of the fundamental group of the adjoint form $\Lambda_W/\Lambda_R$. 
	
	The Cartan part of this Hopf algebra is the abelian group $G=\Lambda/\Lambda'$ where $\Lambda'$  is the centralizer of $\Lambda_R$ with respect to $\Lambda$. The lattice $\Lambda'$ encodes the $K$-relations, and it was proven in \cite{LentnerOhrmann2017} that only one specific choice for $\Lambda'$ has a chance to support an $R$-matrix.    
	
	Our set of potential $R$-matrices is now parametrized by a bicharacter $g$ on $H=\Lambda/\Lambda_R$ in the fundamental group. To be precise, $g$ can be defined on two subgroups $H_1\times H_2$ with  same cardinality and $H_1+H_2=H$, and  we have $H_1=H_2=H$ except for the case $\g=D_{2n}=\mathrm{SO}_{4n}$ with non-cyclic fundamental group. 
	
	The essential varying piece of our $R$-matrix (following Lusztig's ansatz) depending on $g$ is  
	$$R_0=\sum_{\mu,\nu\in G} g(\mu,\nu)q^{-(\mu,\nu)}K_\mu\otimes K_\nu$$
	We view $R_0$ as an $R$-matrix for $k^G$, respectively giving braiding in $\Vect_G$. Two non-trivial results in \cite[Cor.\,5.20 and Cor.\,5.21]{LentnerOhrmann2017}  are that the $R$-matrix is factorizable iff $R_0$ is factorizable iff $G$ has no $2$-torsion.
	
	We conclude Section \ref{Quantum groups uq(g,Lambda) and R-matrices} by working out the example $\g=\sl_2$ in the non-trivial case $\Lambda=\Lambda_W$: For $\ell$ odd we have a single $R$-matrix, which is not modularizable, and for $\ell$ even, we have two $R$-matrices, but only one of them is  modularizable. The result of modularization is a modular tensor category with a nontrivial associator.    
	
	In Section \ref{sec: def of u(omega,sigma)} we construct quasi-triangular quasi-Hopf algebras $u(\omega,\sigma)$ associated to abelian $3$-cocycles $(\omega,\sigma)$ on the dual of $G$, and these generalize small quantum groups. The following nice and systematic method seems to be the evident course of action, if one recalls the classification of pointed Hopf algebras \cite{AndruskiewitschSchneider2010} in terms of Nichols algebras \cite{Heckenberger2006}, although the technical details involving the quasi-Hopf structure become as usual somewhat tedious:
	\begin{enumerate}
    \item
	In Section \ref{sec: YDM} we discuss Yetter-Drinfeld modules. As it should be, the definition in \cite{Majid1998} specialized to the quasi-Hopf algebra $k^G_\omega$ is equivalent to modules over the twisted double $D^\omega(G)$ in \cite{DijkgraafPasquierRoche1992} and equivalent to the categorical definition of an object in the Drinfeld center of $\Vect_G^\omega$. For a given set of elements $(g_i)_{1 \leq i \leq n}$ in $G$ (later the simple roots from~$\Lambda$) and an abelian $3$-cocycle $(\omega,\sigma)$ we can construct such a Yetter-Drinfeld module $V$ with basis $(F_i)_{1 \leq i \leq n}$. Here we have effectively embedded the braided tensor category $\Vect_G^{(\omega,\sigma)}$ into the category of Yetter-Drinfeld modules.
	
	\item
	In Section \ref{sec:A Nichols algebra} we consider the Nichols algebra $\B(V)$ in the category $\Vect_G^\omega$. In complete analogy to the non-quasi case we derive certain relations in this Nichols algebra, which are later the analoga of the Serre-relations. 
	
	\item
	In Section \ref{sec: radford biproduct} we recall the quasi-version of the Radford biproduct from \cite{BulacuNauwelaerts2002} and give the algebra- and coalgebra structure of $\B(V)\#k^G_\omega$. In contrast to the non-quasi case the $F_i$ do not become $(g_i,1)$-primitive elements, but have a slightly more complicated coproduct.
	
	\item
	In Section \ref{sec: Drinfeld double} we construct the Drinfeld double of this quasi-Hopf algebra $\B(V)\#k^G_\omega$ following \cite{Schauenburg2002}, which is again a quasi-Hopf algebra. A main technical difficulty is that the dual of $\B(V)\#k^G_\omega$ is a (nonassociative) coquasi-Hopf algebra, so the relations have to be somewhat modified.
	
	\item
	In Section \ref{sec: a quotient} we take a quotient identifying the two copies of $k^G_\omega$. Again it is not surprising that this works, but it is a tedious calculation to see that some nontrivial identification (involving $\omega$) is indeed a well-defined quotient. In the following, we list important relations for the resulting quasi-Hopf algebra $u(\omega,\sigma)$. Here, the greek letters $\chi,\psi,\xi,..$ denote characters on $G$. Summation goes always over all characters in $\widehat{G}$. The elements $\chi_i \in \widehat{G}$ are part of the input data of $u(\omega,\sigma)$. Finally, the $2$-cocycle $\theta$ is defined in Remark \ref{rm: dual of YDM}.
		 \begin{align*}
		 \Delta(F_i)&= K_{\bar{\chi}_i} \otimes F_i \left( \sum_{\chi,\psi}\, \theta(\chi|\bar{\chi}_i,\psi)\omega(\bar{\chi}_i,\psi,\chi)^{-1} \,\delta_\chi \otimes \delta_\psi \ \right) + F_i \otimes 1\left( \sum_{\chi,\psi}\,\omega(\bar{\chi}_i,\chi,\psi)^{-1} \,\delta_\chi \otimes \delta_\psi \right)   \\
		 \Delta(E_i)&= \left( \sum_{\chi,\psi}\, \theta(\psi|\chi\bar{\chi}_i,\chi_i)^{-1}\omega(\psi,\chi,\bar{\chi}_i)^{-1} \,\delta_\chi \otimes \delta_\psi \ \right) E_i \otimes \bar{K}_{\chi_i}  + \left( \sum_{\chi,\psi}\, \omega(\chi,\psi,\bar{\chi}_i)^{-1} \,\delta_\chi \otimes \delta_\psi \right)1\otimes E_i  \\
		 \Delta(K_\chi)&=(K_\chi \otimes K_\chi)  P_{\chi}^{-1} \qquad \Delta(\bar{K}_\chi)=(\bar{K}_\chi \otimes \bar{K}_\chi) P_{\chi}, \qquad P_\chi:=\sum_{\psi,\xi}\,\theta(\chi|\psi,\xi) \,\delta_\psi \otimes \delta_\xi\\ 
		 &[E^a_iK_{\chi_i},F^b_j]_{\sigma} = \delta_{ij}\sigma(\chi_i,\bar{\chi}_i) 
		 \left(1-K_{\chi_i} \bar{K}_{\chi_i}
		 \right)\left( \sum_{\xi}\, \frac{a_i(\xi)b_i(\xi\chi_i)}{\omega(\bar{\chi}_i,\chi_i,\xi)}\, \delta_\xi \right),\text{ where}\\
		 &E^a_i:=E_i\left( \sum_{\xi}\, a_i(\xi)\,\delta_\xi\right)  \qquad F^b_j:=F_j\left( \sum_{\xi}\, b_j(\xi)\, \delta_\xi\right) , \quad\text{with } a_i,b_j \text{ solutions to Eq. \eqref{eq: equation for the drinfeld double}}.\\
		 &K_\chi E_i = \sigma(\chi,\chi_i) E_iK_\chi Q^{-1}_{\chi,\chi_i}, \qquad \bar{K}_\chi E_i = \sigma(\chi_i,\chi) E_i \bar{K}_\chi Q_{\chi,\chi_i}, \qquad Q_{\chi,\psi}:= \sum_{\xi} \, \theta(\chi|\xi,\psi)\,\delta_{\xi}\\
		 &K_\chi F_i = \sigma(\chi,\bar{\chi}_i) F_iK_\chi Q^{-1}_{\chi,\bar{\chi}_i}, \qquad \bar{K}_\chi F_i = \sigma(\bar{\chi}_i,\chi) F_i \bar{K}_\chi Q_{\chi,\bar{\chi}_i}\\
		 &S(F_i)= -\left( \sum_{\psi} \, \omega(\bar{\psi},\bar{\chi}_i,\chi_i\psi)d\sigma(\chi_i,\psi,\bar{\psi})\theta(\bar{\psi}|\psi\chi_i,\bar{\chi}_i)^{-1} \delta_\psi\right) K_{\chi_i}F_i\\
		 &S(E_i)=- E_i\bar{K}_{\chi_i}^{-1}\left( \sum_{\psi} \, \frac{\omega(\bar{\chi}_i\bar{\xi},\chi_i,\xi)}{\omega(\bar{\xi},\bar{\chi}_i,\chi_i)}\, \delta_{\xi}\right)\\
		 & \epsilon(K_\chi)=\epsilon(\bar{K}_\chi)=1, \qquad \epsilon(E_i)=\epsilon(F_i)=0, \qquad 1_{u(\omega,\sigma)}=K_1	
		 \end{align*}
		 
		 \item
	In Section \ref{sec:R-matrix}, using the above Drinfeld double construction, we obtain an $R$-matrix for $u(\omega,\sigma)$ so that it becomes a quasi-triangular quasi-Hopf algebra. 
\end{enumerate}

	In Section \ref{sec: Modularization} we return to the small quantum groups from Section \ref{Quantum groups uq(g,Lambda) and R-matrices} in order to modularize their representation category. We begin with the quantum group $u_q(\Lambda)$ with Cartan part $G$ and an $R$-matrix built from $R_0(g)$ for some bicharacter $g$. Recall that $R_0$ is given by the Killing form $q^{(\mu,\nu)}$ on the image of the root lattice $\Lambda_R/\Lambda'$ and is modified by $g$ on the larger $G=\Lambda/\Lambda'$.
	
	We note that $R_0$ is an $R$-matrix of the somewhat underlying semisimple category $\Vect_{\widehat{G}}^{(1,\sigma)}$, where $\sigma$ is the Fourier transform of $R_0$. Using the notation of Section \ref{sec: def of u(omega,sigma)}, we can identify $u_q(\Lambda)$ with $u(1,\sigma)$. For the choice we made, this algebra has non-trivial transparent modules, and we denote the group of transparent objects by $T$ (the $2$-torsion of $G$). In section \ref{sec: modularization of Vec}, we have derived conditions for $\Vect_{\widehat{G}}$ to be modularizable and if they hold we get a semisimple modular tensor category $\Vect_{\widehat{G}/T}^{(\bar{\omega},\bar{\sigma})}$ for a class of abelian $3$-cocycles $(\bar{\omega},\bar{\sigma})$, which is most conveniently parametrized by a quadratic form $\bar{Q}$ on $\widehat{G}/T$.
	
	Now for the datum $(\widehat{G}/T, (\bar{\omega},\bar{\sigma}))$ we consider the quasi-Hopf algebra $u(\bar{\omega},\bar{\sigma})$ constructed in Section~\ref{sec: def of u(omega,sigma)}. From the modularization functor $\Vect_{\widehat{G}}^{(1,\sigma)} \to \Vect_{\widehat{G}/T}^{(\bar{\omega},\bar{\sigma})}$ we can now step-by-step use the universal properties of our constructions to write down a ribbon functor 
	$$
	\Rep_{u_q(\Lambda)} \to \Rep_{u(\bar{\omega},\bar{\sigma})}\ .
	$$
	 Moreover, we prove this functor is a  modularization and so  the target category is a (non-semisimple) modular tensor category. 
	
%
\medskip	
	Finally, in Section~\ref{sec:CFT} we discuss our main application to VOAs in logarithmic CFT. We observe that if we demand the maximal divisibility $2,4,6|\ell$ (depending on the root lengths) and take the maximal choice $\Lambda=\Lambda_W$, then we indeed always get a quantum group with a uniform description of $\Lambda'$ and for the group of transparent objects $T=\Z_2^{\mathrm{rank}}$. Hence the modularization described above can be applied, and this provides a uniform family of \textsl{factorizable quasi-quantum groups} for any such a pair $(\g,\ell)$.
	We then conjecture that the corresponding family of modular tensor categories is braided and ribbon equivalent to modular tensor categories of the  VOA $\W(\g,\ell)$ constructed in the screening charge approach. We finally remark that on the CFT side we have the same divisibility constraints. 
	
	 Several appendices contain complementary material and technical lemmas.
%

\vskip-3mm
\mbox{}\\
{\bf Acknowledgements.}\;
The authors thank
J\"urgen Fuchs, Ehud Meir, 
Ingo Runkel, and
Christoph Schweigert for helpful and inspiring discussions. Moreover, we thank Alexey M. Semikhatov for his interest in the early stage of this work.

AMG is supported by CNRS and thanks the Humboldt Foundation for a partial financial support. AMG also thanks the Mathematics Department
of Hamburg University for kind hospitality. SL is grateful to RTG 1670 for partial support. TO is grateful to SFB 676 for financial support. 

\medskip	
	\section{Preliminaries}\label{sec: preliminaries}

\subsection{Quasi-Hopf algebras}\label{sec: Quasi-Hopf algebras}

We start with the definition of a quasi-Hopf algebra as introduced in \cite{Drinfeld1989}.

\begin{definition}
	A quasi-bialgebra over a field $k$ is a $k$-algebra $H$  with algebra homomorphisms $\Delta:H\to H\otimes H$ and $\epsilon:H\to k$, together with an invertible \emph{coassociator} $\phi\in H\otimes H\otimes H$, such that the following conditions hold:
	\begin{align*}
		(\id \otimes \Delta) (\Delta(a)) \cdot \phi &= \phi \cdot (\Delta \otimes \id)(\Delta(a)) \\ 
		(\id \otimes \id \otimes \Delta)(\phi) \cdot (\Delta \otimes \id \otimes \id)(\phi) &= (1\otimes \phi) \cdot (\id \otimes \Delta \otimes \id)(\phi) \cdot (\phi \otimes \id)\\
		(\epsilon \otimes \id)\Delta(h)&=h=(\id\otimes \epsilon)\Delta(h)\\
		(\id\otimes \epsilon \otimes \id)(\phi)&=1
	\end{align*}
	An antipode $(S,\alpha,\beta)$ for $H$ consists of algebra anti-automorphism $S:H \to H$, together with elements $\alpha,\beta \in H$, s.t.
	\begin{align*}
		S(h_{(1)})\alpha h_{(2)} = \epsilon(h) \alpha, \qquad h_{(1)}\beta S(h_{(2)})= \epsilon(h) \beta, \\
		X^1\beta S(X^2)\alpha X^3 =1, \qquad S(x^1) \alpha x^2 \beta x^3=1.
	\end{align*}
	Similar to Sweedlers' notation, we used the short-hand notation $\phi=X^1 \otimes X^2 \otimes X^3$. For the inverse, we use small letters $\phi^{-1}=x^1 \otimes x^2 \otimes x^3$. If more than one associator appears in an equation we use letters $X,Y,Z$ and $x,y,z$, respectively. 
\end{definition}
In the first condition, $\phi$ can be understood as a coassociativity constraint for the coproduct $\Delta$. In particular, $\phi=1 \otimes 1 \otimes 1$ implies coassociativity and $H$ becomes an ordinary Hopf algebra. 

\begin{theorem}[Drinfeld]
	The category of representations of a quasi-Hopf algebra is a rigid monoidal category. As in the Hopf case, the tensor product $V\otimes W$ is endowed with the diagonal action and the unit object is the trivial representation $k$. The associator $(U \otimes V) \otimes W \to U \otimes (V \otimes W)$ is given by:
	\begin{align}
		\omega_{U,V,W}: u \otimes v \otimes w \mapsto X^1.u\otimes X^2.v\otimes X^3.w.
 	\end{align}
 	For $V \in \mathsf{Rep}_H$ we define a dual object $V^\vee:=V^*$ with module structure given by $(h.f)(v):=f(S(h).v)$. Evaluation and coevaluation on $V$ are given by $ev_V(f\otimes v):=f(\alpha.v)$ and $coev_V(1)=\beta.e_i\otimes e^i$. Here, $e_i$ is a basis of $V$ and $e^j$ is the corresponding dual basis.
\end{theorem}
In fact, every finite tensor category (see Def. \ref{def: non-ssi mtc}) with quasi-fiber functor (i.e. weak monoidal functor to $\Vect_k$) is equivalent to the representation category of a quasi-Hopf algebra (see \cite{EtingofGelakiNikshychEtAl2015}, Thm. 5.13.7). A canonical source of equivalences between representation categories of quasi-Hopf algebras are twists:
\begin{defprop} \cite{HausserNill1999} \label{def: twist}
	Let $H$ be a quasi-bialgebra, and $J\in H \otimes H$ an invertible element, s.t. $(\epsilon \otimes \id)(J) = (\id \otimes \epsilon)(J)=1$. Such an element $J$ is called a twist. Given a twist $J$, we can define a new quasi-bialgebra $H^J$ which is $H$ as an algebra, with the same counit, the coproduct is given by
	\begin{align*}
	\Delta^J(x)=J\Delta(x)J^{-1},
	\end{align*}
	and the associator given by
	\begin{align*}
	\phi^J=U_R \phi U_L^{-1},
	\end{align*}
	where $U_L:=(J \otimes 1)(\Delta \otimes \id)(J)$ and $U_R:=(1 \otimes J)(\id \otimes \Delta)(J)$.
	Moreover, we have
	\begin{align*}
		\alpha^J=S(J^{(-1)})\alpha J^{(-2)} \qquad \beta^J=J^1 \beta S(J^2).
	\end{align*}
	The quasi-bialgebra $H^J$ is called twist equivalent to $H$ by the twist $J=J^1 \otimes J^2$. The inverse of $J$ is denoted by $J^{-1}=J^{(-1)}\otimes J^{(-2)}$. If a twist appears more then once in an equation, we use letters J,K,L,\dots for them.
\end{defprop}

\begin{example} \label{twisted dual group algebra}
	Let $G$ be a finite group, $\omega \in Z^3(G,k^\times)$ be a group 3-cocycle. The dual algebra $k^{G}$ equipped with the coassociator
	\begin{align*}
		\phi= \sum_{g_1,g_2,g_3 \in G} \, \omega(g_1,g_2,g_3) \, \delta_{g_1} \otimes \delta_{g_2} \otimes \delta_{g_3} \qquad (\text{where } \delta_g(h):=\delta_{g,h})
	\end{align*}
	is a quasi-Hopf algebra which we denote by $k^{G}_{\omega}$.\\	
	The category of  representations of $k^{G}_{\omega}$ is identified with the category $\Vect_G^\omega$ of $G$-graded vector spaces with associator $\omega_{g_1,g_2,g_3}:(k_{g_1}\otimes k_{g_2})\otimes k_{g_3} \to k_{g_1}\otimes(k_{g_2}\otimes k_{g_3})$ for simple objects $k_{g_i}$ given by
	\begin{align*}
		\omega_{g_1,g_2,g_3}:1_{g_1} \otimes 1_{g_2} \otimes 1_{g_3} \mapsto \omega(g_1,g_2,g_3) \cdot 1_{g_1} \otimes 1_{g_2} \otimes 1_{g_3}.
	\end{align*}
\end{example}

The following example is due to \cite{DijkgraafPasquierRoche1992}\footnote{However, we use the convention of \cite{Majid1998} for later use.}:
\begin{example}\label{ex: PDR group double}
	Again, let $G$ be a finite group, $\omega \in Z^3(G,k^\times)$ be a group 3-cocycle. Then there is a quasi-Hopf algebra $D^\omega(G)$ with a basis $g\otimes \delta_h$. The coassociator of this quasi-Hopf algebra is given by
	\begin{align*}
		\phi= \sum_{g_1,g_2,g_3 \in G} \, \omega(g_1,g_2,g_3) \, (e\otimes\delta_{g_1}) \otimes (e\otimes\delta_{g_2}) \otimes (e\otimes\delta_{g_3}).
	\end{align*}	
	Product and coproduct are given by
	\begin{align*} 
		(h_1 \otimes \delta_{g_1})\cdot (h_2 \otimes \delta_{g_2})& = \delta_{h_1g_1h_1^{-1},g_2} \frac{\omega(h_2,h_1,h_1^{-1}h_2^{-1}g_2h_1h_2)\omega(g_2,h_2,h_1)}{\omega(h_2,h_2^{-1}g_2h_2,h_1)} (h_2h_1\otimes \delta_{g_2})\\
		\Delta(h \otimes \delta_{g}) &= \sum_{g_1g_2=g}\, \frac{\omega(h,h^{-1}g_1h,h^{-1}g_2h)\omega(g_1,g_2,h)}{\omega(g_1,h,h^{-1}g_2h)} (h \otimes \delta_{g_1}) \otimes (h \otimes \delta_{g_2})
	\end{align*}
	Moreover, there is a compatible $R$-matrix 
	\begin{align*}
		R=\sum_{g \in G}\, (e\otimes \delta_g) \otimes (g \otimes 1_{k^G}),
	\end{align*}
	such that the category of representations is a braided monoidal category.
\end{example}
Modules over $D^\omega(G)$ can be identified with Yetter-Drinfeld modules over $k^G_{\omega}$ with a projective coaction instead of an ordinary coaction (see \cite[Prop. 2.2]{Majid1998} or \cite[Def. 6.1]{Vocke2010}). More generally the following universal construction has been estabilshed in \cite{HausserNill1999a}:
\begin{example}
	Let $H$ be a quasi-Hopf algebra, then there is a Drinfeld double $DH$ which is again a quasi-Hopf algebra with $R$-matrix. For example in our first case $D k^{G}_{\omega} \cong D^\omega(G)$.
	
	It has the universal property that the braided category of $DH$-representation is equivalent to the Drinfeld center of the monoidal category $\mathsf{Rep}_H$ of $H$-representations. Similarly, one can introduce Yetter-Drinfeld modules over a quasi-Hopf algebra, and this braided monoidal category is equivalent to the previous categories. 
\end{example}
For later use, we introduce several important elements which were first defined in \cite{Drinfeld1989}:
For an arbitrary quasi-Hopf algebra there is an invertible element $f=f^1 \otimes f^2 \in H \otimes H$, which we refer to as Drinfeld twist, satisfying
\begin{align}\label{eq: Drinfeld twist}
	f\Delta(S(h))f^{-1}=(S \otimes S)(\Delta^{cop}(h)).
\end{align}
Before we give $f$ and $f^{-1}$ explicitly, we follow \cite{BulacuNauwelaerts2002} by defining elements
\begin{align*}\label{def: p's}
	p_R=p^1 \otimes p^2 &=x^1 \otimes x^2 \beta S(x^3) &\qquad p_L=\tilde{p}^1\otimes \tilde{p}^2:&=X^2S^{-1}(X^1\beta) \otimes X^3 \\
	q_R=q^1 \otimes q^2 &=X^1 \otimes S^{-1}(\alpha X^3)X^2 &\qquad q_L=\tilde{q}^1\otimes \tilde{q}^2:&=S(x^1)\alpha x^2 \otimes x^3,
\end{align*}
satisfying the following useful equalities:
\begin{align}
	\Delta(h_{(1)})p_R(1\otimes S(h_{(2)})) & = p_R(h \otimes 1) &\qquad \Delta(h_{(2)})p_L(S^{-1}(h_{(1)})\otimes 1)&= p_L(1 \otimes h)\\\label{eq:p,q identities}
	(1\otimes S^{-1}(h_{(2)}))q_R\Delta(h_{(1)})&=(h \otimes 1)q_R& \qquad (S(h_{(1)})\otimes 1)q_L\Delta(h_{(2)}) &= (1 \otimes h)q_L.
\end{align}
Furthermore, we define elements
\begin{align}
	\delta = \delta^1 \otimes \delta^2 &= x^1 \beta S(x^3_{(2)}\tilde{p}^2) \otimes x^2S(x^3_{(1)}\tilde{p}^1) \\
	\gamma = \gamma^1 \otimes \gamma^2 &= S(q^2x^1_{(2)})x^2 \otimes S(q^1 x^1_{(1)})\alpha x^3
\end{align}
The Drinfeld twist $f$ and its inverse $f^{-1}$ are then explicitly given by:
\begin{align} \label{def: f's}
	\begin{split}
		f&= (S \otimes S)(\Delta^{cop}(p^1))\gamma \Delta(p^2)\\
		f^{-1}&=\Delta(\tilde{q}^1)\delta(S\otimes S)(\Delta^{cop}(\tilde{q}^2)).
	\end{split}
\end{align}
In addition to Eq. \ref{eq: Drinfeld twist}, the Drinfeld twist satisfies
\begin{align*}
	f\Delta(\alpha)=\gamma \qquad \Delta(\beta)f^{-1}=\delta.
\end{align*}
Let $J\in H \otimes H$ be a twist. The respective elements on the twisted quasi-Hopf algebra $H^J$ are denoted by $p_J,q_J,\delta_J, \dots$. Using the above identities, one can show: 
\begin{align*}\label{eq: identities for twisted drinfeld twists, etc}
	{p_R}_J&=U_L^1p^1 \otimes U_L^2p_R^2S(U_L^3) &  {p_L}_J&=U_R^2\tilde{p}^1S^{-1}(U_R^1)\otimes U_R^3 \tilde{p}^2 \\
	{q_R}_J&=q^1U_L^{(-1)} \otimes S^{-1}(U_L^{(-3)})q^2U_L^{(-2)} & {q_L}_J&=S(U_R^{-1})\tilde{q}^1U_R^{(-2)} \otimes \tilde{q}^2U_R^{(-3)} \\
	f_J &=S(J^{(-2)})f^1K^{(-1)} \otimes S(J^{(-1)})f^2K^{(-2)} & f^{-1}_J &=J^1f^{(-1)}S(K^2) \otimes J^2 f^{(-2)} S(K^1).
\end{align*}

Finally, we introduce the notion of a quasitriangular ribbon quasi-Hopf algebra. We start with the notion of quasi-triangularity:

\begin{definition}
	A quasi-triangular quasi-Hopf algebra is a quasi-Hopf algebra $H$ together with an invertible element $R \in H\otimes H$, the so-called $R$-matrix, s.t. the following conditions are fulfilled:
	\begin{align*}
		R\Delta(h)&=\Delta^{op}(h)R \qquad \forall\, h \in H \\
		(\Delta \otimes \id)&=\phi_{321}R_{13}\phi_{132}^{-1}R_{23}\phi \\
		(\id \otimes \Delta)&=\phi_{231}^{-1}R_{13}\phi_{213}R_{12}\phi^{-1}
	\end{align*}	
\end{definition}

The definition of a quasi-triangular quasi-Hopf algebra has an important symmetry: If $R=R^1\otimes R^2 \in H\otimes H$ is an $R$-matrix for $H$, then so is $R_{21}^{-1} \in H\otimes H$. \\
The following lemma is proven in \cite{BulacuNauwelaerts2003}:

\begin{lemma}\label{lm: Drinfeld element}
	Let $(H,R)$ be a quasi-triangular quasi-Hopf algebra. We define the Drinfeld element $u \in H$ as follows:
	\begin{align*}
		u:=S(R^2p^2)\alpha R^1p^1.
	\end{align*}
	This element is invertible and satisfies $S^2(h)=uhu^{-1}$.
\end{lemma}

%
%

We now recall the definition of a ribbon quasi-Hopf algebra. It has been shown in \cite{BulacuNauwelaerts2003} that the following definition of a quasi-triangular ribbon quasi-Hopf algebra is equivalent to the original one given in \cite{AltschuelerCoste1992}.

\begin{definition}\label{dfn: ribbon quasi-hopf algebra}
	A quasi-triangular quasi-Hopf algebra $(H,R)$ is called ribbon if there exists a central element $\nu \in H$, s.t.
	\begin{align*}
		\nu^2 &= uS(u),& 
				S(\nu)&=\nu, \\
		\epsilon(\nu)&=1,& 
		\Delta(\nu)&=(\nu \otimes \nu)(R_{21}R)^{-1}.
	\end{align*}  
\end{definition}

%
%

\subsection{Abelian 3-cocycles on $G$}

From now on, let $G$ be a finite abelian group.

\begin{definition}\label{abelianCohomology}\cite{MacLane1952}\cite{JoyalStreet1993}\footnote{Note that we follow \cite{JoyalStreet1993} and thus have a different convention than \cite{MacLane1952}, this amounts to having $\omega^{-1}$ everywhere.}
	An abelian 3-cocycle $(\omega,\sigma)\in Z^3_{ab}(G)$ is a pair consisting of a ordinary 3-cocycle $\omega \in Z^3(G,\C^\times)$ and an ordinary 2-cochain $\sigma \in C^2(G,\C^\times)$, s.t. the following two equations hold:
	\begin{align}
		\frac{\omega(y,x,z)}{\omega(x,y,z)\omega(y,z,x)} = 	\frac{\sigma(x,y+z)}{\sigma(x,y)\sigma(x,z)} \\
		\frac{\omega(z,x,y)\omega(x,y,z)}{\omega(x,z,y)} = 	\frac{\sigma(x+y,z)}{\sigma(x,z)\sigma(y,z)}
	\end{align}	
	An abelian $3$-coboundary is of the form $d_{ab}\kappa:=(\kappa/\kappa^T,d\kappa^{-1})$ for any ordinary 2-cochain $\kappa \in C^2_{ab}(G):=C^2(G,\C^\times)$. Here, $\kappa^T(x,y):=\kappa(y,x)$.
	The quotient group of abelian cohomology classes is the abelian cohomology group $H^3_{ab}(G)$.
\end{definition}
\begin{theorem}\cite{MacLane1952}\label{3rdAbCohomology and QuadraticForms}
	To any abelian 3-cocycle $(\omega,\sigma)$ there is an associated quadratic form $Q(g):=\sigma(g,g)$ on the group $G$ and the associated symmetric bihomomorphism $\Beta(g,h):=\sigma(g,h)\sigma(h,g)$. We have an identity 
	\begin{align}\label{QuadraticForm and SymmBihom}
		\Beta(g,h)=\frac{Q(g+h)}{Q(g)Q(h)}.
	\end{align}
	This implies, that the symmetric bihomomorphism $\Beta$ characterizes the quadratic form up to a homomorphism $\eta \in \Hom(G, \{\pm 1\})$.\\
	The assignment $\Phi:(\omega,\sigma)\mapsto Q$ yields a group isomorphism  between abelian 3-cohomology classes $H^3_{ab}(G)$ and quadratic forms $QF(G)$ on $G$. 
\end{theorem}

As we shall see in the next section, abelian cohomology classes classify different braiding/tensor structures on the category of $G$-graded vector spaces. 

\begin{example}
	Let $G=\Z_n$ be a cyclic group. By $\zeta_k$ we denote a primitive $k$th root of unity. We have two cases:
	\begin{itemize}
		\item For odd $n$ we have $H^3_{ab}(\Z_n)=\Z_n$ with representatives $(\omega,\sigma)$ for $k=0,\ldots n-1$ given by $\sigma(g^i,g^j)=\zeta_n^{kij},\omega=1$ and the respective quadratic form is given by
		$$Q(g^i)=\zeta_n^{ki^2},\qquad \Beta(g^i,g^j)=\zeta_n^{2kij}$$
		\item For even $n$ we have $H^3_{ab}(\Z_n)=\Z_{2n}$ with
		 quadratic forms for $k=0,\ldots 2n-1$
		 $$Q(g^i)=\zeta_{2n}^{ki^2},\qquad \Beta(g^i,g^j)=\zeta_{2n}^{2kij}$$
		 For even $k$ we have again representatives given by $\sigma(g^i,g^j)=\zeta_n^{(k/2)ij},\omega=1$, but for odd $k$ we have only representatives with $\omega$ in the nontrivial cohomology class of $H^3(\Z_n,\Z_2)$.
	\end{itemize} 
	In particular $G=\Z_4$ has four abelian cohomology classes, two of which have trivial $\omega$, and two of which have nontrivial $\omega$ and nondegenerate $\Beta$.
 \end{example}
 
 \begin{proposition} \label{prop: explicit abelian 3-cocycle}
 	Let $G = \bigoplus_{i=1}^n \, \Z_{m_i}$ be a finite abelian group with generators $g_i$, $i =1, \dots n$. 
 	\begin{enumerate}
 		\item A quadratic form $Q \in QF(G)$ is uniquely determined by elements $0 \leq r_i \leq \gcd(2,m_i)m_i-1$ and $0 \leq r_{kl} \leq \gcd(m_k,m_l)-1$ (for $k < l$), so that
 		\begin{align*}
 			Q(g_i)= \exp \left( \frac{2 \pi i \cdot r_i}{\gcd(2,m_i)m_i} \right) \qquad \Beta(g_k,g_l)=\exp \left( \frac{2 \pi i \cdot r_{kl}}{\gcd(m_k,m_l)} \right) \quad (k<l).
 		\end{align*}
 		\item For a quadratic form $Q \in QF(G)$, the abelian $3$-cocycle $(\omega,\sigma) \in Z^3_{ab}(G)$ given by
 		\begin{align*}
 			\sigma(a,b) :&= \prod_{i=1}^n\, Q(g_i)^{a_ib_i} \, \prod_{1 \leq k < l\leq n} \,\Beta(g_k,g_l)^{a_kb_l} \\
 			\omega(a,b,c):&= \prod_{i=1}^n\, Q(g_i)^{m_i \delta_{a_i +b_i \geq m_i} c_i}
 		\end{align*}
 		satisfies $Q(g)=\sigma(g,g)$.
 	\end{enumerate}
 \end{proposition}
 
 \begin{proof}
 	From Theorem \ref{3rdAbCohomology and QuadraticForms} it is easy to see that for an element $g=g_1^{x_1}\dots g_n^{x_n} \in G$, $Q(g)$ decomposes as follows:
 	\begin{align*}
 		Q(g)=\prod_{i=1}^n\, Q(g_i)^{x_i^2} \cdot \prod_{i<j}\, \Beta(g_k,g_l)^{x_k x_l}.
 	\end{align*} 
 	As $\Beta$ is a bihomomorphism on the finite abelian group $G$, we have 
 	\begin{align*}
 	\Beta(g_k,g_l)=\exp \left( \frac{2 \pi i \cdot r_{kl}}{\gcd(m_k,m_l)} \right)
 	\end{align*} for some $0 \leq r_{kl} \leq \gcd(m_k,m_l)-1$. From this formula also follows 
 	\begin{align*}
 		Q(g_i)=\exp \left( \frac{2 \pi i \cdot \tilde{r}_i}{m_i^2} \right)
 	\end{align*} for some $0\leq \tilde{r}_i \leq m_i^2-1$. Combining this with the axiom $Q(g_i)=Q(-g_i)$ leads to 
 	\begin{align*}
 		Q(g_i)= \exp \left( \frac{2 \pi i \cdot r_i}{\gcd(2,m_i)m_i} \right)
 	\end{align*}
 	for some $0 \leq r_i \leq \gcd(2,m_i)m_i-1$. It is a straightforward computation to check that the pair $(\omega,\sigma)$ defined in the second part of the proposition satisfies the axioms of an abelian $3$-cocycle.
 	Finally, using the above decomposition of $Q(g)$, it is easy to see that $\sigma(g,g)=Q(g)$.
 \end{proof}

\subsection{Modular structures on $\Vect_G$}

Monoidal structures on $\Vect_G$ were first related to the third (abelian) cohomology of $G$ by Ho\`ang Xu\^an S\`inh, a student of Grothendieck. In her unpublished thesis, the corresponding (braided) monoidal categories are referred to as "Gr-categories".

\begin{theorem}\label{braided monoidal str on VectG}
	Let $(\omega,\sigma) \in Z^3_{ab}(G)$ be an abelian $3$-cocycle on the finite abelian group $G$. This induces a canonical braided monoidal structure on the category $\Vect_G$ of $G$-graded $k$-vector spaces. On simple objects $k_{g_i}$, associator, unitors and braiding are given by:
	\begin{align*}
		\omega_{g_1,g_2,g_3}:(k_{g_1} \otimes k_{g_2}) \otimes k_{g_3} &\to k_{g_1} \otimes (k_{g_2} \otimes k_{g_3}), & \quad 
		1_{g_1} \otimes 1_{g_2} \otimes 1_{g_3} &\mapsto \omega(g_1,g_2,g_3) \cdot 1_{g_1} \otimes 1_{g_2} \otimes 1_{g_3} \\
		l_g:k_0 \otimes k_g &\to k_g, & \quad 1_0 \otimes 1_g  & \mapsto \omega(0,0,g)^{-1} \cdot 1_g
		\\
		r_g:k_g \otimes k_0 &\to k_g, &  1_g \otimes 1_0  & \mapsto \omega(g,0,0) \cdot  1_g \\
		\sigma_{g_1,g_2}:k_{g_1} \otimes k_{g_2} &\to k_{g_2} \otimes k_{g_1}, &\quad 
		1_{g_1} \otimes 1_{g_2}  &\mapsto \sigma(g_1,g_2) \cdot 1_{g_1} \otimes 1_{g_2}.
	\end{align*}
	 The resulting braided monoidal category is denoted by $\Vect_G^{(\omega,\sigma)}$. All braided monoidal structures on $\Vect_G$ are classified up to braided monoidal equivalence by the third abelian cohomology group $H^3_{ab}(G)$ modulo automorphisms on $G$.	
\end{theorem}

\begin{proof}
	This is Exercise 8.4.8 in \cite{EtingofGelakiNikshychEtAl2015}.
\end{proof}

\begin{remark}
	Since every $3$-cocycle is equivalent to a normalized one, we can choose the unitors to be trivial.
\end{remark}

We recall the definition of a pre-modular category and modularization in the semisimple case:

\begin{definition}\label{dfn:premodular cat (ssi)}
	A fusion category is a rigid semisimple $k$-linear monoidal category $\mathcal{C}$ with only finitely many isomorphism classes of simple objects, such that $\text{End}(\mathbb{I}) \cong k$. A braided fusion category is called pre-modular if it has a ribbon structure, i.e. an element $\theta \in \text{Aut}(\id_\mathcal{C})$, s.t.
			\begin{align}
				\theta_{X\otimes Y}&=(\theta_X \otimes \theta_Y) \circ c_{Y,X} \circ c_{X,Y} \label{eq:ribbon eq 1}\\
				(\theta_X)^*&=\theta_{X^*} \label{eq:ribbon eq 2},
			\end{align}
	where $c$ denotes the braiding in $\mathcal{C}$.
\end{definition}

\begin{definition}
	In a pre-modular category, we have categorical traces and categorical dimensions:
	\begin{align*}
	tr_{\mathcal{C}}(f):&=d_X\circ c_{X,X^*} \circ ((\theta_X \circ f) \otimes \id_{X^*}) \circ b_X: \mathbb{I} \to \mathbb{I} \qquad f \in \text{End}(X) \\
	dim_{\mathcal{C}}(X):&=tr(\id_X) \qquad X \in \mathcal{C},
	\end{align*}
	where $d$ and $b$ denote evaluation and coevaluation in the rigid category $\mathcal{C}$. 
\end{definition}

In the following, we recall the notions of a modular tensor category and of a modularization. The latter was first introduced in \cite{Bruguieres2000}.

\begin{definition}
	Let $\mathcal{C}$ be a pre-modular category. The so-called $S$-matrix of $\mathcal{C}$, $\mathcal{S}=(\mathcal{S}_{XY})_{X,Y \in \mathcal{O}(\mathcal{C})}$, is indexed by the set $\mathcal{O}(\mathcal{C})$ of isomorphism classes of simple objects in $\mathcal{C}$ with entries defined by
	\begin{align*}
	\mathcal{S}_{XY}:=tr_{\mathcal{C}}(c_{Y,X}c_{X,Y}).
	\end{align*}
	A pre-modular category is said to be modular if its $S$-matrix is non-degenerate. \\
	A linear ribbon functor $F:\mathcal{C} \to \mathcal{D}$ between pre-modular categories is said to be a modularization if 
	\begin{enumerate}
		\item it is dominant, i.e. for every object $D \in \mathcal{D}$ we have $\text{id}_D=p\circ i$ for some $i:D\to F(C)$, $p:F(C)\to D$, $C \in \mathcal{C}$. Equivalently, $D$ is a summand of an image under $F$.
		\item $\mathcal{D}$ is modular. 
	\end{enumerate}	
	If such a functor exists, then $\mathcal{C}$ is called modularizable.
\end{definition}

\begin{remark}
	In \cite{Bruguieres2000} it is shown that if a modularization exist it must be unique.
\end{remark}

The following lemma summarizes Paragraph 2.11.6 in \cite{DrinfeldGelakiNikshychEtAl2010}.

\begin{lemma}\label{RibbonStr on VectG}
	For an abelian $3$-cocycle $(\omega,\sigma) \in Z^3_{ab}(G)$, let $\Vect_G^{(\omega,\sigma)}$ be the corresponding braided monoidal category from Thm. \ref{braided monoidal str on VectG}.
	For a character $\eta:G \to \{\pm 1\}$, the twist $\theta_g:k_g \to k_g$ given by multiplication with $Q(g) \cdot \eta(g)$ defines a ribbon structure on $\Vect_G^{(\omega,\sigma)}$. We denote the resulting ribbon catgegory by $\Vect_G^{(\omega,\sigma,\eta)}$. All ribbon structures on $\Vect_G$ up to ribbon category equivalence are classified by elements $(Q,\eta) \in QF(G)\oplus \Hom(G,\{\pm 1\})$ modulo autmorphisms on $G$.
\end{lemma}
\begin{proof}
	We have already seen that braided monoidal structures on $\Vect_G$ are classified by $H^3_{ab}(G)$ modulo automorphisms on $G$. By Theorem \ref{3rdAbCohomology and QuadraticForms} for every such class there is a unique quadratic form $Q$, which satisfies equation \ref{QuadraticForm and SymmBihom}. This means that the quadratic form defines a twist $\theta$, satisfying equation \ref{eq:ribbon eq 1}. As a quadratic form, $Q$ satisfies $Q(-x)=Q(x)$, which implies equation \ref{eq:ribbon eq 2}. Thus, $\theta$ is a ribbon structure. On the other hand, the theorem says that two quadratic forms have the same associated symmetric bihomomorphism $\frac{Q(x+y)}{Q(x)Q(y)}$ if and only if they differ by a homomorphism $\eta: G\to \{\pm 1\}$, which implies that $Q\eta$ also defines a ribbon structure. From Exercise 8.4.6 in \cite{EtingofGelakiNikshychEtAl2015} we know that a braided monoidal equivalence between  $\Vect_G^{(\omega_1,\sigma_1)}$ and $\Vect_G^{(\omega_2,\sigma_2)}$ is uniquely determined by an automorphism $f:G \to G$ s.t. $Q_1=Q_2 \circ f$ together with some choice of $\kappa: G \times G \to \C$, s.t. $d_{ab}\kappa = (\omega_1,\sigma_1)^{-1}f^*(\omega_2,\sigma_2)$. Given a ribbon structure $\eta_1 \in \Hom(G,\{\pm 1\})$ on $\Vect_G^{(\omega_1,\sigma_1)}$ and $\eta_2=\eta_1 \circ f^{-1}$ on $\Vect_G^{(\omega_2,\sigma_2)}$ it is easy to see that this functor is a ribbon equivalence.
\end{proof}

\begin{remark}\label{when is VecG modular}
	By definition of the associated symmetric bilinear form $\Beta$ of $(\omega, \sigma)\in Z^3_{ab}(G)$, the pre-modular category $\Vect_G^{(\omega,\sigma,\eta)}$ is modular if and only if $\Beta$ is non-degenerate.
\end{remark}
	\section{Modularization of $\mathsf{Vect}_G^{(\omega,\sigma)}$}\label{sec: modularization of Vec}

\begin{proposition}\label{Modularizability of VectG}
	For a given ribbon structure $(\omega, \sigma,\eta)$ on $\Vect_G$, the pre-modular category $\Vect_G^{(\omega,\sigma,\eta)}$ is modularizable if and only if both $Q$ and $\eta$ are trivial on the radical $T:=\text{Rad}(\Beta)$ of the associated symmetric bilinear form $\Beta:G \times G \to \mathbb{C}^\times$ of $(\omega, \sigma) \in Z^3_{ab}(G)$. Explicitly, we construct a functor
	\begin{align*}
		F: \Vect_G^{(\omega,\sigma,\eta)} \to \Vect_{G/T}^{(\bar{\omega},\bar{\sigma},\bar{\eta})}
	\end{align*}
	to a modular category $\Vect_{G/T}^{(\bar{\omega},\bar{\sigma},\bar{\eta})}$, where the triple $(\omega,\sigma,\eta)$ defined on $G$ factors through a triple $(\bar{\omega},\bar{\sigma},\bar{\eta})$ defined on $G/T$.
\end{proposition}

\begin{proof}
	By Paragraph 2.11.6 in \cite{DrinfeldGelakiNikshychEtAl2010} we have $dim(\C_g)=\eta(g)$. Then, the first part of the proposition is an easy application of Theorem 3.1 in \cite{Bruguieres2000}. We now construct the modularization functor explicitly:\\
	Let $\Vect_G^{(\omega, \sigma,\eta)}$ be a pre-modular category satisfying the conditions in Proposition \ref{Modularizability of VectG}.\\
	\\
	We first want to find the modular target category of the modularization functor:\\
	The condition $Q|_T=1$ implies that $Q$ factors to a well-defined quadratic form $\bar{Q}: G/T \to \C^\times$. Let $(\bar{\omega},\bar{\sigma}) \in Z^3_{ab}(G/T)$ denote a representative of $\Phi^{-1}(\bar{Q})$, where the isomorphism $\Phi:H^3_{ab}(G/T) \to QF(G/T)$ was introduced in theorem \ref{3rdAbCohomology and QuadraticForms}. Furthermore, since $\eta:G \to \C^{\times}$ was a character, $\eta|_T=1$ implies that it factors through a character $\bar{\eta}:G/T \to \C^\times$. Hence, we obtain a pre-modular category $\Vect_{G/T}^{(\bar{\omega},\bar{\sigma},\bar{\eta})}$. Since $T$ was defined as the radical of $\Beta$, the new associated symmetric form $\bar{\Beta}:G/T \times G/T \to \C^\times$ is non-degenerate and by remark \ref{when is VecG modular}, $\Vect_{G/T}^{(\bar{\omega},\bar{\sigma},\bar{\eta})}$ is even modular.\\
	\\
	Now, we need to construct a linear ribbon functor $F:\Vect_G^{(\omega, \sigma,\eta)} \to \Vect_{G/T}^{(\bar{\omega},\bar{\sigma},\bar{\eta})}$:\\
	Clearly, the projection $\pi:G \to G/T$ induces a functor $\Vect_G \to \Vect_{G/T}$. We want to endow this functor with a monoidal structure that is compatible with braiding and twist. This amounts to finding a 2-cochain $\kappa:G \times G \to \C$, s.t.
	\begin{align*}
		 \pi^*(\bar{\omega},\bar{\sigma}):=(\bar{\omega}\circ \pi,\bar{\sigma}\circ \pi)=d_{ab}\kappa \cdot (\omega,\sigma).
	\end{align*}
	The second argument in this equation guarantees that the natural transformation induced by $\kappa$ satisfies the axiom of a monoidal structure for $F$, whereas the second argument guarantees that this monoidal structure preserves the braiding.
	Since the associated quadratic form of the abelian $3$-cocycle $(\tilde{\omega},\tilde{\sigma})=(\omega,\sigma)^{-1}\pi^*(\bar{\omega},\bar{\sigma})$ vanishes by assumption and since $\Phi$ is an isomorphism, $(\tilde{\omega},\tilde{\sigma})$ must be an abelian coboundary, hence $\kappa$ exists. Compatibility with the twist is due to $\bar{\eta}\circ \pi = \eta$ and $\bar{Q}\circ \pi = Q$. This functor is clearly dominant, since it sends simple objects to simple objects and all simple objects are in the image.
\end{proof}

\begin{remark}\label{CompareWithBrugieres}
	In \cite{Bruguieres2000}, the modularised category is constructed as the category of modules of a commutative algebra $\mathcal{T}$ inside the non-modular category $\mathcal{C}$. The modularization is then simply the induction functor. Recall that an object $X \in \mathcal{C}$ is called transparent if $c_{Y,X}c_{X,Y}=\id_{X\otimes Y}$ for all $Y \in \mathcal{C}$. As an object, $\mathcal{T}$ is the direct sum of all simple transparent objects. 
	We remark that in our case $\mathcal{C}=\Vect_G^{(\sigma, \omega,\eta)}$ our explicit modularization functor and our modularised category is equivalent to Brugieres' construction for $\mathcal{T}:=\oplus_{t \in T}\, \C_t$.
\end{remark}

	\section{Quantum groups $u_q(\g,\Lambda)$ and $R$-matrices}\label{Quantum groups uq(g,Lambda) and R-matrices}

In this section, we recall the some results from \cite{LentnerOhrmann2017}. To begin with, we collect some data in order to define the small quantum group $u_q(\g,\Lambda)$:\\
In the following, let
\begin{itemize}
	\item $\g$ be a simple complex finite-dimensional Lie algebra with simple roots $\alpha_1,\dots, \alpha_n$ and Killing form $(\alpha_i,\alpha_j)$,
	\item $q$ be a primitive $\ell$th root of unity, where $\ell \in \N$,
	\item $\Lambda_R \subseteq \Lambda \subseteq \Lambda_W$ be an intermediate lattice between the root lattice $\Lambda_R$ and weight lattice $\Lambda_W$ of $\g$, equivalently a subgroup of the fundamental group $H=\Lambda/\Lambda_R \subseteq \pi_1 := \Lambda_W/\Lambda_R$,
	\item $\Lambda'$ be the centralizer of the root lattice with respect to $\Lambda$,
	\begin{align}\label{eq: centralizer formula}
		\Lambda'=\Cent_{\Lambda_R}(\Lambda):=\{ \alpha \in \Lambda_R \, \mid \, q^{(\alpha,\nu)}=1 \quad \forall \nu \in \Lambda  \}
	\end{align}
	and the quotient group $G:=\Lambda/\Lambda'$,
	\item $\Lambda_R \subseteq \Lambda_1,\Lambda_2 \subseteq \Lambda$ sublattices, equivalently subgroups $G_i:=\Lambda_i/\Lambda' \subseteq G$ of common index $d:=|G_i|$, s.t. $G_1+G_2 =G$. Note that in all cases except $\g = D_{2n}$, we have a cyclic fundamental group and thus $\Lambda_1=\Lambda_2=\Lambda$ and $G_1=G_2=G$,
	\item $f:G_1 \times G_2 \to \C^\times$ be a bilinear form,
	\item $H_i:=\Lambda_i/\Lambda_R \subseteq H $.
\end{itemize}

\begin{example}\label{exp:comp of Lambda'}
	We compute $\Lambda'$ for the two extremal cases $\Lambda_W$ and $\Lambda_R$:
	In the first case $\Lambda=\Lambda_W$, we have
	\begin{align*}
		\Lambda_W'=\left\langle  \, \frac{\ell}{gcd(\ell,d_i)} \alpha_i \,|\,1\leq i\leq n \, \right\rangle _{\Z}.
	\end{align*}
	Here, $d_i:= \frac{(\alpha_i,\alpha_i)}{2}$ denotes the length of the simple root $\alpha_i$. In the second case $\Lambda=\Lambda_R$, we have
	\begin{align*}
		\Lambda_R'=\Lambda_R \cap \ell \cdot \Lambda_W^\vee,
	\end{align*} 
	where $\Lambda_W^\vee$ denotes the co-weight lattice, i.e. the $\Z$-span of the fundamental dominant co-weights $\lambda_i^\vee:=\lambda_i/d_i$.
\end{example}

\begin{theorem}\cite{LentnerOhrmann2017}\label{thm: thm from LO}
	For the above data, let $u_q(\g,\Lambda)$ be the so-called extended \textit{small quantum group} with coradical $u_q^0(\g,\Lambda) \cong \C[G]$, as defined for example in \cite{Lentner2016}, Def. 5.3. \\
	Then, the following element defines an $R$-matrix for this quantum group if and only if the bilinear form $f$ is non-degenerate:
	\begin{align*}
		R=R_0(f)\bar{\Theta},
	\end{align*}
	where $\Theta \in u_q^-(\g,\Lambda) \otimes u_q^+(\g,\Lambda)$ is the universal quasi-$R$-matrix constructed by Lusztig (see \cite{Lusztig1993}, Thm. 4.1.2.) and $R_0(f)$ is given by
	\begin{align*}
		R_0(f):= \frac{1}{d}\sum_{\mu \in G_1 \nu \in G_2} \, f(\mu,\nu) K_\mu \otimes K_\nu \in u_q^0(\g,\Lambda) \otimes u_q^0(\g,\Lambda).
	\end{align*}
	The 
	bilinear form $f:G_1 \times G_2 \to \C^\times$ is of the explicit form
	\begin{align}
		f(\mu,\nu)=q^{-(\mu,\nu)}\cdot g([\mu],[\nu]),
	\end{align} 
	where $g:H_1 \times H_2 \to \C^\times$ is another bilinear form. Images under the canonical projections $G_i \to H_i$ are indicated in square brackets.\\
	Moreover, every small quantum group $u_q(\g,\Lambda)$ with $R$-matrix of the form $R_0(f)\bar{\Theta}$ admits a ribbon element.
\end{theorem}	

	In \cite{LentnerOhrmann2017}, the last two authors listed all possible bilinear forms $g$, such that  the corresponding bilinear form $f$ is non-degenerate. Furthermore, they gave necessary and sufficient conditions on $f$ for the corresponding monodromy matrix of the $R$-matrix $R=R_0(f)\bar{\Theta}$ to be factorizable and checked again explicitly when this is  the case.

\begin{convention}
	When $G$ is a finite abelian group, we denote the dual group by $\widehat{G}$. The characters $\chi \in \widehat{G}$ are denoted by greek letters and the product is written multiplicatively. The inverse of a character $\chi \in \widehat{G}$ is denoted by $\bar{\chi}$.
\end{convention}

\begin{remark} \label{relation sigma-f}
	The element $R_0(f) \in \C[G] \otimes \C[G]$ itself is an $R$-matrix of the group algebra $\C[G]$, leading to a braiding $\sigma$ in the category $\mathsf{Rep}_{\C[G]} \cong \mathsf{Rep}G \cong \Vect_{\widehat{G}}$ which is defined on simple objects (i.e. characters) $\C_{\chi},\C_{\psi} \in \Vect_{\widehat{G}}$ by
	\begin{align*}
		\sigma_{\chi,\psi}(1_{\chi} \otimes 1_\psi)&=\left( \frac{1}{d}\sum_{\mu \in G_1 \nu \in G_2} \, f(\mu,\nu) \chi(\mu)\psi(\nu) \right) \cdot 1_{\psi} \otimes 1_\chi \\
		&= \chi \left( f^{-1}(\bar{\psi}|_{G_2})\right) \cdot 1_{\psi} \otimes 1_\chi \\
		&= \sigma(\chi,\psi) \cdot 1_{\psi} \otimes 1_\chi, 
	\end{align*}
	where in the second line the bicharacter $f:G_1 \times G_2 \to \C^\times$ is interpreted as a homomorphism $G_1 \to \widehat{G_2}$. From this, it is clear that $\sigma$ is a bilinear form on $\widehat{G}$ and $(\sigma,1) \in Z_{ab}(\widehat{G})$ is an abelian $3$-cocycle with trivial $3$-cocycle $\omega=1$.
\end{remark}

So far, we introduced the small quantum group $u_q(\g,\Lambda)$ associated to a simple Lie algebra $\g$, an intermediate lattice $\Lambda_R \subseteq \Lambda \subseteq \Lambda_W$ and an $\ell$th root of unity $q$ with $R$-matrix induced by a bilinear form $f$. We now look at the explicit case $\g=\mathfrak{sl}_2$, $\Lambda_1=\Lambda_2=\Lambda_W$.

\begin{example}\label{ex:Azats exampl, pt1}	
	\mbox{}
	\begin{itemize}
		\item The Cartan part $u_{q}(\mathfrak{sl}_2,\Lambda_W)_0$ of this quantum group is given by
		$\C[G]=\C[\Lambda_W/\Cent_{\Lambda_R}(\Lambda_W)]$. Since $\Lambda_R=2\Lambda_W=\Z\cdot \alpha$ and $\Lambda_W'=\ell \Lambda_R$, we have $G \cong \Z_{2\ell}$. Moreover, we have $H_1=H_2 \cong \Z_2$. 
		\item For the defining bilinear form $f:G \times G \to \C^\times$ of the $R_0$-matrix we have two possibilities, namely $f_{\pm}\left(\lambda,\lambda\right)= \pm q^{-(\lambda,\lambda)}=\pm \exp(\frac{-\pi i}{\ell})$, where $\lambda=[\frac{\alpha}{2}]$ is a generator of $G$. From Table~1 in \cite{LentnerOhrmann2017}, we see that in the case $2 \nmid \ell$ only $f_+$ is non-degenerate. In the even case, both choices are allowed.
		\item The radical $\text{Rad}(f\cdot f^T) \subseteq G$  is given by $\ell\Lambda_W / \ell \Lambda_R \cong \Z_2$. Dualizing the representation category of $\C[G]$ with braiding induced by $f_{\pm}$ leads to the braided monoidal category $\Vect_{\widehat{G}}^{(\sigma_{\pm},1)}$, where $\sigma_{\pm}(\chi,\chi):=f_{\pm}(\lambda,\lambda)^{-1}$, where $\chi:=f(\lambda,\_)$ is a generator of $\widehat{G}$. We always use the non-degenerate form $f$ to identify $G$ and $\widehat{G}$. In particular, the radical $T:=\text{Rad}(\Beta) \subseteq \widehat{G}$ of the associated symmetric bihomomorphism $\Beta=\sigma\sigma^T$ defined in Thm. \ref{3rdAbCohomology and QuadraticForms} is isomorphic to, but not equal to the dual of $\text{Rad}(f\cdot f^T)$. It is generated by $\tau := \chi^\ell$. As it is non-trivial for both cases, $\ell$ odd and $\ell$ even, the monodromies $R_21 R$ are not factorizable and so the corresponding braided categories are not modular. 
		\item We now want to check, when the conditions for modularizability given in Prop. \ref{Modularizability of VectG} are satisfied. It is easy to see that the corresponding quadratic form $Q_{\pm}(\chi)=\sigma_{\pm}(\chi,\chi)$ is trivial on $T$ if and only if $\ell$ is even and $f=f_+$ or $\ell$ is odd and $f=f_-$. Combined with the non-degeneracy condition on $f$ from above, this excludes the case $2 \nmid \ell$ for both $f$ and the case $2\mid \ell$ for $f_-$. From now on, we restrict to the case $2 \mid \ell$ and $f=f_+$. Here, both possibilities $\eta_{\pm}(\chi)=\pm 1$ are allowed. We are now looking for an explicit abelian 3-cocycle $(\bar{\omega},\bar{\sigma})\in Z^3_{ab}(\widehat{G}/T)$ corresponding to the pushed down quadratic form $\bar{Q}_{+}([\chi])=Q_{+}(\chi)$ on $\widehat{G}/T$. For concenience, we use the bracket notation $[\_]$ for the quotient map $\widehat{G} \to \widehat{G}/T$. It turns out that the following definition does the job: 
		\begin{align*}
		\bar{\omega}([\chi]^i,[\chi]^j,[\chi]^k):&= q^\frac{i(j+k-[j+k])}{2} \qquad 0 \leq i,j,k \leq \ell-1\\
		\bar{\sigma}([\chi]^i,[\chi]^j):&= q^\frac{ij}{2} \qquad 0 \leq i,j \leq \ell-1.	
		\end{align*} 
		Here, $[j+k]=j+k-\ell\delta_{j+k \geq \ell}$. It is  clear, that $\bar{\sigma}$ is not a bihomomorphism anymore. For further use, we introduce a $2$-cochain
		\begin{align*}
			\zeta_t([\chi]^i,[\chi]^j):=\begin{cases} q^\frac{-tj}{2} &\mbox{if } 		
			i \text{ odd} \\
			1 & \mbox{else} 
			\end{cases} \qquad 0 \leq i,j \leq \ell-1, \quad 2 \nmid t,
		\end{align*}
		leading to an equivalent abelian $3$-cocycle $(\bar{\omega}_t,\bar{\sigma}_t)=(d\zeta_t^{-1}\bar{\omega},\frac{\zeta_t}{\zeta_t^T}\bar{\sigma})$:
		\begin{align*}
		\bar{\omega}_t([\chi]^i,[\chi]^j,[\chi]^k):&=\begin{cases} q^{tk} &\mbox{if } 		i,j \text{ odd} \\
		1 & \mbox{else} 
		\end{cases} \qquad 0 \leq i,j,k \leq \ell-1\\
		\bar{\sigma}_t([\chi]^i,[\chi]^j):&=q^{-\frac{t}{2}(j\delta_{2\nmid i} - i \delta_{2\nmid j})}q^{\frac{ij}{2}} \qquad 0 \leq i,j \leq \ell-1.
		\end{align*}
		Note that $\bar{\omega}_t$ is the inverse of the $3$-cocycle defined in \cite[Eq. (6.43)]{CreutzigGainutdinovRunkel2017}, since we used a different convention for abelian $3$-cocycles.\\
		\item \textbf{Summarizing:} For $\Lambda=\Lambda_W$, we have for $\ell $ odd a single $R$-matrix, which is not modularizable and for $\ell $ even, we have two $R$-matrices, one of which is not modularizable and one of which modularizes to a modular tensor category with $|\Z_\ell|=\ell$ many simple objects. We have two choices for the ribbon structure in this category.
	\end{itemize} 
\end{example}

	\section{Definition of a quasi-Hopf algebra $u(\omega,\sigma)$}\label{sec: def of u(omega,sigma)}

Let $k$ be an algebraically closed field of characteristic zero. In this section we construct a quasi-Hopf algebra $u(\omega,\sigma)$ from the following data:
\begin{itemize}
	\item a finite abelian group $G$ 
	\item an abelian $3$-cocycle $(\omega,\sigma) \in Z^3_{ab}(\widehat{G})$ on its dual
	\item a subset $\{\,\chi_i \in \widehat{G} \, \}_{1 \leq i \leq n} \subseteq \widehat{G}$
\end{itemize}
The following theorem summarizes the results of this chapter:
\begin{theorem} \label{thm:quasi-Hopf algebra}
	Given the above data, there is a quasi-Hopf algebra $u(\omega,\sigma)$ with the following properties:
	\begin{enumerate}
		\item The quasi-Hopf algebra $u(\omega,\sigma)$ contains $k^{\widehat{G}}_{\omega}$ (see Ex. \ref{twisted dual group algebra}) as a quasi-Hopf subalgebra. We introduce the elements
		\begin{align*}
		K_{\chi}:&= \sum_{\psi \in \widehat{G}}\, \sigma(\chi, \psi) \, \delta_{\psi} \in k^{\widehat{G}}\\
		\bar{K}_{\chi}:&= \sum_{\psi \in \widehat{G}}\, \sigma(\psi,\chi) \, \delta_{\psi} \in k^{\widehat{G}}
		\end{align*}
		Since $B=\sigma\sigma^T$ is a bihomomorphism, the element $K_{\chi}\bar{K}_{\chi}$ is grouplike. 
		\item Let $V$ be the $k$-vector space spanned by basis elements $\{F_i\}_{1\leq i \leq n}$ and endowed with the Yetter-Drinfeld module structure over $k^{\widehat{G}}_{\omega}$ from Cor. \ref{cor: our YDM}. In particular, action, coaction and braiding are given by:
		\begin{align*}
		L.F_i:&= \bar{\chi}_i(L)F_i \qquad L \in G \cong \widehat{\widehat{G}} \subseteq k^{\widehat{G}}, \\
		\delta(F_i):&= L_i \otimes F_i, \qquad L_i:=K_{\bar{\chi}_i} \\
		c_{V,V}(F_i \otimes F_j):&=q_{ij} F_j \otimes F_i, \qquad q_{ij}:=\sigma(\bar{\chi}_i,\bar{\chi}_j).
		\end{align*}  
		Let $B(V)$ denote the corresponding Nichols algebra (see Section \ref{sec:A Nichols algebra}). We assume $B(V)$ to be finite-dimensional. The quasi-Hopf algebra $u(\omega,\sigma)$ contains the Radford biproduct $u(\omega,\sigma)^{\leq}:=B(V) \# k^{\widehat{G}}_{\omega}$ (see Section \ref{sec: radford biproduct}) as a quasi-Hopf subalgebra.
		Explicit relations are given in Cor. \ref{cor: rel's Nichols algebra} and Prop. \ref{prop: RelationsRadfordBiproduct}. Under the assumption from Def.~\ref{def: nice 3-cocycle}, they simplify considerably (see Ex. \ref{ex: u(omega,sigma) for nice 3-cocycle}).
		\item The quasi-Hopf algebra $u(\omega,\sigma)$ is a quotient of the Drinfeld double $D(u(\omega,\sigma)^{\leq})$ (see Section \ref{sec: Drinfeld double}) of the Radford biproduct $u(\omega,\sigma)^{\leq}$ \footnote{In contrast to the ($\omega=1$)-case, $D(u(\omega,\sigma)^{\leq})$ contains $u(\omega,\sigma)^{\leq}$ as a quasi-Hopf subalgebra, but not $(u(\omega,\sigma)^{\leq})^*$, which is a coquasi-Hopf algebra}. After defining certain elements $E_i \in \left( u(\omega,\sigma)^{\leq} \right) ^*$ the quasi-quantum group  $u(\omega,\sigma)$ is generated by $E_i$'s, $F_j$'s and $\delta_{\chi}$'s, where $1\leq i,j \leq n$ and $\chi \in \widehat{G}$. Note that the elements $K_\chi$ do not necessarily form a basis of $k^{\widehat{G}}$.
		\item The finite-dimensional quasi-Hopf algebra $u(\omega,\sigma)$ has a canonical quasi-triangular structure defined in Prop. \ref{prop: R-matrix}.
		\item As a vector space, we have $u(\omega,\sigma)\cong B(V)\otimes kG\otimes B(V^*)$. 		
	\end{enumerate}		
	 In the following, we give a list of relations:
	 \begin{align*}
	 \Delta(F_i)&= K_{\bar{\chi}_i} \otimes F_i \left( \sum_{\chi,\psi}\, \theta(\chi|\bar{\chi}_i,\psi)\omega(\bar{\chi}_i,\psi,\chi)^{-1} \,\delta_\chi \otimes \delta_\psi \ \right) + F_i \otimes 1\left( \sum_{\chi,\psi}\,\omega(\bar{\chi}_i,\chi,\psi)^{-1} \,\delta_\chi \otimes \delta_\psi \right)   \\
	 \Delta(E_i)&= \left( \sum_{\chi,\psi}\, \theta(\psi|\chi\bar{\chi}_i,\chi_i)^{-1}\omega(\psi,\chi,\bar{\chi}_i)^{-1} \,\delta_\chi \otimes \delta_\psi \ \right) E_i \otimes \bar{K}_{\chi_i}  + \left( \sum_{\chi,\psi}\, \omega(\chi,\psi,\bar{\chi}_i)^{-1} \,\delta_\chi \otimes \delta_\psi \right)1\otimes E_i  \\
	 \Delta(K_\chi)&=(K_\chi \otimes K_\chi)  P_{\chi}^{-1} \qquad \Delta(\bar{K}_\chi)=(\bar{K}_\chi \otimes \bar{K}_\chi) P_{\chi}, \qquad P_\chi:=\sum_{\psi,\xi}\,\theta(\chi|\psi,\xi) \,\delta_\psi \otimes \delta_\xi\\ 
	 &[E^a_iK_{\chi_i},F^b_j]_{\sigma} = \delta_{ij}\sigma(\chi_i,\bar{\chi}_i) 
	 \left(1-K_{\chi_i} \bar{K}_{\chi_i}
	 \right)\left( \sum_{\xi}\, \frac{a_i(\xi)b_i(\xi\chi_i)}{\omega(\bar{\chi}_i,\chi_i,\xi)}\, \delta_\xi \right),\text{ where}\\
	 &E^a_i:=E_i\left( \sum_{\xi}\, a_i(\xi)\,\delta_\xi\right)  \qquad F^b_j:=F_j\left( \sum_{\xi}\, b_j(\xi)\, \delta_\xi\right) , \quad\text{with } a_i,b_j \text{ solutions to Eq. \ref{eq: equation for the drinfeld double}}.\\
	 &K_\chi E_i = \sigma(\chi,\chi_i) E_iK_\chi Q^{-1}_{\chi,\chi_i}, \qquad \bar{K}_\chi E_i = \sigma(\chi_i,\chi) E_i \bar{K}_\chi Q_{\chi,\chi_i}, \qquad Q_{\chi,\psi}:= \sum_{\xi} \, \theta(\chi|\xi,\psi)\,\delta_{\xi}\\
	 &K_\chi F_i = \sigma(\chi,\bar{\chi}_i) F_iK_\chi Q^{-1}_{\chi,\bar{\chi}_i}, \qquad \bar{K}_\chi F_i = \sigma(\bar{\chi}_i,\chi) F_i \bar{K}_\chi Q_{\chi,\bar{\chi}_i}\\
	 &S(F_i)= -\left( \sum_{\psi} \, \omega(\bar{\psi},\bar{\chi}_i,\chi_i\psi)d\sigma(\chi_i,\psi,\bar{\psi})\theta(\bar{\psi}|\psi\chi_i,\bar{\chi}_i)^{-1} \delta_\psi\right) K_{\chi_i}F_i\\
	 &S(E_i)=- E_i\bar{K}_{\chi_i}^{-1}\left( \sum_{\psi} \, \frac{\omega(\bar{\chi}_i\bar{\xi},\chi_i,\xi)}{\omega(\bar{\xi},\bar{\chi}_i,\chi_i)}\, \delta_{\xi}\right)\\
	 & \epsilon(K_\chi)=\epsilon(\bar{K}_\chi)=1, \qquad \epsilon(E_i)=\epsilon(F_i)=0, \qquad 1_{u(\omega,\sigma)}=K_1	
	 \end{align*}
	 Here, we have omitted the inclusion $\iota: B(V)\#k^{\widehat{G}} \to u(\omega,\sigma)$ and the quotient map $[\_]:D\left(B(V)\#k^{\widehat{G}}_\omega \right) \to u(\omega,\sigma) $. 	 
\end{theorem}

\begin{remark}\label{rem: another commutator relation}
	Another interesting form of $E_i,F_j$-commutator is the following: If we set 
	\begin{align*}
	M_{ij}:=\sum_{\xi \in \widehat{G}}\, \omega(\bar{\chi}_j,\xi\chi_j,\bar{\chi}_i)^{-1}\,\delta_{\chi},
	\end{align*}
	then we obtain
	\begin{align*}
	E_{i}F_j - M_{ij}F_jE_i = \delta_{ij}(K_{\bar{\chi}_i}-\bar{K}_{\bar{\chi}_i}^{-1}).
	\end{align*}
\end{remark}

\begin{remark}
	The reader familiar with ordinary quantum groups at roots of unity would certainly expect such a construction of a "quasi-quantum group", since up to the technicalities of quasi-Hopf algebras
	 it is based on the construction of quantum groups as Drinfeld doubles of Nichols algebras \cite{AndruskiewitschSchneider2002}.
\end{remark}

\begin{remark}
	In Thm. \ref{thm:quasi-Hopf algebra}, we assumed the Nichols algebra $B(V)$ to be finite-dimensional. This is a condition on the braiding matrix $q_{ij}=\sigma(\bar{\chi}_i,\bar{\chi}_j)$. More precisely, one can show that the Nichols algebra $B(V)$ associated to the Yetter-Drinfeld module $V$ over $k^{\widehat{G}_{\omega}}$ is finite-dimensional if and only if the Nichols algebra $B(V,q_{ij})$ associated to the braided vector space $(V,q_{ij})$ with diagonal braiding $q_{ij}$ is finite-dimensional. Necessary and sufficient conditions for this are known (see \cite{Heckenberger2009}).
\end{remark}

\begin{lemma} \label{lm: grouplike elements}
	The element $K_\chi$ is grouplike if and only if the $2$-cocycle $\theta(\chi)\in Z^2(\widehat{G}/T)$ from Remark \ref{rm: dual of YDM} is trivial.
\end{lemma}

\begin{proof}
	We have
	\begin{align*}
	\Delta(K_\chi)= &\sum_{\psi \in \widehat{G}/T}\, \sigma(\chi, \psi) \, \Delta(\delta_{\psi}) \\
	=&\sum_{\psi_1,\psi_2 \in \widehat{G}/T}\, \sigma(\chi, \psi_1\psi_2) \, \delta_{\psi_1} \otimes \delta_{\psi_1} \\
	=&\sum_{\psi_1,\psi_2 \in \widehat{G}/T}\, \theta(\chi) (\psi_1, \psi_2) \sigma(\chi, \psi_1)\sigma(\chi,\psi_2) \, \delta_{\psi_1} \otimes \delta_{\psi_1}.
	\end{align*}
	This proves the claim.
\end{proof}

\begin{definition}\label{def: nice 3-cocycle} For a given datum $(\omega,\sigma,\chi_i \in \widehat{G})$ as above, the abelian $3$-cocycle $(\omega, \sigma) \in Z^3_{ab}(\widehat{G})$ is called nice if the following two conditions are fulfilled:
	\begin{align*}
	\omega(\chi_i,\chi_j,\psi)&=\omega(\chi_j,\chi_i,\psi) \qquad \forall \, \psi \in \widehat{G}\\
	\omega(\bar{\chi}_i,\chi_i,\psi) &=1 \qquad \forall \, \psi \in \widehat{G}.
	\end{align*}	  
\end{definition}

\begin{lemma}
	Let $G = \bigoplus_{i=1}^n \, \Z_{m_i}$ be a finite abelian group with generators $g_i$, $i =1, \dots ,n$. Every abelian $3$-cocycle $(\omega,\sigma) \in Z^3_{ab}(G)$ is cohomologous to a nice abelian $3$-cocycle. An explicit representative is given by:
	\begin{align*}
	\omega(a,b,c):&= \prod_{i=1}^n\, Q(g_i)^{m_i \delta_{a_i +b_i \geq m_i} c_i}.
	\end{align*}
\end{lemma}

\begin{proof}
	This is the second part of Prop. \ref{prop: explicit abelian 3-cocycle}.
\end{proof}

\begin{example}\label{ex: u(omega,sigma) for nice 3-cocycle}
	Let $(\omega,\sigma) \in Z^3_{ab}(\widehat{G})$ be a nice abelian $3$-cocycle in the sense of Def. \ref{def: nice 3-cocycle}. As it is stated in the previous theorem, the quasi-Hopf algebra $u(\omega,\sigma)$ is generated by elements $E_i,F_j$ and $\delta_{\chi}$ for $1 \leq i,j \leq n$ and $\chi \in \widehat{G}$. We have the same quasi-Hopf algebra relations as for an arbitrary abelian $3$-cocycle, except that the braided commutator simplifies to:
	\begin{align*}
		[ E_iK_{\chi_i}, F_j]_\sigma &= \delta_{ij}\sigma(\chi_i,\bar{\chi}_i) 
		\left(
		1- \bar{K}_{\chi_i} K_{\chi_i}   \right).
	\end{align*}
\end{example}

\begin{example}\label{ex: u(omega,sigma) for omega=1}
	We now fix a datum $(\g,q,\Lambda,\Lambda_1,\Lambda_2,\Lambda',f)$ as in Section \ref{Quantum groups uq(g,Lambda) and R-matrices}. On $\widehat{G}:=\widehat{\Lambda/\Lambda'}$, we define a bihomomorphism $\sigma$ as in Remark \ref{relation sigma-f} and set $\omega =1$. Moreover, we define $\chi_i:=q^{(\alpha_i,\_)}$ for a choice of simple roots $\alpha_i \in \Lambda_R$. From Ex. \ref{ex: u(omega,sigma) for nice 3-cocycle} it is easy to see that in this case, $u(\omega,\sigma)$ turns out to be the ordinary extended small quantum group $u_q(\g,\Lambda)$ as introduced in Section \ref{Quantum groups uq(g,Lambda) and R-matrices}. 
\end{example}

We now prove the previous theorem by constructing the quasi-Hopf algebra $u(\sigma,\omega)$ step-by-step, starting with the Yetter-Drinfeld module $V$ over the quasi-Hopf algebra $k^{\widehat{G}}_{\omega}$.

\subsection{A Yetter-Drinfeld module}\label{sec: YDM}

We start with the definition of a Yetter-Drinfeld module over a quasi-Hopf algebra $H$. 

\begin{definition}\cite{Majid1998}\cite{Schauenburg2002}\label{YetterDrinfeldModule over QHA}
	Let $H$ be a quasi-Hopf algebra. Let $\rho:H\otimes V \to V$ be a left $H$-module and let $\delta:V \to H \otimes V$, $v\mapsto v_{[-1]}\otimes v_{[0]}$ be a linear map, s.t.
	\begin{enumerate}
		\item $(\epsilon \otimes \text{id})\circ \delta = \text{id}$
		\item  $X^1(Y^1.v)_{[-1](1)} Y^2 \otimes X^2(Y^1.v)_{[-1](2)}Y^3 \otimes X^3.(Y^1.v)_{[0]}$\\
		$=X^1v_{[-1]}\otimes (X^2.v_{[0]})_{[-1]}X^3 \otimes(X^2.v_{[0]})_{[0]}$
		\item $h_{(1)}v_{[-1]}\otimes h_{(2)}v_{[0]} = (h_{(1)}.v)_{[-1]}h_{(2)} \otimes  (h_{(1)}.v)_{[0]}$,
	\end{enumerate} 
	where $\phi=X^1\otimes X^2 \otimes X^3 = Y^1\otimes Y^2 \otimes Y^3 $ denotes the associator of $H$. Then, the triple $(V,\rho,\delta)$ is called a Yetter-Drinfeld module over $H$.
\end{definition}

Compared to the case of Yetter-Drinfeld modules over ordinary Hopf algebras, only the coassociativity of the coaction is modified in the previous definition. Obviously, for $\phi=1\otimes 1\otimes 1$ it matches the usual definition. As in this case, we have the following:

\begin{proposition} \cite{Majid1998} \cite{BulacuNauwelaerts2002} 
	Let $H$ be a quasi-Hopf algebra. The category ${}^H_H\mathcal{YD}$ of finite dimensional Yetter-Drinfeld modules over $H$ is a braided monoidal category, with usual tensor product $V \otimes W$ of $H$-modules $V,W \in {}^H_H\mathcal{YD}$. The comodule structure on $V \otimes W$ is given by
	\begin{align*}
		\delta_{V \otimes W}(v \otimes w)=& \,X^1(x^1Y^1\cdot v)_{[-1]}x^2(Y^2\cdot w)_{[-1]}X^3 \\
		&\otimes X^2\cdot (x^1Y^1\cdot v)_{[0]} \otimes X^3x^3(Y^2\cdot w)_{[0]}.
	\end{align*}	
	The associator in ${}^H_H\mathcal{YD}$ is the same as in $\mathsf{Rep}_H$ and the braiding given by 
		\begin{align*}
		c_{V,W}(v \otimes w) = v_{[-1]}.w \otimes v_{[0]}.
		\end{align*}
\end{proposition}

If we plug in the data of the twisted dual group algebra $k^G_\omega$ from Example \ref{twisted dual group algebra} and identify $k^G_\omega$-modules with $G$-graded vector spaces, the above definition simplifies significantly \cite{Majid1998}:

\begin{lemma}\label{lemmaYetterDrinfeldModule} For $G$ a finite abelian group, a Yetter-Drinfeld module over $k^G_\omega$ consists of the following data:
	\begin{itemize}
		\item A set of elements $g_i \in G$ for $1 \leq i \leq n$.
		\item a $G$-graded vector space $V = \bigoplus_{i=1}^n \, k\cdot F_i$ with basis $\{F_i\}_{1 \leq i \leq n}$ in degrees $|F_i|=g_i$.
		\item a map $\rho:kG \otimes V  \to V$, $g\otimes F_i \mapsto g.F_i$, s.t.
		\begin{align}\label{eq:YDM over kG}
			g.(h.F_i) =\frac{\omega(g_i,g,h)\omega(g,h,g_i)}{\omega(g,g_i,h)} (g+h).F_i, \qquad 
			F_i =0.F_i, \qquad
			|g.F_i|=|F_i|.
		\end{align}	
	\end{itemize}
\end{lemma}
Eq. \ref{eq:YDM over kG} in the previous lemma looks very similar to the defining relations of an abelian 3-cocycle (see Def. \ref{abelianCohomology}). In particular, we have:
\begin{corollary} \label{cor: our YDM}
	Let $(\sigma,\omega) \in Z^3_{ab}(\widehat{G})$ be an abelian 3-cocycle and $\{\chi_i \}_{i \in I} \subseteq \widehat{G}$ a subset. Then, setting $V := \bigoplus_{i \in I} \, k\cdot F_i$, with homogeneous degrees $|F_i|=\bar{\chi}_i$ and action
	\begin{align*}
	\rho:\; k\widehat{G} \otimes V & \longrightarrow V \\
	  \chi \otimes F_i  &\longmapsto  \sigma(\bar{\chi}_i,\chi) \, F_i,
	\end{align*}
	indeed defines a Yetter-Drinfeld module over $k^{\widehat{G}}_\omega$. Note that substituting $\sigma$ by $(\sigma^T)^{-1}$ also defines a Yetter-Drinfeld structure.
\end{corollary}

\begin{remark} \label{rm: dual of YDM}
	Note that a dual of $V$ is given by $V^\vee = V^*= \bigoplus_{i \in I} k \cdot F_i^*$, with homogeneous degrees $|F_i|=\chi_i$ and action
	\begin{align*}
		\rho^\vee: \; \chi \otimes F_i^*  \mapsto \theta(\bar{\chi})(\chi\chi_i,\bar{\chi}_i) d\sigma(\chi_i,\chi,\bar{\chi})^{-1}\sigma(\bar{\chi}_i,\bar{\chi}) F_i^*= \sigma(\chi_i,\chi) F_i^*,
	\end{align*}
	where $\theta(\chi)(\psi_1,\psi_2)\in Z^2(G,k^G)$ is the $2$-cocycle defined by
	\begin{align*}
		\theta(\chi)(\psi_1,\psi_2):=\frac{\omega(\chi,\psi_1,\psi_2)\omega(\psi_1,\psi_2,\chi)}{\omega(\psi_1,\chi,\psi_2)}.
	\end{align*}
	The evaluation $V^\vee \otimes V \to k$ is given by $F_i^\vee \otimes F_j \mapsto \delta_{i,j}$.
\end{remark}

\subsection{A Nichols algebra}\label{sec:A Nichols algebra}

We now give relations for the corresponding Nichols algebra $B(V) \in {}_{k^G_\omega}^{k^G_\omega}\mathcal{YD}$ of the Yetter-Drinfeld module $V$ constructed in Corollary \ref{cor: our YDM}.\\
As in the Hopf-case (see \cite{Heckenberger2008}), the Nichols algebra $B(V)= \bigoplus_{n \geq 0} \, B^nV:=T(V)/I$ of $V$ is a quotient of the tensor (Hopf-)algebra $T(V):= \bigoplus_{n \geq 0} \, T^nV$ in ${}_{k^G_\omega}^{k^G_\omega}\mathcal{YD}$ by the maximal Hopf ideal $I$, s.t. $B^1V:=V$ are exactly the primitive elements in $B(V)$. A brief introduction to Nichols algebras in arbitrary abelian braided monoidal categories is given in App. \ref{app: quantum shuffle product}. Details can be found in \cite{BazlovBerenstein2013}.\\ 
Since we are dealing with a non-trivial associator, we have to fix a bracketing $T^nV:=T^{n-1}V \otimes V$ and $T^0V := k$. Accordingly, we define $F_i^n:=F_i^{n-1}F_i$ for primitive generators $F_i \in V \subseteq B(V)$. In order to compare our results with \cite{Rosso1998}, we use the short-hand notation $q_{ij}:=\sigma(\bar{\chi}_i,\bar{\chi}_j)$. In particular, we will see that the resulting relation for the adjoint representation on $B(V)$ depends only on $q_{ij}$ and not on the $3$-cocycle $\omega$.

\begin{lemma}
	For $k,l \in \Z_{\geq 0}$, let $a^i_{k,l} \in k^\times$ denote the elements defined by $F_i^k F_i^l=a^i_{k,l} F_i^{k+l}$.
	They are given explicitly by $a^i_{k,l}=\prod_{r=0}^{l-1}\, \omega(\bar{\chi}_i^k, \bar{\chi}_i^r,\bar{\chi}_i)^{-1}$ and satisfy the following identities:
	\begin{enumerate}
		\item $\frac{a^i_{k,l+m}a^i_{l,m}}{a^i_{k+l,m}a^i_{k,l}}=\omega(\bar{\chi}_i^k, \bar{\chi}_i^l,\bar{\chi}_i^m)^{-1}$
		\item $a^i_{k,l}= \frac{\sigma(\bar{\chi}_i^k,\bar{\chi}_i^l)}{q_{ii}^{kl}}a^i_{l,k}$.
	\end{enumerate}
\end{lemma}

\begin{proof}
	We prove the first part by induction in $m$. The case $m=0$ is trivial. For
	\begin{align*}
		\frac{a^i_{k,l+m+1}a^i_{l,m+1}}{a^i_{k+l,m+1}a^i_{k,l}}&=\frac{a^i_{k,l+m}a^i_{l,m}}{a^i_{k+l,m}a^i_{k,l}}
		\frac{\omega(\bar{\chi}_i^{l+k},\bar{\chi}_i^m,\bar{\chi}_i)}{\omega(\bar{\chi}_i^k,\bar{\chi}_i^{l+m},\bar{\chi}_i)\omega(\bar{\chi}_i^l,\bar{\chi}_i^m,\bar{\chi}_i)} \\
		&=\frac{\omega(\bar{\chi}_i^{k+l},\bar{\chi}_i^m,\bar{\chi}_i)}{\omega(\bar{\chi}_i^{k},\bar{\chi}_i^l,\bar{\chi}_i^m)\omega(\bar{\chi}_i^k,\bar{\chi}_i^{l+m},\bar{\chi}_i)\omega(\bar{\chi}_i^l,\bar{\chi}_i^m,\bar{\chi}_i)} \\
		&=\omega(\bar{\chi}_i^{k},\bar{\chi}_i^l,\bar{\chi}_i^{m+1})^{-1}.
	\end{align*}
	For the second part, we compute:
	\begin{align*}
		a_{k+1,l}^i&= \prod_{r=0}^{l-1}\, \omega(\bar{\chi}_i^{k+1},\bar{\chi}_i^{r},\bar{\chi}_i)^{-1} \\
		&=\prod_{r=0}^{l-1}\,\frac{\omega(\bar{\chi}_i,\bar{\chi}_i^{k},\bar{\chi}_i^{r+1})}{\omega(\bar{\chi}_i,\bar{\chi}_i^{k},\bar{\chi}_i^{r})\omega(\bar{\chi}_i^k,\bar{\chi}_i^{r},\bar{\chi}_i)\omega(\bar{\chi}_i,\bar{\chi}_i^{k+r},\bar{\chi}_i)} \\
		&=\omega(\bar{\chi}_i,\bar{\chi}_i^{k},\bar{\chi}_i^l)a_{k,l}^i\prod_{r=0}^{l-1}\,\omega(\bar{\chi}_i,\bar{\chi}_i^{k+r},\bar{\chi}_i)^{-1} \\
		&=\omega(\bar{\chi}_i,\bar{\chi}_i^{k},\bar{\chi}_i^l)a_{k,l}^i \prod_{r=0}^{l-1}\, \frac{\sigma(\bar{\chi}_i^{k+r},\bar{\chi}_i)\sigma(\bar{\chi}_i,\bar{\chi}_i)}{\sigma(\bar{\chi}_i^{k+r+1},\bar{\chi}_i)} \\
		&=\omega(\bar{\chi}_i,\bar{\chi}_i^{k},\bar{\chi}_i^l)a_{k,l}^i \frac{\sigma(\bar{\chi}_i^{k},\bar{\chi}_i)q_{ii}^l}{\sigma(\bar{\chi}_i^{k+l},\bar{\chi}_i)}.
	\end{align*}
	We want to prove the second part by induction in $l$. For $l=1$, we obtain
	\begin{align*}
		a^i_{1,k}=\prod_{r=0}^{k-1}\, \omega(\bar{\chi}_i,\bar{\chi}_i^{r},\bar{\chi}_i)^{-1}
		=\prod_{r=0}^{k-1}\, \frac{\sigma(\bar{\chi}_i^{r},\bar{\chi}_i)\sigma(\bar{\chi}_i,\bar{\chi}_i)}{\sigma(\bar{\chi}_i^{r+1},\bar{\chi}_i)}=\frac{q_{ii}^k}{\sigma(\bar{\chi}_i^{k},\bar{\chi}_i)}\cdot 1=\frac{q_{ii}^k}{\sigma(\bar{\chi}_i^{k},\bar{\chi}_i)}a^i_{k,1}.			
	\end{align*}
	For $l+1$ we obtain
	\begin{align*}
		a^i_{k,l+1}&=a^i_{k,l} \omega(\bar{\chi}_i^k,\bar{\chi}_i^l,\bar{\chi}_i)^{-1} \\
		&=\frac{\sigma(\bar{\chi}_i^k,\bar{\chi}_i^l)}{q_{ii}^{kl}}\omega(\bar{\chi}_i^k,\bar{\chi}_i^l,\bar{\chi}_i)^{-1}a^i_{l,k} \\
		&=\frac{\sigma(\bar{\chi}_i^k,\bar{\chi}_i^l)}{q_{ii}^{kl}}\omega(\bar{\chi}_i^k,\bar{\chi}_i^l,\bar{\chi}_i)^{-1}\omega(\bar{\chi}_i,\bar{\chi}_i^l,\bar{\chi}_i^k)^{-1}\frac{\sigma(\bar{\chi}_i^{k+l},\bar{\chi}_i)}{\sigma(\bar{\chi}_i^{l},\bar{\chi}_i)q_{ii}^k}a_{l+1,k}^i \\
		&= \frac{\sigma(\bar{\chi}_i^{l+1},\bar{\chi}_i^k)}{q_{ii}^{k(l+1)}}a_{l+1,k}^i.
	\end{align*}
	In the last line, we used the identity $\omega\cdot \omega^T=d\sigma^{-1}$.
\end{proof}

\begin{lemma} Let $\text{ad}_c(X)(Y)=\mu\circ(\id - c)(X\otimes Y)$ be the adjoint representation of the associative algebra $B(V) \in {}_{k^G_\omega}^{k^G_\omega}\mathcal{YD}$. Here, $c$ denotes the braiding and $\mu$ is the multiplication on $B(V)$. We have 
	\begin{align*}
	\text{ad}^n_c(F_i)(F_j)&= \sum_{k=0}^n \, \mu_n(k)\,\left(F_i^kF_j\right) F_i^{n-k}\, \text{, where}\\
	\mu_n(k)&=(-1)^{n-k}\sigma(\bar{\chi}_i^{n-k},\bar{\chi}_j)\frac{\omega(\bar{\chi}_i^k,\bar{\chi}_i^{n-k},\bar{\chi}_j)}{\omega(\bar{\chi}_i^k,\bar{\chi}_j,\bar{\chi}_i^{n-k})}{q_{ii}}^{\binom{n-k}{2}-\binom{n}{2}}\prod_{r=0}^{n}\frac{\sigma(\bar{\chi}_i,\bar{\chi}_i^r)}{\omega(\bar{\chi}_i,\bar{\chi}_i^r,\bar{\chi}_j)} (a_{k,n-k}^i)^{-1}\binom{n}{k}_{q_{ii}}
	\end{align*}
\end{lemma}

\begin{proof}
	For the sake of readability, we use the short-hand notation $i$ for $\bar{\chi}_i$ during this proof. 
	For $n=1$, we obtain $ad_c(F_i)(F_j)=F_iF_j - q_{ij} F_jF_i$. For larger $n$, we first want to find an inductive expression for the coefficients in the following expansion:
	\begin{align*}
		ad_c^n(F_i)(F_j)=\sum_{k=0}^n\,\mu_n(k)\,(F_i^kF_j)F_i^{n-k}.
	\end{align*} 
	To this end, we compute
	\begin{align*}
		ad_c^n(F_i)(F_j)&=ad_c(F_i)(ad_c^{n-1}(F_i)(F_j)) \\
		&=ad_c(F_i)\left( \sum_{k=0}^{n-1}\,\mu_{n-1}(k)\,(F_i^kF_j)F_i^{n-k-1} \right) \\
		&= \sum_{k=0}^{n-1}\,\mu_{n-1}(k)\, \left( q_{ii}^k(\sigma(ki,i)\omega(i,ki,j)\omega(i,ki+j,(n-k-1)i))^{-1} \,(F_i^{k+1}F_j)F_i^{n-k-1} \right. \\
		&-\left. \sigma(i,(n-1)i+j)\omega(ki+j,(n-k-1)i,i)\, (F_i^kF_j)F_i^{n-k} \right).
	\end{align*}
	From this, we obtain
	\begin{align*}
		\mu_n(n)&=\prod_{r=0}^{n-1}\,\frac{\sigma(i,ri)}{q^r\omega(i,ri,j)}=:c_n \\
		\mu_n(0)&=(-1)^n\sigma(ni,j)q_{ii}^{\frac{(n-1)n}{2}}c_n\\
		\mu_n(k)&=\mu_{n-1}(k-1)q_{ii}^{(k-1)}\sigma((k-1)i,i)\omega(i,(k-1)i,j)^{-1}\omega(i,(k-1)i+j,(n-k)i)^{-1} \\&-\mu_{n-1}(k)\sigma(i,(n-1)i+j)\omega(ki+j,(n-k-1)i,1).
	\end{align*}
	This allows us to prove the following formula by induction, even though we are going to omit the proof here, since it is long and tedious without any interesting inputs except an exhaustive use of the abelian $3$-cocycle conditions and $q$-binomial coefficients.\footnote{The hard part was rather to find the formula by tracing down the inductive formula to $n=1$, than to prove it.}
	\begin{align*}
		\mu_n(k) = (-1)^{n-k}c_na_{k,n-k}^{-1}\sigma((n-k)i,j) \frac{\omega(ki,(n-k)i,j)}{\omega((n-k)i,ki,j)}q_{ii}^{\binom{n-k}{2}} \binom{n}{k}_{q_{ii}}.
	\end{align*}
\end{proof}

\begin{proposition} The coproducts of $F_i^n$ and $\text{ad}^n_c(F_i)(F_j)$ are given by
	\begin{align*}
	\Delta(F_i^n)&= \sum_{k=0}^n \, (a^i_{k,n-k})^{-1} \binom{n}{k}_{q_{ii}} \,F_i^k \otimes F_i^{n-k} \\
	\Delta(\text{ad}^n_c(F_i)(F_j))&= \left(\text{ad}^n_c(F_i)(F_j) \right) \otimes 1 +1 \otimes \left(\text{ad}^n_c(F_i)(F_j) \right)\\ 
	&+\sum_{k=0}^n\, \mu_n(k) \frac{\left(q_{ii}^{2k}q_{ij}q_{ji}\right)^{n-k}}{\sigma(\bar{\chi}_i^{n-k},\bar{\chi}_i^k\bar{\chi}_j)} \sum_{m=0}^{n-k-1}\, \frac{\omega(\bar{\chi}_i^{n-k-m},\bar{\chi}^m_i,\bar{\chi}^k_i\bar{\chi}_j)}{\sigma(\bar{\chi}_i^k\bar{\chi}_j,\bar{\chi}_i^m)}(a^i_{n-k-m,m})^{-1}\binom{n-k}{m}_{q_{ii}}\\ &\times \prod_{r=0}^{n-k-m-1}\,\left(1-\frac{q^{k+m}}{q_{ii}^r q_{ij}q_{ji}} \right) \,F_i^{n-k-m} \otimes \left(F_i^kF_j\right) F_i^m
	\end{align*}
\end{proposition}

\begin{proof}
	During this proof, we use the abbreviation $i$ for $\bar{\chi}_i$.
	We first proof the equation for $\Delta(F_i^n)$. We set $\Delta(F_i^n):=\sum_{k=0}^{n}\,f_n(k) \, F_i^k\otimes F_i^{n-k}$. First, we want to find an inductive relation between the coefficients $f_n(k)$. We have 
	\begin{align*}
		\Delta(F_i^n) &= \Delta(F_i^{n-1})\Delta(F_i) \\
		&=\Delta(F_i^{n-1})(F_i \otimes 1 + 1 \otimes F_i) \\
		&=\sum_{k=0}^{n-1}\,f_{n-1}(k) \, (F_i^k\otimes F_i^{n-k})(F_i \otimes 1 + 1 \otimes F_i)
	\end{align*}
	After computing the product, we obtain
	\begin{align*}
		f_n(k)&=f_{n-1}(k)\omega(k i,(n-1-k) i,i)\\
		&+f_{n-1}(k-1)\frac{\omega((k-1) i,(n-k) i,i)}{\omega((k-1) i,i,(n-k) i)} \sigma((n-k) i,i).
	\end{align*}
	Moreover, we have $f_n(0)=f_{n-1}(0)=1$ and $f_n(n)=f_{n-1}(n-1)=1$. Now, we want to show the following formula by induction:
	\begin{align*}
		f_n(k)=\prod_{r=0}^{(n-k)-1} \, \omega(ki,ri,i)\,\binom{n}{k}_{q_{ii}}.
	\end{align*}
	For $n=1$, we obtain $\binom{1}{0}_{q_{ii}}=1=f_1(0)$ and $\binom{1}{1}_{q_{ii}}=1=f_1(1)$. Now, we assume that the formula holds for $n-1$. Then,
	\begin{align*}
		f_n(k)&=f_{n-1}(k)\omega(k i,(n-1-k) i,i)\\
		&+f_{n-1}(k-1)\frac{\omega((k-1) i,(n-k) i,i)}{\omega((k-1) i,i,(n-k) i)} \sigma((n-k) i,i) \\
		&=\prod_{r=0}^{(n-1-k)-1} \, \omega(ki,ri,i)\,\binom{n-1}{k}_{q_{ii}}\omega(k i,(n-1-k) i,i) \\
		&+\prod_{r=0}^{(n-k)-1} \, \omega((k-1)i,ri,i)\,\binom{n-1}{k-1}_{q_{ii}}\frac{\omega((k-1) i,(n-k) i,i)}{\omega((k-1) i,i)} \sigma((n-k) i,(n-k) i,i) \\
		&=\prod_{r=0}^{(n-k)-1} \, \omega(ki,ri,i)\,\binom{n-1}{k}_{q_{ii}} \\
		&+ \prod_{r=0}^{(n-k)} \,\frac{\omega(ki,ri,i)\omega(i,(k-1)i,(r+1)i)}{\omega(i,(k-1)i,ri)\omega(i,(k+r-1)i,i)} \frac{\sigma((n-k) i,i)}{\omega((k-1)i,i,(n-k)i)}\binom{n-1}{k-1}_{q_{ii}} \\
		&=\prod_{r=0}^{(n-k)-1} \, \omega(ki,ri,i)\,\binom{n-1}{k}_{q_{ii}} \\
		&+ \prod_{r=0}^{(n-k)-1} \,\omega(ki,ri,i)\frac{\omega(ki,(n-k)i,i)\omega(i,(k-1)i,(n-k+1)i)}{\omega(i,(k-1)i,i)} \frac{q_{ii}^{(n-k)}\sigma(ki,i)}{\sigma(ni,i)} \\
		&\times \frac{\sigma((n-k) i,i)}{\omega((k-1)i,i,(n-k)i)}\binom{n-1}{k-1}_{q_{ii}}\\
		&=\prod_{r=0}^{(n-k)-1} \omega(ki,ri,i)\left(\binom{n-1}{k}_{q_{ii}} - \right. \\
		&+\left. \frac{\omega(ki,(n-k)i,i)\omega(i,ki,(n-k)i)}{\omega(ki,i,(n-k)i)} \frac{\sigma(ki,i)\sigma((n-k)i,i)}{\sigma(ni,i)}q_{ii}^{(n-k)}\binom{n-1}{k-1}_{q_{ii}}\right) \\
		&=\prod_{r=0}^{(n-k)-1}\, \omega(ki,ri,i)\left(\binom{n-1}{k}_{q_{ii}} - q_{ii}^{(n-k)}\binom{n-1}{k-1}_{q_{ii}}\right) \\
		&=\prod_{r=0}^{(n-k)-1}\,\omega(ki,ri,i)\binom{n}{k}_{q_{ii}},	
	\end{align*}
	where we used the abelian $3$-cocycle conditions exhaustively.	\\
	Instead of giving a rigorous proof for the second part, which would take a few pages, we describe instead what we did. First, we computed the coefficients $A$, $B$ of the following expression:
	\begin{align*}
		\Delta\left( (F_i^kF_j)F_i^{(n-k)}\right) &= \left( \Delta\left( F_i^k\right) \Delta(F_j)\right) \Delta\left( F_i^{(n-k)}\right) \\
		&=\sum_{l=0}^k\sum_{m=0}^{n-k}\,f_k(l)f_{n-k}(m)\left(A\, \left(F_i^lF_j \right)F_i^m \otimes F_i^{n-m-l}\right.  \\&+ \left. B \,F_i^{m+l}\otimes \left(F_i^{(k-l)}F_j \right)F_i^{n-k-m}   \right).
	\end{align*}
	Then we plugged this in $\Delta(\text{ad}^n_c(F_i)(F_j))$, using the expansion for $\text{ad}^n_c(F_i)(F_j)$ from the previous lemma:
	\begin{align*}
		\Delta(\text{ad}^n_c(F_i)(F_j)) &= \sum_{k=0}^n\sum_{l=0}^{k}\sum_{m=0}^{n-k}\,\mu_n(k)\binom{k}{l}_{q_{ii}}\binom{n-k}{m}_{q_{ii}}(a^i_{l,k-l}a^i_{m,n-k-m})^{-1} \\
		&\times\left(A\, \left(F_i^lF_j \right)F_i^m \otimes F_i^{n-m-l}+ B \,F_i^{m+l}\otimes \left(F_i^{(k-l)}F_j \right)F_i^{n-k-m}   \right)
	\end{align*}	
	After changing the order of summation and plugging in our expression for the coefficients $\mu_n(k)$ from the previous lemma, we see that the $B$-summand cancels completely, whereas the $A$-summand can be brought in to a form, where we can apply the $q$-binomial theorem. This gives the above result.
\end{proof}

From the above coproducts we can read off the following relations:

\begin{corollary}\label{cor: rel's Nichols algebra} For any $n \in \N$,
	\begin{enumerate}
		\item  $F_i^n\neq0$ if and only if $(n)_{q_{ii}}! \neq 0$.
		\item  $\text{ad}^n_c(F_i)(F_j) \neq 0$ if and only if $(n)_{q_{ii}}!\prod_{r=0}^{n-1} \, \left( 1-q_{ii}^rq_{ij}q_{ji} \right) \neq 0 $.
	\end{enumerate}
\end{corollary}

\begin{remark}\label{rm: nichols alg rel's dont depend on omega}
	Note that the above relations do not depend on the $3$-cocycle $\omega$ and are identical with the ones given in \cite{Rosso1998}, Lemma 14.
\end{remark}

\subsection{A Radford biproduct}\label{sec: radford biproduct}

For a general quasi-Hopf algebra $H$ and a braided Hopf algebra $B \in {}^H_H \mathcal{YD}$ in the category of Yetter-Drinfeld modules over $H$, the Radford biproduct $B\# H$ was defined in \cite[Sec. 3]{BulacuNauwelaerts2002}. It is again a quasi-Hopf algebra.\\
The Nichols algebra $B(V)$ constructed in the previous section is a Hopf algebra in the category ${}^{k^{\widehat{G}}_\omega}_{k^{\widehat{G}}_\omega} \mathcal{YD}$ of Yetter-Drinfeld modules over $k^{\widehat{G}}_\omega$ and thus the definition in \cite{BulacuNauwelaerts2002} applies. We collect the relevant relations in the following proposition:

\begin{proposition}\label{prop: RelationsRadfordBiproduct}
	As a vector space, the quasi-Hopf algebra $B(V)\# k^{\widehat{G}}_{\omega}$ is given by $B(V) \otimes k^{\widehat{G}}_{\omega}$. A general product of generators is given by
	\begin{align*}
	(F_{i}\#\delta_{\psi_1})(F_{j}\#\delta_{\psi_2})=\delta_{\psi_1\chi_j,\psi_2}\omega(\bar{\chi}_i,\bar{\chi}_j,\psi_2)^{-1}(F_i F_j) \# \delta_{\psi_2}.
	\end{align*}
	The inclusion $k^{\widehat{G}}_{\omega} \hookrightarrow B(V)\# k^{\widehat{G}}_{\omega}$, $\delta_\psi \mapsto 1 \# \delta_\psi$ is a homomorphism of quasi-Hopf algebras. This legitimises the short-hand notation $\delta_\chi=1\# \delta_\chi \in  B(V)\# k^{\widehat{G}}_{\omega}$. Moreover, we set $F_i=F_i \# 1_{k^{\widehat{G}}_{\omega}}$. Then, the following algebra relations hold:
	\begin{itemize}
		\item $\delta_\chi \cdot (F_i \# \delta_\psi) = \delta_{\chi\chi_i,\psi}(F_i \# \delta_\psi)$,
		in particular $\delta_\chi\cdot F_i = F_i \# \delta_{\chi\chi_i}$
		\item $(F_i \# \delta_\chi) \cdot \delta_\psi = \delta_{\chi,\psi} (F_i \# \delta_\psi)$,
		in particular $F_i \cdot \delta_\chi = F_i \# \delta_\chi$
		\item $F_iF_j=(F_iF_j)\#\left(\sum_{\chi \in \widehat{G}} \, \omega(\bar{\chi}_i,\bar{\chi}_j,\chi)^{-1} \,\delta_\chi  \right) $
	\end{itemize}
	The comultiplication is given by
	\begin{align*}
	\Delta(F_i \# \delta_\psi) = \sum_{\psi' \in G} \; \frac{\omega(\psi',\bar{\chi}_i,\psi\bar{\psi}')}{\omega(\bar{\chi}_i,\psi',\psi\bar{\psi}')}  \sigma(\bar{\chi}_i,\psi') \delta_{\psi'} \otimes (F_i \# \delta_{\psi\bar{\psi}'}) + \sum_{\psi' \in \widehat{G}}\;  \omega(\bar{\chi}_i,\psi',\psi\bar{\psi}')^{-1} (F_i \# \delta_{\psi'}) \otimes \delta_{\psi\bar{\psi}'},
	\end{align*}
	in particular
	\begin{align*}
	\Delta(F_i)= &\sum_{\chi,\psi \in \hat{G}} \; \frac{\omega(\chi,\bar{\chi}_i,\psi)}{\omega(\bar{\chi}_i,\chi,\psi)}  \sigma(\bar{\chi}_i,\chi) \delta_{\chi} \otimes (F_i \# \delta_{\psi}) + \sum_{\chi,\psi \in \widehat{G}}\;  \omega(\bar{\chi}_i,\chi,\psi)^{-1} (F_i \# \delta_{\chi}) \otimes \delta_{\psi}.
	\end{align*}
	The antipode is given by
	\begin{align*}
		S(F_i \# \delta_\psi) = -\omega(\psi,\bar{\chi}_i,\chi_i\bar{\psi}) \sigma(\bar{\chi}_i,\psi) (F_i \# \delta_{\chi_i\bar{\psi}}).
	\end{align*}
\end{proposition}
In the later chapters we will also need the following formula:
	\begin{align*}
	(\Delta \otimes \id) \circ \Delta (F_i) =& \sum_{\psi_1,\psi_2,\psi_3 \in G} \;  \alpha^i_1(\psi_1,\psi_2,\psi_3) (F_i \# \delta_{\psi_1}) \otimes \delta_{\psi_2} \otimes \delta_{\psi_3}\\
	+ & \; \alpha^i_2(\psi_1,\psi_2,\psi_3) \delta_{\psi_1} \otimes (F_i\#\delta_{\psi_2})  \otimes \delta_{\psi_3} \\
	+ &\;\alpha^i_3(\psi_1,\psi_2,\psi_3) \delta_{\psi_1} \otimes \delta_{\psi_2}  \otimes (F_i\#\delta_{\psi_3})	,
	\end{align*}
	where 
	\begin{align*}
		\alpha_1^i(\psi_1,\psi_2,\psi_3) &= \omega(\bar{\chi}_i,\psi_1\psi_2,\psi_3)^{-1}\omega(\bar{\chi}_i,\psi_1,\psi_2)^{-1}\\
		\alpha_2^i(\psi_1,\psi_2,\psi_3) &=\frac{\omega(\psi_1,\bar{\chi}_i,\psi_2)}{\omega(\bar{\chi}_i,\psi_1\psi_2,\psi_3)\omega(\bar{\chi}_i,\psi_1,\psi_2)} \sigma(\bar{\chi}_i,\psi_1)\\
		\alpha_3^i(\psi_1,\psi_2,\psi_3) &= \frac{\omega(\psi_1\psi_2,\bar{\chi}_i,\psi_3)}{\omega(\bar{\chi}_i,\psi_1\psi_2,\psi_3)} \sigma(\bar{\chi}_i,\psi_1\psi_2).			
	\end{align*}
We now want to consider twists of the Radford biproduct as defined in Def. \ref{def: twist}. We are mainly interested in twists coming from elements 
	\begin{align*}
		J=\sum_{\chi,\psi \in \widehat{G}} \; \zeta(\chi,\psi) \, \delta_\chi \otimes \delta_\psi \in k^{\widehat{G}}_\omega \otimes k^{\widehat{G}}_\omega
	\end{align*}
	in the group part of $(B(V)\# k^{\widehat{G}}_{\omega})^{\otimes 2}$. The corresponding coproduct $\Delta^J(F_i)$ is then given by:
\begin{align*}
\Delta^J(F_i) &= \sum_{\chi,\psi \in \widehat{G}} \; \frac{\omega(\chi,\bar{\chi}_i,\psi)}{\omega(\bar{\chi}_i,\chi,\psi)} \sigma(\bar{\chi}_i,\chi) \frac{\zeta(\chi,\psi\bar{\chi}_i)}{\zeta(\chi,\psi)} \delta_\chi \otimes (F_i \# \delta_\psi) \\
&+ \sum_{\chi,\psi \in \widehat{G}} \; \omega(\bar{\chi}_i,\chi,\psi)^{-1} \frac{\zeta(\chi\bar{\chi}_i,\psi)}{\zeta(\chi,\psi)} (F_i \# \delta_\chi) \otimes \delta_\psi.
\end{align*}
Not surprisingly, the corresponding associator is given by
\begin{align*}
\phi^J= \sum_{\psi_1,\psi_2,\psi_3} \; \omega(\psi_1,\psi_2,\psi_3) d\zeta(\psi_1,\psi_2,\psi_3)\, \delta_{\psi_1} \otimes \delta_{\psi_2} \otimes \delta_{\psi_3}.
\end{align*}

\begin{lemma}\label{lm: iso betwwen twist of biproducts}
	The identity on $k^{\widehat{G}}_{\omega d\zeta}$ extends to an isomorphism of quasi-Hopf algebras defined by: 
	\begin{align*}
	f_\zeta:\left( B(V)\# k^{\widehat{G}}_\omega\right) ^J &\longrightarrow B(V)\# k^{\widehat{G}}_{{\omega d\zeta}}\\
	F_i &\longmapsto \sum_{\chi \in \widehat{G}} \; \zeta(\bar{\chi}_i,\chi) \, F_i \# \delta_\chi
	\end{align*}
\end{lemma}
\begin{proof}
	This is a special case of Thm. 5.1. in \cite{BulacuNauwelaerts2002}.
\end{proof}

\begin{example}\label{ex: azads exampl, pt2}
	We continue with the example from Sec. \ref{Quantum groups uq(g,Lambda) and R-matrices}. The corresponding Yetter-Drinfeld module for the group $\widehat{G}/T$ is given by $V:= k\cdot F$, where $|F|=[\bar{\chi}]^{2}$. For concenience, we use the bracket notation $[\_]$ for the quotient map $\widehat{G} \to \widehat{G}/T$. Using the abelian $3$-cocycle $(\bar{\omega},\bar{\sigma})$ on $\widehat{G}/T$ as defined in Example \ref{ex:Azats exampl, pt1}, we obtain for the twisted coproduct on $B(V)\#k^{\widehat{G}/T}_{\bar{\omega}}$:
	\begin{align*}
	\Delta^J(F)&= \sum_{i,j = 0}^{\ell -1} \; q^{\frac{i(j-2-[j-2])}{2}}q^{-i} \frac{\zeta([\chi]^i,[\chi]^{[j-2]})}{\zeta([\chi]^i,[\chi]^j)} \delta_{[\chi]^i} \otimes F \# \delta_{[\chi]^j} + \sum_{i,j = 0}^{\ell -1} \; \frac{\zeta([\chi]^{[i-2]},[\chi]^j)}{\zeta([\chi]^i,[\chi]^j)} F\#\delta_{[\chi]^i} \otimes \delta_{[\chi]^j} \\
	&= K^{-1} \otimes F \cdot \sum_{i,j=0}^{\ell-1} q^{\frac{i(j-2-[j-2])}{2}} \frac{\zeta([\chi]^i,[\chi]^{[j-2]})}{\zeta([\chi]^i,[\chi]^j)} \delta_{[\chi]^i} \otimes \delta_{[\chi]^j} \\
	&+ F \otimes 1 \cdot \sum_{i,j = 0}^{\ell -1} \; \frac{\zeta([\chi]^{[i-2]},[\chi]^j)}{\zeta([\chi]^i,[\chi]^j)} \delta_{[\chi]^i} \otimes \delta_{[\chi]^j},
	\end{align*}	
	where $K^{-1}:=\sum_{i=0} q^{-i} \delta_{[\chi]^i}$ and $[k]:=k'$ so that $k =k'(mod \ell)$ and $0 \leq k \leq \ell-1$. If we now set $\zeta=\zeta_t^{-1}$ from Example \ref{ex:Azats exampl, pt1}\footnote{Again, we have to take the inverse because we want to obtain the same coproduct as in \cite{CreutzigGainutdinovRunkel2017} although we used a different convention for abelian $3$-cocycles.}, we obtain
	\begin{align*}
	\Delta^J(F)&= K^{-1}\left( q^t \sum_{2 \nmid i} \delta_{[\chi]^i} +\sum_{2 \mid i} \delta_{[\chi]^i}\right) \otimes F + F \otimes 1 \\
	\phi^J &= \sum_{i,j,k=0}^{\ell-1} \; q^{-tk\delta_{2\nmid i}\delta_{2 \nmid j}} \delta_{[\chi]^i} \otimes \delta_{[\chi]^j} \otimes \delta_{[\chi]^k}=\left( \sum_{2 \nmid i} \delta_{[\chi]^i}\right)  \otimes\left( \sum_{2 \nmid j} \delta_{[\chi]^j}\right) \otimes \left( \sum_{k=0}^{\ell -1} \;q^{-tk} \delta_{[\chi]^k}\right).
	\end{align*}	
\end{example}

\subsubsection{Dualization of $B(V)\# k^{\widehat{G}}_{\omega}$} \label{sec: dualization of radford biprod}

Since $B(V)\#k^{\widehat{G}}_\omega$ is a finite-dimensional quasi-Hopf algebra, the dual space $\left( B(V)\#k^{\widehat{G}}_\omega\right) ^*$ carries the structure of a coquasi-Hopf algebra, which is the dual analogue of a quasi-Hopf algebra. In particular, a coquasi-Hopf algebra is a coassociative coalgebra and has an algebra structure which is only associative up to an element $ \Psi \in (H\otimes H \otimes H)^*$. \\
\\
We define dual elements $\tilde{E}_i:=(F_i \# \delta_1)^*$ and $\tilde{K}_\psi:=\delta_\psi^*$ in $(B(V)\#k^{\widehat{G}}_\omega)^*$. We find the following relations for these elements:
\begin{itemize}
	\item $\Delta(\tilde{E}_i)=\tilde{K}_i^{-1} \otimes \tilde{E}_i + \tilde{E}_i \otimes 1$
	\item $\Delta(\tilde{K}_\psi) = \tilde{K}_\psi \otimes \tilde{K}_\psi$
	\item $\tilde{K}_\psi  \tilde{E}_i = \sigma(\bar{\chi}_i,\psi) (F_i \# \delta_\psi)^* = \sigma(\bar{\chi}_i,\psi) \tilde{E}_i \tilde{K}_\psi$.
\end{itemize}

From the last relation we see immediately that the product cannot be associative, since $\sigma$ is not a bihomomorphism on $\widehat{G}$. 

\subsection{A Drinfeld double}

From now on, we assume the Nichols algebra $B(V)$ and hence $B(V)\#k^{\widehat{G}}_{\omega}$ to be finite dimensional.

\subsubsection{The general case}
In the following, we recall the definition of the Drinfeld double $D(H)$ of a finite-dimensional quasi-Hopf algebra as introduced in \cite{HausserNill1999,HausserNill1999a}. As an algebra, Hausser and Nill defined the Drinfeld double as a special case of a diagonal crossed product $H^* \bowtie_{\delta} M$ (see \cite{HausserNill1999}, Def. 10.1) with a so-called two-sided coaction $(\delta,\Psi)$ (see \cite{HausserNill1999}, Def. 8.1) of the quasi-Hopf algebra $H$ on the algebra $M=H$ as input data. In this case, we have
\begin{align*}
\delta&= (\Delta \otimes \id) \circ \Delta:H \to H^{\otimes 3} \\
\Psi&= \left((\id \otimes \Delta \otimes \id)(\phi) \otimes 1 \right) (\phi \otimes 1 \otimes 1) 
\left( (\delta \otimes \id \otimes \id) (\phi^{-1}) \right).
\end{align*}
The multiplication on $D(H)$ is given by
\begin{align}\label{eq: product formula}
(\varphi \bowtie m)(\psi \bowtie n) := \left( (\Omega^1 \rightharpoonup \varphi \leftharpoonup \Omega^5)(\Omega^2m_{(1)(1)} \rightharpoonup \psi \leftharpoonup S^{-1}(m_{(2)})\Omega^4) \right) \bowtie \Omega^3 m_{(1)(2)}n,
\end{align}
where
\begin{align*}
	\Omega=(\id \otimes \id \otimes S^{-1} \otimes S^{-1})((1\otimes 1\otimes 1 \otimes f)\Psi^{-1}) \in H^{\otimes 5}.
\end{align*}
The Drinfeld twist $f \in H \otimes H$ is defined in Eq. \ref{def: f's}.
We summarize some of the main results in \cite{HausserNill1999,HausserNill1999a} in the following theorem, even though we state them in a more explicit way:
\begin{theorem}[Hausser, Nill]\label{thm: maps iota and Gamma}
	We have an algebra inclusion $\iota:H \to D(H)$ and a linear map $\Gamma :H^* \longrightarrow D(H)$ given by
	\begin{align*}
	\iota :H &\longrightarrow D(H) & \Gamma :H^* &\longrightarrow D(H)\\
	h &\longmapsto (1 \bowtie h) & \varphi &\longmapsto \left( p^1_{(1)} \rightharpoonup \varphi \leftharpoonup S^{-1}(p^2) \right) \bowtie p^1_{(2)}, 
	\end{align*}
	s.t. the algebra $D(H)$ is generated by the images $\iota(H)$ and $\Gamma(H^*)$. In particular, we have
	\begin{align*}
		\varphi \bowtie h = \iota(q^1)\Gamma(\varphi \leftharpoonup q^2) \iota(h).
	\end{align*} 
	The elements $p_R=p^1 \otimes p^2, q_R=q^1 \otimes q_2 \in H\otimes H$ were defined in Eq. \ref{def: p's}.
\end{theorem}
Without going into detail here, in \cite{HausserNill1999}, Ch. 11, Hausser and Nill showed that it is possible to define coproduct and antipode on the diagonal crossed product $D(H)$, such that it becomes a quasi-Hopf algebra.
The above theorem implies that it is sufficient to define the coproduct on $D(H)$ on elements $\iota(h)$ and $\Gamma(\varphi)$. It is given by
\begin{align*}
\Delta(\iota(h)) :&= (\iota \otimes \iota)(\Delta(h)) \\
\Delta(\Gamma(\varphi)) :&=\left( \varphi \otimes \id_{D(H)\otimes D(H)}\right)\left(\phi^{-1}_{312} \Gamma_{13} \phi_{213} \Gamma_{12} \phi^{-1} \right), \label{eq: coprod on Gamma(varphi)}
\end{align*}
where the inclusions in Eq. \ref{eq: coprod on Gamma(varphi)} are understood. The element $\Gamma \in H \otimes D(H)$ is defined by $\Gamma:= e_i \otimes \Gamma(e^i)$, where $e_i$ and $e^i$ are a dual pair of bases on $H$ and $H^*$.

\begin{remark} \label{rm: definition of coproduct different from [HN]}
	Note, that the definition of $\Delta(\Gamma(\varphi))$ in  Eq. \ref{eq: coprod on Gamma(varphi)} is different from the one in \cite{HausserNill1999a}, Thm. 3.9. Using their terminology, we checked both coherency and normality for the resulting $\lambda \rho$-intertwiner 
	\begin{align*}
	T:=\phi^{-1}_{312} \Gamma_{13} \phi_{213} \Gamma_{12} \phi^{-1},
	\end{align*}
	but we don't know if the definition in Eq. (11.4) in \cite{HausserNill1999} is false or just another possibility. We needed the one in Eq. \ref{eq: coprod on Gamma(varphi)} in order to make the next Proposition work.
\end{remark}
The antipode on $D(H)$ is defined by
\begin{align}
S(\iota(h)):&= \iota(S(h)) \\ \label{eq: antipode on Gamma}
S(\Gamma(\varphi)):&=(1 \bowtie f^1)\left(p^1_{(1)}g^{(-1)} \rightharpoonup \varphi \circ S^{-1} \leftharpoonup f^2 S^{-1}(p^2) \right) \bowtie p^1_{(2)}g^{(-2)}. 
\end{align}
As in Eq. \ref{eq: coprod on Gamma(varphi)}, we gave the action of $S$ on $\Gamma(\varphi)$ explicitly instead of defining it in terms of generating matrices as in \cite{HausserNill1999a}, Thm. 3.9.\\
The associator on $D(H)$ as well as the elements $\alpha$ and $\beta$ are simply inherited from $H$ by the inclusion $\iota:H \to D(H)$, which becomes then an inclusion of quasi-Hopf algebras.\\
Unit and counit are given by:
\begin{align*}
	1_{D(H)}:=1_{H^*} \bowtie 1_H, \qquad \epsilon_{D(H)}(\iota(h)):=\epsilon_H(h), \qquad \epsilon_{D(H)}\left( \Gamma(\varphi) \right) := \varphi(1_H). 
\end{align*}
Finally, we recall that a two-sided coaction $(\delta, \Psi)$ of $H$ on an algebra $M$ can be twisted by an element $U\in H \otimes M \otimes H$, giving rise to a twist-equivalent two-sided coaction $(\delta',\Psi')$ on $M$ (see \cite{HausserNill1999}, Dfn. 8.3):
\begin{align}\label{eq: twisted two-sided coaction}
\begin{split}
\delta'(h) :&= U\delta(h)U^{-1} \\
\Psi':&=(1 \otimes U \otimes 1)(\id \otimes \delta \otimes \id)(U)\Psi(\Delta \otimes \id_M \otimes \Delta)(U^{-1})
\end{split}
\end{align}  
In \cite{HausserNill1999} Prop. 10.6.1., Hausser and Nill show that twist equivalent two-sided coactions give rise to equivalent diagonal crossed products $H^* \bowtie_{\delta} M$ and $H^* \bowtie_{\delta'} M$.\\
\\
On the other hand, for any twist $J \in H \otimes H$, the pair $(\delta,\Psi^J)$ is a two-sided coaction of the twisted quasi-Hopf algebra $H^J$ (see Def. \ref{def: twist}) on $M$, where 
\begin{align}\label{eq: twisted two-sided coaction pt2}
\Psi^J:=\Psi(J^{-1} \otimes J^{-1}).
\end{align}

Again, in \cite{HausserNill1999} Prop. 10.6.2., Hausser and Nill show that for two-sided coactions $(\delta,\Psi)$ of $H$ and $(\delta,\Psi^J)$ of $H^J$, we get $H^* \bowtie M = (H^J)^* \bowtie M$ with trivial identification.

\begin{proposition}\label{prop: isom between twists of drinfeld doubles}
	Let $J\in H\otimes H$ be a twist on $H$ and $\tilde{J}:= (\iota \otimes \iota)(J) \in D(H)\otimes D(H)$ the corresponding twist on $D(H)$.
	For $U_L= (J \otimes 1)(\Delta \otimes \id)(J) \in H\otimes H \otimes H$, the following map is an isomorphism of quasi-Hopf algebras:
	\begin{align*}
	F_J: D(H)^{\tilde{J}} & \longrightarrow D(H^J) \\
	\varphi \bowtie a& \longmapsto \left( U_L^1 \rightharpoonup \varphi \leftharpoonup S^{-1}\left( U_L^3\right) \right)  \bowtie U_L^2a
	\end{align*}	
\end{proposition}

\begin{proof}
	As in Eq. \ref{eq: twisted two-sided coaction}, the element $U_L \in H^{\otimes 3}$ defines a twisted two-sided coaction $(\delta',\Psi')$ of $H$ on $H$. On the other hand, we can twist $(\delta',\Psi')$ in the sense of Eq. \ref{eq: twisted two-sided coaction pt2} via the twist $J \in H \otimes H$, giving rise to a two-sided coaction $(\delta',\Psi')$ of $H^J$ on the algebra $H$. This defines a diagonal crossed product $(H^J) \bowtie_{\delta'} H$, serving as the underlying algebra of $D(H^J)$. The fact that the map $F_J$ is an algebra isomorphism follows then simply from \cite{HausserNill1999} Prop. 10.6. Note that Hausser and Nill showed this for crossed products of the form $M \bowtie H^*$, but this is no problem due to Thm. 10.2 in \cite{HausserNill1999} relating $M \bowtie H^*$ and $H^* \bowtie M$. Using their terminology, we simply choose the left $\delta$-implementer $\tilde{L}=L' \prec U_L$ instead of the right $\delta$-implementer $\tilde{R}=U_L^{-1} \succ R'$ in \cite{HausserNill1999}, Eq. (10.46).\\
	Next, we want to show $(F_J \otimes F_J) \circ \Delta = \Delta \circ F_J$. By Thm. \ref{thm: maps iota and Gamma} it remains to prove this for elements of the form $\iota(h)$ and $\Gamma(\varphi)$. It is easy to see that $F_J$ is the identity on $H$, thus
	\begin{align*} 
	(F_J \otimes F_J)\circ \Delta_{D(H)^{\tilde{J}}}(\iota(h)) &= (F_J \otimes F_J)(\tilde{J}\Delta(\iota(h))\tilde{J}^{-1})\\
	&= (F_J \otimes F_J)(\tilde{J}(\iota(h_{(1)})\otimes \iota(h_{(2)})\tilde{J}^{-1})\\
	&= (F_J \otimes F_J)(\iota(J^1h_{(1)}K^{(-1)}) \otimes \iota(J^2h_{(2)}K^{(-2)}))\\
	&= \iota(J^1h_{(1)}K^{(-1)}) \otimes \iota(J^2h_{(2)}K^{(-2)}) = \Delta_{D(H^J)}\circ F_J(\iota(h)).		
	\end{align*}
	In order to show $(F_J \otimes F_J) \circ \Delta (\Gamma(\varphi)) = \Delta \circ F_J(\Gamma(\varphi))$, we use the identities 
	\begin{align*}
	\Gamma_J:=\Gamma_{D(H^J)}(\varphi)&=(1 \bowtie J^1)F_J(\Gamma(K^{(-1)}\rightharpoonup \varphi \leftharpoonup J^2))(1 \bowtie K^{(-2)}) \\
	F_J \circ \Gamma(\varphi) &= (1 \bowtie K^{(-1)}) \Gamma_J(J^1 \rightharpoonup \varphi \leftharpoonup K^{(-2)})(1 \bowtie J^2).		
	\end{align*}
	Then, 
	\begin{align*}
	\Delta_{D(H^J)}\circ F_J(\Gamma(\varphi)) &= \Delta_{D(H^J)}((1 \bowtie K^{(-1)}) \Gamma_J(J^1 \rightharpoonup \varphi \leftharpoonup K^{(-2)})(1 \bowtie J^2))  \\
	&= (\iota \otimes \iota)(\Delta(K^{(-1)}))\Delta_{D(H)^J}(\Gamma_J(J^1 \rightharpoonup \varphi \leftharpoonup K^{(-2)}))(\iota \otimes \iota)(\Delta(J^2)) =(\star) \\
	\end{align*}
	Using the definition in Eq. \ref{eq: coprod on Gamma(varphi)} we can show
	\begin{align*}
	\Delta_{D(H)}(\Gamma(\varphi)) &= (1 \bowtie x^1X^1)\Gamma(y^1 \rightharpoonup \varphi_{(2)})(1 \bowtie y^2) \\
	&\otimes (1 \bowtie x^2)\Gamma(X^2 \rightharpoonup \varphi_{(1)} \leftharpoonup x^3)(1 \bowtie X^3y^3).
	\end{align*}
	Hence,
	\begin{align*}
	(F_J \otimes F_J)\circ \Delta_{D(H)^{\tilde{J}}} (\Gamma(\varphi)) &= (1 \bowtie J^1 x^1X^1) F_J\Gamma(y^1 \rightharpoonup \varphi_{(2)})(1 \bowtie y^2 K^{(-1)}) \\
	& \otimes (1 \bowtie J^2x^2)F_J \Gamma (X^2 \rightharpoonup \varphi_{(1)} \leftharpoonup x^3)(1 \bowtie X^3 y^3 K^{(-2)})
	\end{align*}
	Using the identity 
	\begin{align*}
	a_{(1)}F_J \Gamma(b_{(1)} \rightharpoonup \varphi a_{(2)})b_{(2)}= F_J\Gamma((ab)_{(1)} \rightharpoonup \varphi)(ab)_{(2)},
	\end{align*}
	a simple but tedious calculation shows that this equals $(\star)$.\\
	For the antipode, it is again sufficient to show $S \circ F_J = F_J \circ S$ on generators $\iota(h)$ and $\Gamma(\varphi)$, where the former is trivial by the same argument as above. Using the definition in Eq. \ref{eq: antipode on Gamma}, we obtain
	\begin{align}
	F_J \circ S(\Gamma(\varphi))=(1 \bowtie f^1)\left( (U^1p^1_{(1)}g^{(-1)} \rightharpoonup \varphi \circ S^{-1} \leftharpoonup f^2 S^{-1}(U^3p^2)) \bowtie U^2p^1_{(2)}g^{(-2)}\right) 
	\end{align}  
	and
	\begin{align*}
	S \circ F_J(\Gamma(\varphi)) &= (1 \bowtie S(J^2)f^1_J) \\
	& \times \left( \left(  {p^1_J}_{(1)(1)}g_J^{(-1)}S(K^{(-2)}) \rightharpoonup \varphi \circ S^{-1} \leftharpoonup S(J^1)f_J^2S^{-1}(p_J^2) \right)  \bowtie {p^1_J}_{(1)(2)}g_J^{(-2)}S(K^{(-1)})\right) .
	\end{align*}
	The Drinfeld twist $f=g \in H \otimes H$ is defined in \ref{def: f's}. Using the identities in \ref{eq: identities for twisted drinfeld twists, etc}, a simple but tedious calculation shows that both terms are equal.\\
	Finally, since $F_J$ is the identity on $H \subseteq D(H)$ it is easy to see that $F_J$ preserves the associator and the elements $\alpha$ and $\beta$, which are inherited from $H$. This proves the proposition.  
\end{proof}

The following proposition shows that retractions of quasi-Hopf algebras induce monomorphisms between the Drinfeld doubles:

\begin{proposition}\label{prop: isos of QHA's induce isos of Drinfeld doubles}
	Let $i:K \to H$ be a split monomorphism of quasi-Hopf algebras, i.e. there exists a homomorphism of quasi-Hopf algebras $p:H\to K$, s.t. $p\circ i=\id_K$. Then the map $\Phi_{i}:D(K) \to D(H)$ defined by
	\begin{align*}
		\Phi_{i}(\iota_K(h)) &:= \iota_{H}(i(h)) \\
		\Phi_{i}(\Gamma_K(\varphi)) &:= \Gamma_{H}(\varphi \circ p)
	\end{align*}
	is a monomorphism of quasi-Hopf algebras with left inverse $\Phi_p:D(H) \to D(K)$ being a coalgebra homomorphism defined by
	\begin{align*}
		\Phi_{p}(\iota_H(h)) &:= \iota_{K}(p(h)) \\
		\Phi_{p}(\Gamma_H(\varphi)) &:= \Gamma_{K}(\varphi \circ i).
	\end{align*}
\end{proposition}

\begin{proof}
	The fact that $\Phi_{i}$ is an algebra isomorphism follows from the universal property of the map $\Gamma$ (see Thm. II in \cite{HausserNill1999}). In their terminology, we choose a normal element $T:=e_i \otimes \Gamma_{H}(e^i \circ p) \in K \otimes D(H)$. Using the identities given in Eq. \ref{eq:p,q identities} and $p\circ i =\id$, it is a straight-forward calculation to check that this element satifies both conditions (6.2) and (6.3) in the first part of the theorem. \\
	Since $i$ is a homomorphism of quasi-Hopf algebras, $\Phi_{i}$ preserves coproduct and antipode on elements of the form $\iota_K(h)$ as well as the associator $\phi$ and the elements $\alpha$ and $\beta$. For the image of $\Gamma_K$ we have
	\begin{align*}
		(\Phi_{i} \otimes \Phi_{i}) \circ \Delta_{D(K)} (\Gamma_K(\varphi)) &= (\Phi_{i} \otimes \Phi_{i})((\varphi \otimes \id)(T_K)) \\
		&=(\Phi_{i} \otimes \Phi_{i})(x^1X^1 \Gamma_K(y^1 \rightharpoonup \varphi_{(2)})y^2\otimes x^2 \Gamma_K(X^2 \rightharpoonup \varphi_{(1)} \leftharpoonup x^3)X^3y^3) \\
		&=i(x^1X^1) \Gamma_{H}((y^1 \rightharpoonup \varphi_{(2)})\circ p)i(y^2) \\
		&\otimes i(x^2) \Gamma_{H}((X^2 \rightharpoonup \varphi_{(1)} \leftharpoonup x^3) \circ p )i(X^3y^3)) \\
		&=\tilde{x}^1\tilde{X}^1 \Gamma_{H}(\tilde{y}^1 \rightharpoonup (\varphi \circ p)_{(2)}) \tilde{y}^2 \otimes \tilde{x}^2 \Gamma_{H}(\tilde{X}^2 \rightharpoonup (\varphi \circ p)_{(1)} \leftharpoonup \tilde{x}^3) \tilde{X}^3\tilde{y}^3 \\
		&= (\varphi \circ p \otimes \id)(T_{H}) \\
		&= \Delta_{D(H)} (\Gamma_{H}(\varphi \circ p)) = \Delta_{D(H)}(\Phi_{i}(\Gamma_K(\varphi))),
	\end{align*} 
	where $T$ is the element defined in Remark \ref{rm: definition of coproduct different from [HN]} and $\phi_H=\tilde{X}_1\otimes \tilde{X}_2\otimes \tilde{X}_3=\tilde{Y}_1\otimes \tilde{Y}_2\otimes \tilde{Y}_3=i\otimes i\otimes i(\phi_K)$ and $\Phi_H^{-1}=\tilde{x}_1\otimes \tilde{x}_2\otimes \tilde{x}_3$. We omit the proof that $\Phi_{i}$ preserves the antipode on the image of $\Gamma$, since it is completely analogous.\\
	Using $p\circ i=\id_K$, it is easy to see that $\Phi_p$ is a left inverse of $\Phi_i$. By analogous arguments as above, $\Phi_p$ is a coalgebra homomorphism.
\end{proof}

\subsubsection{The Drinfeld double of $B(V) \# k^G_\omega$}\label{sec: Drinfeld double}

In order to state the defining relations for the Drinfeld double of $B(V) \# k^{\widehat{G}}_\omega$, we compute the element $\Omega \in (B(V) \# k^{\widehat{G}}_\omega)^{\otimes 5}$ from the previous subsection for this case:

\begin{lemma}
	In the case $H=B(V) \# k^{\widehat{G}}_\omega$, the element $\Omega \in (B(V) \# k^{\widehat{G}}_\omega)^{\otimes 5}$ is given by
	\begin{align*}
	\Omega = \sum_{\chi_i} \, f(\chi_1,\chi_2,\chi_3,\chi_4,\chi_5) \, &(1 \# \delta_{\chi_1}) \otimes (1 \# \delta_{\chi_2}) \otimes (1 \# \delta_{\chi_3}) \otimes (1 \# \delta_{\chi_4}) \otimes (1 \# \delta_{\chi_5}),\text{ where} \\ 
	f(\chi_1,\chi_2,\chi_3,\chi_4,\chi_5) &= \frac{\omega(\chi_1\chi_2\chi_3,\bar{\chi}_4,\bar{\chi}_5)\omega(\chi_5,\chi_4,\bar{\chi}_4)}{\omega(\chi_1,\chi_2\chi_3,\bar{\chi}_4)\omega(\chi_1,\chi_2,\chi_3)\omega(\chi_4\chi_5,\bar{\chi}_4,\bar{\chi}_5)}. 
	\end{align*}
\end{lemma}
In particular, we obtain $f(\psi_1,\psi_2,\chi,\psi_2,\psi_1)^{-1}=\theta(\chi)(\psi_1,\psi_2) d\nu(\psi_1,\psi_2)(\chi)$, where $\theta \in Z^2(\widehat{G},k^{\widehat{G}})$ is the $2$-cocycle from Remark \ref{rm: dual of YDM} and $\nu(\psi)(\chi):=\omega(\psi,\bar{\chi},\chi)$. From this, it can be seen that $D(k^{\widehat{G}}_\omega)$ is indeed isomorphic to the double $D^\omega(\widehat{G})$ in the sense of \cite{DijkgraafPasquierRoche1992}.\\

Before we are going to derive the braided commutator relations, we define:
\begin{align*}
	E_i:&=\Gamma\left((F_i \#\delta_1)^* \right), \qquad F_j:=\iota(F_j),\\
	\bar{L}_\chi:=\iota(L_{\chi}),& \qquad L_\chi:=\iota(\bar{L}_{\chi}), \qquad \bar{K}_{\chi}:=\Gamma((1\#\delta_{\bar{\chi}})^*)\bowtie \left( \sum_{\psi}\, \theta(\psi|\chi,\bar{\chi}) \, \iota(\delta_\psi)\right) .
\end{align*}

\begin{lemma}
	Let $a_i,b_j : \widehat{G} \to k^\times$ be solutions to the equation
	\begin{align}\label{eq: equation for the drinfeld double}
		\frac{a_i(\psi\bar{\chi}_j)b_j(\psi)}{a_i(\psi)b_j(\psi\chi_i)}=\frac{\omega(\chi_i,\bar{\chi}_j,\psi)}{\omega(\bar{\chi}_j,\chi_i,\psi)}.
	\end{align}			
	For $\chi,\psi \in \widehat{G}$, we set
	\begin{align*}
		E_i^\chi := c_\chi E_iL_{\bar{\chi}_i}\iota\left( \sum_{\xi}\,a_i(\xi)\, (1\# \delta_\xi) \right)  \qquad
		F_j^\psi := c_\psi F_j\iota\left( \sum_{\xi}\,b_j(\xi)\, (1\# \delta_\xi) \right)  .
	\end{align*}
	Here, the elements $c_\chi \in \left( k^{\widehat{G}}\right)^*\otimes k^{\widehat{G}} $ are defined by
	\begin{align*}
		c_{\chi}:= \bar{K}_{\bar{\chi}}\bar{L}_{\chi}^{-1}=(1\# \delta_\chi)^*\bowtie \left( \sum_{\psi \in \widehat{G}} \, \frac{\sigma(\psi,\chi)}{\omega(\psi,\chi,\bar{\chi})}\, (1\# \delta_\psi) \right).
	\end{align*}
	Let $[ E_i^\chi, F_j^\psi]_\sigma:=E_i^\chi F_j^\psi - \sigma(\chi_i,\bar{\chi}_j)F^\psi_jE^\chi_i$ denote the braided commutator in $D\left( B(V) \# k^{\widehat{G}}_\omega\right) $. Then we have
	\begin{align}
		[ E_i^\chi, F_j^\psi]_\sigma &= \delta_{ij}\sigma(\chi_i,\bar{\chi}_i) c_{\chi\psi}
		\left(1-\bar{K}_{\chi_i} L_{\bar{\chi}_i}
		   \right)\iota \left( \sum_{\xi}\, \frac{a_i(\xi)b_i(\xi\chi_i)}{\omega(\bar{\chi}_i,\chi_i,\xi)} (1\# \delta_\xi) \right)
	\end{align}
\end{lemma}
\begin{proof}
	For $\lambda_i,\mu_j:\widehat{G}\times \widehat{G} \to k$ arbitrary non-zero maps, we define elements
	\begin{align*}
		E_i^{\lambda}:&= \sum_{\chi,\psi} \, \lambda_i(\chi,\psi) (F_i \# \delta_\chi)^* \bowtie (1\#\delta_\psi) \\
		F_j^{\mu}:&= \sum_{\chi,\psi} \, \mu_j(\chi,\psi) (1 \# \delta_\chi)^* \bowtie (F_j\#\delta_\psi).
	\end{align*}
	Using the general multiplication formula \ref{eq: product formula} for the Drinfeld double, we obtain
	\begin{align*}
		E_i^\lambda F_j^\mu &= \sum_{\chi,\psi} \left( \sum_{\xi}\, r_{\chi,\psi}(\xi) \mu_j(\chi \bar{\xi},\psi) \lambda_i(\xi,\psi \bar{\chi}_j) \right) \,(F_i \# \delta_\chi)^* \bowtie (F_j \# \delta_\psi), \text{ where}\\
		r_{\chi,\psi}(\xi)&=  \frac{f(\xi, \chi\bar{\xi},\psi\bar{\chi}_j,\chi\bar{\xi},\xi \bar{\chi}_i)}{\omega(\bar{\chi}_i,\xi,\chi\bar{\xi})}  
	\end{align*}
	For $F_j^\mu E_i^\lambda$, we obtain
	\begin{align*}
		F_j^\mu E_i^\lambda &= \delta_{ij} 
		\left( \sum_{\chi,\psi} \lambda_i(\chi,\psi) \left(  \sum_{\xi} \, \alpha_1^i(\chi,\psi,\chi_i\bar{\chi}) f(\xi,\chi\bar{\chi}_i,\psi,\chi\bar{\chi}_i,\xi) \mu_j(\xi,\psi\chi_i)  \, (1\# \delta_{\chi\xi\bar{\chi}_i})^*  \right)  \bowtie (1 \# \delta_\psi) \right.\\
		&-\left. \sum_{\chi,\psi} \lambda_i(\chi,\psi) \left(  \sum_{\xi} \, \frac{\alpha_3^i(\chi,\psi,\chi_i\bar{\chi}) f(\xi,\chi,\psi,\chi,\xi) }{\omega(\chi,\bar{\chi}_i,\chi_i\bar{\chi})\sigma(\bar{\chi}_i,\chi)}\mu_j(\xi,\psi\chi_i)   \, (1\# \delta_{\chi\xi})^*  \right) \bowtie (1 \# \delta_\psi)  \right)\\
		&+ \sum_{\chi,\psi} \left( \sum_{\xi} \, s_{\chi,\psi}(\xi) \mu_j(\chi \bar{\xi},\psi\chi_i)\lambda_i(\xi,\psi)   \right)   \, (F_i \# \delta_{\chi})^* \bowtie (F_j \# \delta_{\psi}),\text{ where} \\
		s_{\chi,\psi}(\xi)&= r_{\chi,\psi}(\xi) \alpha_2^j(\xi,\psi,\chi_i\bar{\xi}) \frac{\sigma(\chi_i,\xi\bar{\chi})\omega(\chi \psi \bar{\xi} \bar{\chi}_j,\chi_i,\xi \bar{\chi})}{\omega(\chi \psi \bar{\xi} \bar{\chi}_j,\xi \bar{\chi},\chi_i)\omega(\chi \bar{\xi},\psi \bar{\chi}_j,\chi_i)}.
	\end{align*}
	The functions $\alpha^i_k$ were defined after Prop. \ref{prop: RelationsRadfordBiproduct}. For the following choices of $\mu_j$ and $\lambda_i$, the above formulas simply significantly:
	\begin{align*}
		\lambda_i(\chi,\psi) = \delta_{\chi,\xi}a_i^\xi(\psi) \qquad \qquad
		\mu_j(\chi,\psi) = \delta_{\chi,\xi}b_j^\xi(\psi),
	\end{align*}
	where $a_i^\xi,b_j^\xi:\widehat{G} \to k^\times$.	Or goal is to find functions $\lambda_i,\mu_j:\widehat{G} \times \widehat{G} \to k$, s.t. that the braided commutator $[E_i^\lambda,F_j^\mu]_\sigma$ takes values in $\left( k^{\widehat{G}}\right)^* \bowtie k^{\widehat{G}} $. The necessary and sufficient condition for this is
	\begin{align*}
		0 &=\sum_{\xi} \, r_{\chi,\psi}(\xi) \\
		&\times \left( \mu_j(\chi \bar{\xi},\psi) \lambda_i(\xi,\psi \bar{\chi}_j)
		-\sigma(\chi_i,\bar{\chi}_j) \mu_j(\chi \bar{\xi},\psi\chi_i)\lambda_i(\xi,\psi)  \frac{\alpha_2^j(\xi,\psi,\chi_i\bar{\xi})\sigma(\chi_i,\xi\bar{\chi})\omega(\chi \psi \bar{\xi} \bar{\chi}_j,\chi_i,\xi \bar{\chi})}{\omega(\chi \psi \bar{\xi} \bar{\chi}_j,\xi \bar{\chi},\chi_i)\omega(\chi \bar{\xi},\psi \bar{\chi}_j,\chi_i)}\right),
	\end{align*}
	At least for the above choices of $\lambda_i$ and $\mu_j$, this is equivalent to the existence of solutions $a_i,b_j:\widehat{G} \to k^\times$ to the equation:
	\begin{align*}
		\frac{a_i(\psi\bar{\chi}_j)b_j(\psi)}{a_i(\psi)b_j(\psi\chi_i)}=\frac{\omega(\chi_i,\bar{\chi}_j,\psi)}{\omega(\bar{\chi}_j,\chi_i,\psi)}.
	\end{align*}
	This can be seen by setting
	\begin{align*}
		a_i^\xi(\psi)=\frac{\sigma(\psi,\xi)\sigma(\chi_i,\psi)}{\omega(\psi,\xi,\chi_i\bar{\xi})}a_i(\psi) \qquad
		b_j^\xi(\psi)= \frac{\omega(\bar{\xi},\xi,\psi\bar{\chi}_j)}{\sigma(\psi\bar{\chi}_j\xi,\bar{\xi})}b_j(\psi)
	\end{align*}
	and using the abelian $3$-cocycle conditions. Plugging in these solutions for $\lambda_i$ and $\mu_j$ in the braided commutator and using that $c_{\chi\psi}=c_\chi c_\psi$ yields the claimed result.
\end{proof}

\begin{corollary}\label{cor: commutation rel for twisted generator}
	Let $\omega \in Z^3(\widehat{G})$ be a nice $3$-cocycle. Then, the braided commutator in $D\left( B(V)\#k^{\widehat{G}}_{\omega d\zeta} \right)$ is given by:
	\begin{align}
		[ E_i^\chi, F_j^\psi]_\sigma &= \delta_{ij}\sigma(\chi_i,\bar{\chi}_i) c_{\chi\psi}
		\left(1-\bar{K}_{\chi_i} L_{\bar{\chi}_i}
		\right).
	\end{align}
\end{corollary}

\begin{proof}
	In case of a nice $3$-cocycle, we can set $a_i,b_j=1$. Together with the second niceness condition this implies
	\begin{align*}
		\iota \left( \sum_{\xi}\, \frac{a_i(\xi)b_i(\xi\chi_i)}{\omega(\bar{\chi}_i,\chi_i,\xi)} (1\# \delta_\xi) \right)=1.
	\end{align*}
\end{proof}

The remaining relations simplify drastically after taking the quotient in Sec. \ref{sec: a quotient}, so we wait until then before we state them.

\subsection{A quotient of the Drinfeld double}\label{sec: a quotient}

In this subsection we define the small quasi-quantum group as a quotient of the Drinfeld double $D:=D\left( B(V)\#k^{\widehat{G}}_\omega\right)$ (see Section \ref{sec: Drinfeld double}) by the biideal $I \subseteq D$ induced by the following map: 

\begin{proposition}
	Let $D:=D\left( B(V)\#k^{\widehat{G}}_\omega\right)$ be the Drinfeld double from Section \ref{sec: Drinfeld double}. Moreover, let $k\widehat{G}$ be the group algebra of $\widehat{G}$. 
	\begin{enumerate}
		\item The following map is an algebra inclusion into the center of $D$:
		\begin{align*}
		j: k\widehat{G} &\longrightarrow Z(D) \\
		\chi &\longmapsto c_{\chi}=(1\# \delta_\chi)^*\bowtie \left( \sum_{\psi \in \widehat{G}} \, \frac{\sigma(\psi,\chi)}{\omega(\psi,\chi,\bar{\chi})}\, (1\# \delta_\psi) \right).
		\end{align*}
		Moreover, we have $\Delta\circ j=j\otimes j \circ \Delta$ and $\epsilon = \epsilon\circ j$.
		\item Set $N^+:=ker\left( \epsilon|_{j\left( k\widehat{G}\right) } \right)$ and 
		\begin{align*}
			I:=N^+D.
		\end{align*}
		Then, $I \subseteq D$ is a biideal, i.e. a two-sided ideal and coideal.
		\item The quotient $D/I$ is a quasi-Hopf algebra with quasi-Hopf structure induced by the quotient map $D \twoheadrightarrow D/I$.
	\end{enumerate}
\end{proposition}

\begin{proof}
	We first show that $j$ is an algebra homomorphism. We have
	\begin{align*}
		c_{\chi \psi} &= (1\#\delta_{\chi\psi})^*\bowtie \left( \sum_{\xi \in \widehat{G}}\, \frac{\sigma(\xi,\psi\chi)}{\omega(\xi,\chi\psi,\bar{\chi}\bar{\psi})}\, (1\# \delta_\xi) \right) \\
		&=(1\#\delta_{\chi})^*(1\#\delta_{\psi})^*\bowtie\left( \sum_{\xi \in \widehat{G}}\, \theta(\xi|\chi,\psi)^{-1} \frac{\sigma(\xi,\psi)\sigma(\xi,\chi)}{\omega(\xi,\chi\psi,\bar{\chi}\bar{\psi})}\, (1\# \delta_\xi) \right)	\\
		&=\sum_{\xi \in \widehat{G}}\, \theta(\xi|\chi,\psi)^{-1} \frac{\sigma(\xi,\psi)\sigma(\xi,\chi)}{\omega(\xi,\chi\psi,\bar{\chi}\bar{\psi})}\, (1\#\delta_{\chi})^*(1\#\delta_{\psi})^*\bowtie(1\# \delta_\xi) \\
		&=\sum_{\xi,\nu \in \widehat{G}}\,  \frac{\sigma(\xi,\psi)\sigma(\xi,\chi)}{\omega(\xi,\chi,\bar{\chi})\omega(\xi,\psi,\bar{\psi})}\, ((1\#\delta_{\chi})^*\bowtie(1\# \delta_\xi))((1\#\delta_{\psi})^*\bowtie(1\# \delta_\nu))  \\
		&=c_\chi c_\chi.
	\end{align*}
	
	Obviously, $j$ preserves the unit. We continue with the coproduct:
	\begin{align*}
		\Delta(c_\chi)&= \Delta(\Gamma(1\#\delta_{\chi})^*))\left(\sum_{\psi \in \widehat{G}} \, \sigma(\psi,\chi)\, \Delta(1\# \delta_\psi) \right) \\
		&= \Gamma((1\#\delta_{\chi})^*)\otimes\Gamma((1\#\delta_{\chi})^*) \left(\sum_{\psi \in \widehat{G}} \, \theta(\chi|\psi,\chi)^{-1} \sigma(\psi\xi,\chi)\, (1\# \delta_\psi)\otimes (1\# \delta_\xi) \right)\\
		&=\Gamma((1\#\delta_{\chi})^*)\otimes\Gamma((1\#\delta_{\chi})^*) \left(\sum_{\psi \in \widehat{G}} \, \sigma(\psi,\chi)\sigma(\xi,\chi)\, (1\# \delta_\psi)\otimes (1\# \delta_\xi) \right)\\
		&=c_{\chi}\otimes c_{\chi}.
	\end{align*}
	The counit is again trivial. Next, we check that $j$ takes values in the center. It suffices to show that $c_\chi$ commutes with elements of the form $\Gamma((F_i\#\delta_{\chi})^*)$ and $\iota(F_j\#\delta_{\psi})$:
	\begin{align*}
		c_{\chi}\Gamma((F_i\#\delta_{\xi})^*)&=\Gamma((1\#\delta_{\chi})^*)\bowtie \left( \sum_{\psi \in \widehat{G}} \, \sigma(\psi,\chi)\, (1\# \delta_\psi)\Gamma((F_i\#\delta_{\xi})^*) \right)\\
		&=\Gamma((1\#\delta_{\chi})^*)\bowtie \left( \sum_{\psi \in \widehat{G}} \, \sigma(\psi,\chi)\, \Gamma((F_i\#\delta_{\xi})^*)(1\# \delta_{\bar{\chi}_i\psi}) \right)\\
		&=\Gamma((1\#\delta_{\chi})^*)\Gamma((F_i\#\delta_{\xi})^*)\bowtie \left( \sum_{\psi \in \widehat{G}} \, \sigma(\psi\chi_i,\chi)\, (1\# \delta_{\psi}) \right)\\
		&=\sigma(\chi_i,\chi)^{-1}\Gamma((F_i\#\delta_{\xi})^*)\Gamma((1\#\delta_{\chi})^*)\bowtie \left( \sum_{\psi \in \widehat{G}} \, \theta(\chi|\psi,\chi_i)\sigma(\psi\chi_i,\chi)\, (1\# \delta_{\psi}) \right)\\
		&=\Gamma((F_i\#\delta_{\xi})^*)\Gamma((1\#\delta_{\chi})^*)\bowtie \left( \sum_{\psi \in \widehat{G}} \, \sigma(\psi,\chi)\, (1\# \delta_{\psi}) \right)\\
		&=\Gamma((F_i\#\delta_{\xi})^*)c_{\chi}.
	\end{align*}
	On the other hand,
	\begin{align*}
		c_{\chi}(F_i\#\delta_\xi)&=\Gamma((1\#\delta_{\chi})^*)\bowtie \left( \sum_{\psi \in \widehat{G}} \, \sigma(\psi,\chi)\, (1\# \delta_\psi)(F_i\#\delta_{\xi}) \right) \\
		&=\Gamma((1\#\delta_{\chi})^*)\bowtie \left( \sum_{\psi \in \widehat{G}} \, \sigma(\psi,\chi)\, (F_i\#\delta_{\xi})(1\# \delta_{\psi\chi_i}) \right) \\
		&=\Gamma((1\#\delta_{\chi})^*)\bowtie(F_i\#\delta_{\xi}) \left( \sum_{\psi \in \widehat{G}} \, \sigma(\psi\bar{\chi}_i,\chi)\, (1\# \delta_{\psi}) \right) \\
		&=\sigma(\bar{\chi}_i,\chi)^{-1}(F_i\#\delta_{\xi})\Gamma((1\#\delta_{\chi})^*) \left( \sum_{\psi \in \widehat{G}} \, \theta(\chi|\psi,\bar{\chi}_i)^{-1}\sigma(\psi\bar{\chi}_i,\chi)\, (1\# \delta_{\psi}) \right) \\
		&=(F_i\#\delta_{\xi})\Gamma((1\#\delta_{\chi})^*) \left( \sum_{\psi \in \widehat{G}} \, \sigma(\psi,\chi)\, (1\# \delta_{\psi}) \right) \\
		&=(F_i\#\delta_\xi)c_{\chi}.
	\end{align*}	
	This proves the first part of the proposition. We now come to the second part. 
	Since $j(k\widehat{G}) \subseteq Z(D)$, we have $DN^+=N^+D$, hence $I\subseteq D$ is an both-sided ideal.
	As a kernel of a coalgebra homomorphism, $N^+ \subseteq D$ is a coideal and so is $I=HN^+$. The fact that we are dealing with non-coassociative coalgebras plays no role so far. We have shown that $I$ is a biideal.\\
	As it is stated in \cite{Schauenburg2005}, Section 2, this is equivalent to $D/I$ being a quotient quasi-bialgebra. Since $D$ is a finite dimensional quasi-Hopf algebra, we can apply Thm. 2.1 in \cite{Schauenburg2005} in order to prove that $D/I$ is a quasi-Hopf algebra.
\end{proof}

\begin{definition}
	Let $G$ be a finite abelian group, $(\omega,\sigma) \in Z^3_{ab}(\widehat{G})$ an abelian $3$-cocycle on the dual group $\widehat{G}$ and $\{\, \chi_i \in \widehat{G} \, \}_{1\leq i \leq n} \subseteq \widehat{G}$ a subset of $\widehat{G}$, s.t. the corresponding Nichols algebra $B(V)$ is finite dimensional. Then we define the small quasi-quantum group corresponding to that data to be the quotient
	\begin{align*}
		u(\omega,\sigma):=D\left( B(V)\#k^{\widehat{G}}_\omega\right)/I.
	\end{align*}
\end{definition}

It is clear that in the quotient $u(\omega,\sigma)$, we have $\bar{L}_{\chi}=\bar{K}_{\bar{\chi}}$. In order to get rid of $L$'s, we also define $K_\chi:= L_{\bar{\chi}}$. The relations for $u(\omega,\sigma)$ are collected in Thm. \ref{thm:quasi-Hopf algebra}.

\begin{example}
	We continue with Exp. \ref{ex: azads exampl, pt2}. As we have quasi-Hopf inclusions $u(\bar{\omega},\bar{\sigma})^{\leq} \subseteq u(\bar{\omega},\bar{\sigma})$ and $\left( u(\bar{\omega},\bar{\sigma})^{\leq}\right)^J  \subseteq u(\bar{\omega},\bar{\sigma})^J $, where the twist $J$ is defined in Exp. \ref{ex: azads exampl, pt2}, the coproducts of the generator $F$ in $u(\bar{\omega},\bar{\sigma})$ and $u(\bar{\omega},\bar{\sigma})^J$ agree with the ones in Exp. \ref{ex: azads exampl, pt2}.  Since $|F|=[\bar{\chi}]^2$, it is easy to see that $(\bar{\omega},\bar{\sigma})$ is a nice abelian $3$-cocycle in the sense of Def. \ref{def: nice 3-cocycle} and $Q_{\bar{|F|},|F|}=1=Q_{\bar{|F|},\bar{|F|}}$, where $Q_{\chi,\psi}=\sum_{\xi}\,\theta(\chi|\psi,\xi)\,\delta_\xi$. Moreover, we have $K_{\psi}=\bar{K}_{\psi}$ for all $\psi \in \widehat{G}/T$. Setting $K:=K_{\bar{|F|}}$, the algebra relations from Thm. \ref{thm:quasi-Hopf algebra} then simplify as follows:
	\begin{align*}
		K^\ell=1\, \qquad KE=q^2EK, \qquad KF=q^{-2}FK, \qquad EF-FE=K^{-1}-K.
	\end{align*}
	Moreover, from \ref{cor: rel's Nichols algebra} we obtain the relations:
	\begin{align*}
		F^{\ell/2}=0, \qquad E^{\ell/2}=0.
	\end{align*}
	These relations don't change if we go from $u(\bar{\omega},\bar{\sigma})$ to $ u(\bar{\omega},\bar{\sigma})^J$.
\end{example}

\begin{proposition}\label{prop: split monomorphism}
	The following map is a monomorphism of quasi-Hopf algebras, which is split as a coalgebra homomorphism:
	\begin{align*}
		\xi:k^{\widehat{G}}_\omega &\longrightarrow u(\omega,\sigma) \\
		\delta_\chi & \longmapsto \left[ \iota(1\#\delta_\chi) \right]. 
	\end{align*}
\end{proposition}

\begin{proof}
	It is clear that the inclusion $i:k^{\widehat{G}}_\omega \to B(V)\#k^{\widehat{G}}_\omega$, $\delta_\chi \mapsto 1\#\delta_\chi$ is a split monomorphism, with left inverse denoted by $p$. We saw in Prop. \ref{prop: isos of QHA's induce isos of Drinfeld doubles} that it must therefore induce a quasi-Hopf algebra monomorphism $\Phi_i:D\left(k^{\widehat{G}}_\omega \right)  \to D\left( B(V)\#k^{\widehat{G}}_\omega \right) $ with left inverse $\Phi_p$ being a coalgebra homomorphism. Since the inclusion $j:k\widehat{G} \to D\left( B(V)\#k^{\widehat{G}}_\omega \right)$ factors through $D\left(k^{\widehat{G}}_\omega \right)$, the following diagram commutes:
	\begin{equation*}
		\begin{tikzcd}
			D\left(k^{\widehat{G}}_\omega \right) \arrow[hookrightarrow]{r}{\Phi_i}\arrow[twoheadrightarrow]{d} &D\left( B(V)\#k^{\widehat{G}}_\omega \right) \arrow{r}{\Phi_p}\arrow[twoheadrightarrow]{d} &D\left(k^{\widehat{G}}_\omega\right) \arrow[twoheadrightarrow]{d}\\
			D\left(k^{\widehat{G}}_\omega \right)/\tilde{I} \arrow[hookrightarrow]{r}{[\Phi_i]}& D\left( B(V)\#k^{\widehat{G}}_\omega \right)/I \arrow{r}{[\Phi_p]}& D\left(k^{\widehat{G}}_\omega\right)/\tilde{I},
		\end{tikzcd}
	\end{equation*}
	where $\tilde{I}:=N^+D\left(k^{\widehat{G}}_\omega \right)$. If we show, that the map $f:k^{\widehat{G}}_\omega \to D\left(k^{\widehat{G}}_\omega \right)/\tilde{I}$ given by $\delta_\chi \mapsto [\iota(\delta_\chi)]$ is an isomorphism, we can define a left inverse of $\xi$ by $f^{-1}\circ[\Phi_p]$.\\
	It is clear that $f$ is surjective, since for an arbitrary element $\sum_{\chi,\psi}\, a(\chi,\psi) \Gamma((1\#\delta_{\chi})^*)\iota(\delta_\psi) \in D\left(k^{\widehat{G}}_\omega\right)$, we obtain
	\begin{align*}
		\left[ \sum_{\chi,\psi}\, a(\chi,\psi) \Gamma((1\#\delta_{\chi})^*)\iota(\delta_\psi)\right] &=
		\left[ \sum_{\chi,\psi}\, a(\chi,\psi)c_{\chi}\bar{K}_{\chi}^{-1}\iota(\delta_\psi)\right]  
		=\left[ \sum_{\chi,\psi}\, \frac{a(\chi,\psi)}{\sigma(\psi,\chi)}c_{\chi}\iota(\delta_\psi)\right]  \\
		&=\left[ \sum_{\chi,\psi}\, \frac{a(\chi,\psi)}{\sigma(\psi,\chi)}\iota(\delta_\psi)\right] 
		=f\left( \sum_{\chi,\psi}\, \frac{a(\chi,\psi)}{\sigma(\psi,\chi)}\delta_\psi\right) ,
	\end{align*}
	where we used that $[c_\chi]=[1]$ holds in the quotient. Moreover, we have
	\begin{align*}
		dim\left(D\left(k^{\widehat{G}}_\omega\right)/\tilde{I} \right)=dim\left(D\left(k^{\widehat{G}}_\omega\right)\right)/dim\left(k\widehat{G}\right)=|G|
	\end{align*} 
	by the quasi-Hopf algebra version of the Nichols-Zoeller theorem (see \cite{Schauenburg2004}). Since $dim\left( k^{\widehat{G}}_\omega\right)=|G|$, $f$ must be an isomorphism.
\end{proof}

In the following, we will usually omit the map $\xi$.

\begin{remark}
	\begin{enumerate}
		\item We can identify the group part $\left( k^{\widehat{G}}\right) ^*\otimes k^{\widehat{G} }\subseteq D$ with $D^\omega(\widehat{G})$ from Exp. \ref{ex: PDR group double} via
		\begin{align*}
			D^\omega(\widehat{G}) &\longrightarrow \left( k^{\widehat{G}}\right) ^*\otimes k^{\widehat{G}} \subseteq D \\
			\chi \otimes \delta_{\psi} &\longmapsto \Gamma\left( (1\#\delta_\chi)^* \right)\iota(\delta_{\psi}).
		\end{align*}		
		\item Note that the elements $K_\chi$ and $\bar{K}_\psi$ do not necessarily generate the group part $k^{\widehat{G}} \subseteq u(\omega,\sigma)$ and are not grouplike in general. 
	\end{enumerate}
\end{remark}

\subsection{The $R$-matrix of $u(\omega,\sigma)$}\label{sec:R-matrix}
By Theorem 3.9 in \cite{HausserNill1999a}, the Drinfeld double $D(H)$ of a quasi-Hopf algebra $H$ has the structure of a quasi-triangular quasi-Hopf algebra. Using their formula, we compute the $R$-matrix for the case $H=B(V)\# k^{\widehat{G}}_\omega$:
\begin{align*}
\tilde{R}&= \sum_{b \in \mathcal{B}} \sum_{\chi \in \widehat{G}} \, \iota (b \# \delta_\chi)  \otimes \Gamma\left(  (b \# \delta_{\chi})^* \right).
\end{align*} 
Here, $\mathcal{B}$ is a basis of the Nichols algebra $B(V)$. In order to match our results with the Hopf case described in \cite{Lusztig1993}, we will work with the reverse $R$-matrix
\begin{align*}
	R:=\left( \tilde{R}^T\right)^{-1}.
\end{align*}

The following Lemma is an easy exercise:

\begin{lemma}
	Let $H$ and $H'$ be quasi-Hopf algebras and $\varphi:H \to H'$ a surjective homomorphism of quasi-Hopf algebras. If $R \in H\otimes H$ is an $R$-matrix in $H$, then $R':=(\varphi \otimes \varphi)(R)$ is an $R$-matrix in $H'$. Moreover, if $\nu \in H$ is a ribbon in $(H,R)$, then $\nu':=\varphi(\nu)$ is a ribbon element in $(H',R')$.
\end{lemma}

By the previous Lemma, we can transport the $R$-matrix from $D(H)$ to an $R$-matrix of $u(\omega,\sigma)$. By abuse of notation, we will denote this $R$-matrix also by $R$.

\begin{remark}
	We saw in Section \ref{sec: YDM} that instead of $\sigma$, we could have taken $(\sigma^T)^{-1}$ in order to define our Yetter-Drinfeld module. If we would have defined $u(\omega,\sigma)$ as an algebra over $k_{\sigma}:=k\left( \sigma(\chi,\psi)| \chi,\psi \in \widehat{G} \right)$, then the (well-defined) involution $k_{\sigma} \to k_{\sigma}$ given by $\sigma(\chi,\psi) \mapsto \sigma(\psi,\chi)^{-1}$ would induce an involution $i:u(\omega,\sigma) \to u(\omega,(\sigma^T)^{-1})$ with $i(F_i)=F_i$, $i(E_i)=E_i$ and $i(K_\chi)=\bar{K}_{\chi}^{-1}$.
\end{remark}

From now on, we will omit the quotient map $[\_]:D\left(B(V)\#k^{\widehat{G}}_\omega \right) \to u(\omega,\sigma)$.

\begin{proposition}\label{prop: R-matrix}
	We define elements in $u(\omega,\sigma)^{\otimes 2}$:
	\begin{align*}
		\Theta:&= \sum_{b \in \mathcal{B}}\, \left( \Gamma((b\#\delta_1)^*) \otimes \iota(b)\right)\gamma_{|b|},
		 \quad \text{where} \quad \gamma_{|b|}:= \sum_{\chi,\psi \in \widehat{G}} \, \omega(\chi|\bar{b}|,|b|,\psi) \delta_\chi \otimes \delta_\psi\\
		R_0:&=\sum_{\chi,\psi \in \widehat{G}}\, \sigma(\chi,\psi)\, \delta_\chi \otimes \delta_\psi.
	\end{align*}
	They have the following properties:
	\begin{enumerate}
		\item $\tilde{R}$ decomposes as $\tilde{R}=\Theta^T(R_0^T)^{-1}$, in particular we have $R=R_0\Theta^{-1}$ for the $R$-matrix of $u(\omega,\sigma)$.
		\item Let $\bar{\Delta}$ denote the coproduct in $u(\omega,(\sigma^T)^{-1})$. Then
		\begin{align*}
			\Delta^{op}(h)R_0=R_0\bar{\Delta}(h).
		\end{align*}
		\item The element $\Theta$ is a quasi-$R$-matrix in the sense of \cite{Lusztig1993}, i.e.
		\begin{align*}
			\Delta(h)\Theta = \Theta\bar{\Delta}(h).
		\end{align*}
		Moreover, we have
		\begin{align*}
			\Theta\Delta(h) = \bar{\Delta}(h)\Theta \qquad \text{and} \qquad \Theta^2=1.
		\end{align*}
		\item The Drinfeld element $u \in u(\omega,\sigma)$ (see Lemma \ref{lm: Drinfeld element}) is given by
		\begin{align*}
			u=u_0 \left( \sum_{b \in \mathcal{B}}\,\iota(S_{B(V)}^2(b))\bar{K}_{|b|}\Gamma((b\#\delta_1)^*)\right) 
			= u_0 \left( \sum_{b \in \mathcal{B}}\,\lambda_b\,\iota(b) \bar{K}_{|b|}\Gamma((b\#\delta_1)^*)\right),		
		\end{align*}
		where $u_0 := \sum_{\chi} \,\sigma(\chi,\bar{\chi})\,\delta_{\chi}$ and
		\begin{align*}
			\lambda_b= \prod_{k<l}\,\Beta(\bar{\chi}_{i_k},\bar{\chi}_{i_l}) \qquad \text{for} \qquad b=(\dots(F_{i_1}F_{i_2})\dots)F_{i_n}.
		\end{align*}
	\end{enumerate}
\end{proposition}

\begin{proof}
	We start with (1): We have a general product formula
	\begin{align*}
		\Gamma((b\#\delta_\chi)^*)\Gamma((b'\#\delta_\psi)^*) &= \sum_{\pi \in \widehat{G}}\, 
		\frac{\omega(\pi|\bar{b}'||\bar{b}|,|b|,\chi)\omega(\pi|\bar{b}'|,|b'|,\psi)}{\omega(\pi|\bar{b}'||\bar{b}|,|b|,|b'|)\omega(\pi|\bar{b}'||\bar{b}|,|b||b'|,\chi\psi)} \frac{\sigma(\pi,\chi\psi)}{\sigma(\pi,\psi)\sigma(\pi|b'|,\chi)} \\
		&\times \Gamma\left( (b*b' \#\delta_{\chi\psi})^*\right) \iota(\delta_\pi),
	\end{align*}
	where $b*b'$ is the quantum shuffle product in $B(V)$ as introduced Appendix \ref{app: quantum shuffle product}.
	In particular, we have 
	\begin{align*}
		\Gamma((b\#\delta_\chi)^*)&=\Gamma((b\#\delta_1)^*)\Gamma((1\#\delta_\chi)^*)\left(\sum_{\psi \in \widehat{G}}\,\omega(\psi|\bar{b}|,|b|,\chi)\, \iota(\delta_{\psi}) \right) \\
		&=\Gamma((b\#\delta_1)^*)\bar{K}_{\chi}^{-1}\left(\sum_{\psi \in \widehat{G}}\,\omega(\psi|\bar{b}|,|b|,\chi)\, \iota(\delta_{\psi}) \right),
	\end{align*}
	where we used in the second line that $c_{\chi}=1$ holds in the quotient $u(\omega, \sigma)$. Thus,
	\begin{align*}
		\tilde{R}&= \sum_{b \in \mathcal{B}} \sum_{\chi \in \widehat{G}} \, \iota (b \# \delta_\chi)  \otimes \Gamma\left(  (b \# \delta_{\chi})^* \right) \\
		&=\sum_{b \in \mathcal{B}} \sum_{\chi \in \widehat{G}} \, \iota (b \# \delta_\chi)  \otimes \Gamma((b\#\delta_1)^*)\bar{K}_{\chi}^{-1}\left(\sum_{\psi \in \widehat{G}}\,\omega(\psi|\bar{b}|,|b|,\chi)\, \iota(\delta_{\psi}) \right)\\
		&=\sum_{b \in \mathcal{B}} \, (\iota (b)  \otimes \Gamma((b\#\delta_1)^*))\gamma_{|b|}\left(\sum_{\chi,\psi \in \widehat{G}}\, \sigma(\psi,\chi)^{-1} \iota(\delta_{\chi})\otimes \iota(\delta_{\psi}) \right)\\
		&=\Theta^T(R_0^T)^{-1}.
	\end{align*}
	We now prove (2): It is sufficient to prove the formula for $h=E_i,F_j,\iota(\delta_\chi)$, since these elements generate $u(\omega, \sigma)$. We will only show the computation for $E_i$, since for $\delta_{\chi}$ it is trivial and for $F_j$ it is very similar.
	\begin{align*}
		\Delta^{op}(E_i)R_0 &=\left( \sum_{\chi,\psi}\, \theta(\psi|\chi\bar{\chi}_i,\chi_i)^{-1}\omega(\psi,\chi,\bar{\chi}_i)^{-1} \,  \delta_\psi \otimes \delta_\chi \ \right)(\bar{K}_{\chi_i} \otimes E_i )R_0  \\
		&+ \left( \sum_{\chi,\psi}\, \omega(\chi,\psi,\bar{\chi}_i)^{-1} \, \delta_\psi \otimes \delta_\chi  \right) (E_i \otimes 1)R_0\\
		&=\left( \sum_{\chi,\psi}\, \theta(\psi|\chi\bar{\chi}_i,\chi_i)^{-1}\sigma(\psi,\chi_i)\sigma(\psi,\chi\bar{\chi}_i)\omega(\psi,\chi,\bar{\chi}_i)^{-1} \,  \delta_\psi \otimes \delta_\chi \ \right)(1 \otimes E_i ) \\
		&+ \left( \sum_{\chi,\psi}\,\sigma(\psi,\bar{\chi}_i)\omega(\chi,\psi,\bar{\chi}_i)^{-1} \, \delta_\psi \otimes \delta_\chi  \right) (E_i \otimes 1)\\
		&=\left( \sum_{\chi,\psi}\, \sigma(\psi,\chi)\omega(\psi,\chi,\bar{\chi}_i)^{-1} \,  \delta_\psi \otimes \delta_\chi \ \right)(1 \otimes E_i ) \\
		&+ \left( \sum_{\chi,\psi}\,\theta(\chi|\psi\bar{\chi}_i,\chi_i)^{-1} \frac{\sigma(\psi,\chi)}{\sigma(\chi_i,\chi)}\omega(\chi,\psi,\bar{\chi}_i)^{-1} \, \delta_\psi \otimes \delta_\chi  \right) (E_i \otimes 1)\\
		&=R_0\left( \left( \sum_{\chi,\psi}\, \omega(\chi,\psi,\bar{\chi}_i)^{-1} \,  \delta_\chi \otimes \delta_\psi \ \right)(1 \otimes E_i )\right. \\
		&\left. + \left( \sum_{\chi,\psi}\,\theta(\psi|\chi\bar{\chi}_i,\chi_i)^{-1} \omega(\psi,\chi,\bar{\chi}_i)^{-1} \, \delta_\chi \otimes \delta_\psi  \right) (E_i \otimes K_{\chi_i}^{-1})\right) =R_0\bar{\Delta}(E_i).
	\end{align*}
	We continue with (3): Since $R=R_0\Theta^{-1}$ is an $R$-matrix, we have
	\begin{align*}
		R_0\Theta^{-1}\Delta(h)=\Delta^{op}(h)R_0\Theta^{-1}=R_0\bar{\Delta}(h)\Theta^{-1},
	\end{align*}
	where we used (2) in the second equation. This proves the first claim of (2). If we would have used $(\sigma^T)^{-1}$ instead of $\sigma$ in our construction of $u(\omega,\sigma)$, the element $\Theta$ would be exactly the same, whereas $R_0$ would change to $R_0^{-1}$. Since $\Theta^T R_0$ is then an $R$-matrix in $u(\omega,(\sigma^T)^{-1})$, we have
	\begin{align*}
		\bar{\Delta}^{op}(h)\Theta^TR_0=\Theta^TR_0\bar{\Delta}(h)=\Theta^T\Delta^{op}(h)R_0,
	\end{align*} 
	which implies $\bar{\Delta}(h)\Theta=\Theta\Delta(h)$. In particular, $\Theta^{-1}$ is a quasi-$R$-matrix as well. By an analogous argument as given in \cite{Lusztig1993} for the quasi-Hopf case, a quasi-$R$-matrix is unique, hence $\Theta^{-1}=\Theta$. \\
	Finally, we prove (4): We want to compute the Drinfeld element for the $R$-matrix $R$, but it is easier to compute it in terms of the $R$-matrix $\tilde{R}=\tilde{R}^1\otimes \tilde{R}^2$. Using graphical calculus it is not hard to find the following formula for $u$:
	\begin{align*}
		u=S^2(\tilde{q}^2\tilde{R}^1\tilde{p}^1)\tilde{q}^1\tilde{R}^2\tilde{p}^2.
	\end{align*} 
	After simplifying, we obtain
	\begin{align*}
		u= \sum_{b \in \mathcal{B}} \, S^2(\iota(b))u_0K_{|b|}^{-1}\Gamma((b\#\delta_1)^*).
	\end{align*}
	The square of the antipode is given by
	\begin{align*}
		S^2(\iota(b)) &= S\left(\sum_{\chi \in \widehat{G}}\, \omega(\bar{\psi}|\bar{b}|,|b|,\psi)\sigma(|b|,\bar{\psi}|\bar{b}|)\, (S_{B(V)}(b) \#\delta_{\psi}) \right)\\
		&= u_0\iota\left( S^2_{B(V)}(b)\right)K_{|b|}\bar{K}_{|b|}u_0^{-1}.
	\end{align*}
	Hence, 
	\begin{align*}
		u=u_0 \left(\sum_{b \in \mathcal{B}}\, \iota(S^2_{B(V)}(b))\bar{K}_{|b|}\Gamma((b\#\delta_1)^*) \right).
	\end{align*}
	The antipode in the Nichols algebra $S_{B(V)}(b)$ is given by
	\begin{align*}
		S_{B(V)}(b)=(-1)^{tr|b|} \mu_{b}\, b^T,
	\end{align*}
	where
	\begin{align*}
		\mu_b&=\,\prod_{j=1}^{n-1}\,\sigma\left(\prod_{k=1}^{j} \,\bar{\chi}_{i_k},\bar{\chi}_{i_{j+1}} \right) \,\prod_{j=2}^{n-1} \omega\left( \prod_{k=j+1}^{n} \,\bar{\chi}_{i_k},\bar{\chi}_{i_j}, \prod_{l=1}^{j-11} \,\bar{\chi}_{i_l}\right) \quad \text{for} \quad
		b=(\dots(F_{i_1}F_{i_2})\dots)F_{i_n}.		
	\end{align*}
	A tedious calculation shows that $\mu_b\mu_{b^T}=\lambda_b$, hence $S^2_{B(V)}(b)=\lambda_b\,b$. This proves the claim.
\end{proof}

\begin{remark}
	Since $(\omega,\sigma) \in Z^3_{ab}(\widehat{G})$ is an abelian $3$-cocycle, it is clear that 
	\begin{align*}
		R_0=\sum_{\chi,\psi \in \widehat{G}} \, \sigma(\chi,\psi)\, \delta_\chi \otimes \delta_\psi
	\end{align*}
	is an $R$-matrix for the quasi-Hopf algebra $k^{\widehat{G}}_{\omega}$, so that the monomorphism $\xi:k^{\widehat{G}}_{\omega} \to u(\omega,\sigma)$ from Prop. \ref{prop: split monomorphism} becomes a homomorphism of quasitriangular quasi-Hopf algebras.
\end{remark}

	\section{Modularization}\label{sec: Modularization}

In the following definition due to \cite{Shimizu2016}, we specify the class of categories we want to consider in this chapter.

\begin{definition}\label{def: non-ssi mtc}
	By a (braided) finite tensor category $\mathcal{C}$, we mean a $k$-linear category that is equivalent to $\mathsf{Rep}_A$ for some finite dimensional $k$-algebra $A$. Here, we assume $k$ to be algebraically closed. Moreover, $\mathcal{C}$ should be monoidal, rigid (and braided). A functor $F:\mathcal{C} \to \mathcal{D}$ is called a (braided) tensor functor if it is (braided) $k$-linear monoidal.	
\end{definition}

The following proposition due to Lyubashenko is proven in \cite{KerlerLyubashenko2001}.

\begin{proposition}\label{prop: coend is hopf alg}
	Let $\mathcal{C}$ be a finite braided tensor category. In this case, the coend $\mathbb{F}_{\mathcal{C}}= \int^{X \in \mathcal{C}} X^\vee \otimes X$ exists and has the canonical structure of a Hopf algebra in $\mathcal{C}$. Moreover, there is a symmetric Hopf pairing $\omega_{\mathcal{C}}:\mathbb{F}_{\mathcal{C}} \otimes \mathbb{F}_{\mathcal{C}} \to \mathbb{I}$ on $\mathbb{F}_{\mathcal{C}}$. 
\end{proposition}

\begin{definition}
	Let $\mathcal{C}$ be a finite braided tensor category. If moreover $\mathcal{C}$ is a ribbon category, we call $\mathcal{C}$ premodular. In \cite{Shimizu2016}, Shimizu showed that the following conditions are equivalent.
	\begin{enumerate}
		\item The Hopf pairing $\omega_{\mathcal{C}}:\mathbb{F}_{\mathcal{C}}\otimes \mathbb{F}_{\mathcal{C}} \to \mathbb{I}$ on the coend $\mathbb{F}_{\mathcal{C}}= \int^{X \in \mathcal{C}} X^\vee \otimes X$ is non-degenerate.
		\item Every transparent object is isomorphic to the direct sum of finitely many copies of the unit object $\mathbb{I} \in \mathcal{C}$. Equivalently, the M\"uger center $\mathcal{C}'$ of $\mathcal{C}$, which is defined as the full subcategory of transparent objects, is braided equivalent to $\mathsf{Vect}_k$. 
	\end{enumerate}
	A premodular category satisfying these conditions is called modular. In the semisimple case, we refer to $\mathcal{C}$ as a (pre)modular fusion category.
\end{definition} 

\begin{remark}\label{rm: dont need ribbon}
	In order to define a Hopf structure on the coend $\mathbb{F}$, the ribbon structure is usually used to define the antipode. For the case of $\mathcal{C}=\mathsf{Rep}_H$, where $H$ is a quasi-triangular quasi-Hopf algebra, the antipode on $\mathbb{F}=\underline{H^*}$ is induced by the Drinfeld element $u\in H$ (see App. \ref{app:Factorizability}) and $H$ does not necessarily have to be ribbon. 
\end{remark}
	We suggest a definition for a modularization in the non-semisimple case
	 (we also refer to a similar definition in~\cite{FuchsGainutdinovSchweigert}):
\begin{definition}\label{def: non ssi modularization}
	A ribbon (i.e. twist-preserving braided tensor) functor between premodular categories $F:\mathcal{C} \to \mathcal{D}$ is a modularization of $\mathcal{C}$, if $\mathcal{D}$ is modular and
	\begin{align*}
	F \left( \mathbb{F}_\mathcal{C}/Rad(\omega_\mathcal{C}) \right)  \cong \mathbb{F}_\mathcal{D}.
	\end{align*} 
	as braided Hopf algebras.
\end{definition} 

\begin{remark}
	We should point out here, that there are other approaches in order to define a non-semisimple modularization. For example, the notion of a dominant functor still makes sense in this case. It is therefore tempting to define a modularization of $\mathcal{C}$ as an exact sequence $\mathcal{C}' \to \mathcal{C} \xrightarrow{F} \mathcal{D}$ of tensor categories in the sense of \cite{BruguieresNatale2011}, where $\mathcal{C}'$ is the M\"uger center of $\mathcal{C}$ and $\mathcal{D}$ is modular. At least in the setting of Thm. \ref{thm:modularzation}, this coincides with Def. \ref{def: non ssi modularization}. It would also be interesting to generalize the actual construction of a modularization in \cite{Bruguieres2000} to the non-semisimple case. A step in this direction has been made in \cite{BruguieresNatale2011}, where the authors show that a dominant functor $F:\mathcal{C} \to \mathcal{D}$ with exact right adjoint is equivalent to the free module functor $\mathcal{C} \to \text{mod}_{\mathcal{C}}(A)$ for some commutative algebra $A$ in the center of $\mathcal{C}$. One of the reasons for choosing the given definition is its closeness to one of the equivalent definitions of a modular tensor category. We conjecture that it reduces to the definition of Bruguier\`es in the semisimple case. 
	The mentioned problems are studied in~\cite{FuchsGainutdinovSchweigert}.
\end{remark}

Let $u:=(u_q(\g,\Lambda),R_0\bar{\Theta})$ be the quasi-triangular ribbon Hopf algebra as described in Section \ref{Quantum groups uq(g,Lambda) and R-matrices} with Cartan part $u^0=\C[G]$, where $G=\Lambda/\Lambda'$. In particular, the category $\mathsf{Rep}_u$ is a non-semisimple premodular category. As we have seen before, we have an equivalence
\begin{align*}
\mathsf{Rep}_{(u^0,R_0)} \longrightarrow \Vect_{\widehat{G}}^{(1,\sigma)},
\end{align*}  
where $\sigma:\widehat{G} \times \widehat{G} \to \C^\times$ is given as in Remark \ref{relation sigma-f}. \\
\\
\textbf{Assumption:} From now on, we assume that $\Vect_{\widehat{G}}^{(1,\sigma)}$ is modularizable,
i.e. the quadratic form $Q$ associated to $\sigma$ is trivial on $T:=Rad(\sigma\sigma^T)$. Let $\Vect_{\widehat{G}/T}^{(\bar{\sigma},\bar{\omega})}$ be the modularized category from Prop. \ref{Modularizability of VectG}.\\
\\
The aim of this section is to modularize the category $\mathsf{Rep}_u$. To this end, we first define a quasi-Hopf algebra $\bar{u}$ and an algebra monomorphism $M:\bar{u} \to u$. Then we show that restriction along this algebra inclusion defines a modularization functor $\mathcal{F}:\mathsf{Rep}_u \to \mathsf{Rep}_{\bar{u}}$.\\
Let $\bar{u}:= (u(\bar{\omega},\bar{\sigma}),\bar{R})$ denote the quasi-Hopf algebra from Thm. \ref{thm:quasi-Hopf algebra}, associated to the data $(\widehat{\bar{G}},\bar{\sigma},\bar{\omega},\chi_i:=q^{(\alpha_i, \_)}|_{\bar{G}})$. Here, $\bar{G}:=Ann(T) \subseteq G$ is the subgroup introduced in App. \ref{sec: group theoretic part} and 
\begin{align*}
	\bar{\omega}(\chi,\psi,\xi)=\sigma(s(\xi),r(\chi,\psi))\frac{\eta(r(\chi,\psi\xi),r(\psi,\xi))}{\eta(r(\chi\psi,\xi),r(\chi,\psi))}, \qquad \bar{\sigma}(\chi,\psi)=\sigma(s(\chi),s(\psi))
\end{align*}
is the abelian $3$-cocycle on $\widehat{G}/T$ associated to an arbitrary set-theoretic section $s:\widehat{G}/T \to \widehat{G}$ of the canonical projection $\pi:\widehat{G} \to \widehat{G}/T$ as defined in Lemma \ref{lm: explicit 3 cocycle omegabar}. The $2$-cocycle $r$ was defined as $r(\chi,\psi):= s(\chi)s(\psi)s(\chi\psi)^{-1}$. Note that $\widehat{\bar{G}}\cong \widehat{G}/T$. In particular, $\bar{u}$ has all the necessary structure to endow $\mathsf{Rep}_{\bar{u}}$ with a finite braided tensor structure. \\
\\
We now state the main result of this chapter:

\begin{theorem} \label{thm:modularzation}
	\begin{enumerate}
		\item The category $\mathsf{Rep}_{\bar{u}}$ is a modular tensor category.
		\item The restriction along the algebra monomorphism $M$ from Prop. \ref{prop: algebra inclusion M} defines a modularization $\mathcal{F}:\mathsf{Rep}_{u} \to \mathsf{Rep}_{\bar{u}}$ in the sense of Def. \ref{def: non ssi modularization}. The monoidal structure of this functor is given by
		\begin{align*}
		\tau_{V,W} : \mathcal{F}(V)\otimes \mathcal{F}(W) &\longrightarrow \mathcal{F}(V\otimes W) \\
		v\otimes w &\longmapsto \sum_{\chi,\psi \in \widehat{G}}\,\kappa(\chi,\psi)\, \delta_{\chi}.v \otimes \delta_{\psi}.w,
		\end{align*}
		where the $2$-cochain $\kappa \in C^2(\widehat{G})$ is defined in Lemma \ref{lm: explicit 3 cocycle omegabar}. It satisfies $\pi^*(\bar{\omega},\bar{\sigma})=(1,\sigma)$. The natural duality isomorphism is given by
		\begin{align*}
		\xi_V:\mathcal{F}(V^\vee) &\longrightarrow \mathcal{F}(V)^\vee \\
		f &\longmapsto (f \leftharpoonup S(\kappa^1)\kappa^2),
		\end{align*}
		where $\kappa=\kappa^1\otimes \kappa^2=\sum_{\chi,\psi}\, \kappa(\chi,\psi)\, \delta_\chi \otimes \delta_\psi$ by abuse of notation.
	\end{enumerate}	
\end{theorem}

The rest of this chapter is devoted to the proof of this theorem. We will need the following proposition.

\begin{proposition}\label{prop: left adjoint with inverse allows for modularization}
	Let $F:\mathcal{C} \to \mathcal{D}$ be a cocontinuous left exact braided tensor functor between finite braided tensor categories.
	Let $\mu:(\_)^\vee\otimes (\_) \Rightarrow \mathbb{F}_{\mathcal{C}}$ and $\nu: (\_)^\vee \otimes (\_) \to \mathbb{F}_\mathcal{D}$ denote the coends in $\mathcal{C}$ and $\mathcal{D}$. Then there is a braided Hopf algebra epimorphism $p:F(\mathbb{F}_{\mathcal{C}}) \to \mathbb{F}_{\mathcal{D}}$ in $\mathcal{D}$, s.t. the Hopf pairings on the coends are related as follows:
	\begin{align*}
	\omega_{_{\mathcal{D}}}\circ (p \otimes p) = F(\omega_{\mathcal{C}}) \circ \tau_{\mathbb{F}_{\mathcal{C}},\mathbb{F}_{\mathcal{C}}}.
	\end{align*}
	If $\mathcal{D}$ is a modular tensor category, then $F$ is a modularization in the sense of Def. \ref{def: non ssi modularization}, with isomorphism $F(\mathbb{F}_{\mathcal{C}}/Rad(\omega_{\mathcal{C}})) \cong \mathbb{F}_\mathcal{D}$ induced by $p$. 
\end{proposition}

\begin{proof}
	We first note that since $F$ is cocontinuous, it preserves the coend and hence $F(\mathbb{F}_\mathcal{C})$ is the coend over the functor $F((\_)^\vee\otimes (\_))$ with dinatural transformation $F(\mu)$. We also have a Hopf pairing on $F(\mathbb{F}_\mathcal{C})$ given by $\omega_{F(\mathbb{F}_\mathcal{C})}=F(\omega_{\mathcal{C}})\circ \tau_{\mathbb{F}_{\mathcal{C}},\mathbb{F}_{\mathcal{C}}}$, where $\tau$ denotes the monoidal structure on $F$. Let  $m:F((\_)^\vee\otimes (\_)) \Rightarrow F(\_)^\vee\otimes F(\_)$ denote the natural isomorphism constructed from the structure isomorphisms of the monoidal dual preserving functor $F$. Then,
	\begin{align*}
	\zeta_V:=\nu_{F(V)}\circ m_V:F(V^\vee \otimes V) \to \mathbb{F}_{\mathcal{D}}
	\end{align*}
	defines a dinatural transformation $F((\_)^\vee \otimes (\_)) \to \mathbb{F}_{\mathcal{D}}$. Hence, due to the universal property of $F(\mathbb{F}_\mathcal{C})$, there is a unique morphism $p:F(\mathbb{F}_\mathcal{C}) \to \mathbb{F}_{\mathcal{D}}$, s.t.
	\begin{align*}
	\nu_{F(V)}\circ m_V = p \circ F(\mu_V).
	\end{align*}
	Again by the universal property, $p$ must be an epimorphism. If we can show that
	\begin{align*}
	\omega_{\mathcal{D}}\circ(p \otimes p)\circ(F(\mu_V) \otimes F(\mu_W))= \omega_{F(\mathbb{F}_\mathcal{C})}\circ (F(\mu_V) \otimes F(\mu_W)),
	\end{align*}
	for all $V,W \in \mathcal{C}$, then $\omega_{\mathcal{D}}\circ(p \otimes p)=\omega_{F(\mathbb{F}_\mathcal{C})}$ holds again by the universal property of $F(\mathbb{F}_\mathcal{C})$. We have
	\begin{align*}
	\omega_{\mathcal{D}}\circ(p \otimes p)\circ(F(\mu_V) \otimes F(\mu_W)) &= \omega_{\mathcal{D}}\circ(\nu_{F(V)}\circ m_V \otimes \nu_{F(W)}\circ m_W) \\
	&={\omega_{\mathcal{D}}}_{F(V),F(W)}\circ(m_V \otimes m_W) \\
	&=\left( F({\omega_{\mathcal{C}}}_{V,W})\circ \tau_{V^\vee \otimes V,W^\vee \otimes W} \circ (m_V\otimes m_W)^{-1} \right) \circ  (m_V \otimes m_W) \\
	&=F({\omega_{\mathcal{C}}}_{V,W})\circ \tau_{V^\vee \otimes V,W^\vee \otimes W} \\
	&=\omega_{F(\mathbb{F}_{\mathcal{C}})}\circ (F(\mu_V ) \otimes F(\mu_W)),
	\end{align*}
	where ${\omega_{\mathcal{D}}}_{X,Y}$ and ${\omega_{\mathcal{C}}}_{V,W}$ denote the bi-dinatural transformations given, morally speaking, by $(ev \otimes ev)\circ (\id \otimes c^2\otimes \id)$. We used the equality
	\begin{align*}
	{\omega_{\mathcal{D}}}_{F(V),F(W)}=F({\omega_{\mathcal{C}}}_{V,W})\circ \tau_{V^\vee \otimes V,W^\vee \otimes W} \circ (m_V\otimes m_W)^{-1},
	\end{align*}
	which holds since $F$ is a braided tensor functor. The fact that $p$ is a morphism of braided Hopf algebras follows from the very same arguments. We now prove the second part of the statement. By Lemma 5.2.1 in \cite{KerlerLyubashenko2001}, we can identify $Rad(\mathcal{\omega}_{\mathcal{C}})$ with the kernel of the adjunct $\omega_\mathcal{C}^!:\mathbb{F}_{\mathcal{C}} \to \mathbb{F}_{\mathcal{C}}^\vee$. Since $F$ is left exact, we have $F(ker(\omega_{\mathbb{F}^!_{\mathcal{C}}}))=ker(\omega^!_{F(\mathbb{F}_{\mathcal{C}})})$. As we have seen, $\omega^!_{F(\mathbb{F}_{\mathcal{C}})}$ is given by $p^\vee \circ \omega_{\mathbb{F}^!_{\mathcal{D}}}\circ p$. If $\mathcal{D}$ is modular, then $p^\vee \circ \omega_{\mathbb{F}^!_{\mathcal{D}}}$ is a monomorphism and hence $ker(\omega^!_{F(\mathbb{F}_{\mathcal{C}})})=ker(p)$. As a cocontinuous functor, $F$ preserves quotients and hence $F(\mathbb{F}_{\mathcal{C}}/ker(\omega_\mathcal{C}^!)) \cong F(\mathbb{F}_{\mathcal{C}})/ker(p) \overset{p}{\cong} \mathbb{F}_{\mathcal{D}}$.
\end{proof}

We recall that the braided monoidal category $\Vect_{\widehat{G}}^{(\omega,\sigma)}$ is embedded into the braided monoidal category of Yetter-Drinfeld modules over $k^{\widehat{G}}_{\omega}$ with $k^{\widehat{G}}_{\omega}$-coaction on simple objects $\C_{\chi}$ given by $1_{\chi} \mapsto L_{\bar{\chi}} \otimes 1_{\chi}$. From now on, we will treat $\Vect_{\widehat{G}}^{(\omega,\sigma)}$ as a braided monoidal subcategory of Yetter-Drinfeld modules over $k^{\widehat{G}}_{\omega}$.\\
\\
Let $F: \Vect_{\widehat{G}}^{(\sigma,1)} \to \Vect_{\widehat{G}/T}^{(\bar{\sigma},\bar{\omega})}$ be the braided monoidal modularization functor from Section \ref{Modularizability of VectG} and $\tau_{V,W}:F(V)\otimes F(W)\to F(V\otimes W)$ the corresponding monoidal structure. For the moment, we forget about the ribbon structure. Moreover, let $V=\oplus_i\,F_i\C \in \Vect_{\widehat{G}}^{(\sigma,\omega)}$ be the Yetter-Drinfeld module with $|F_i|=\bar{\chi}_i=q^{-(\alpha_i,\_)}$. 
If $B(V)$ is the Nichols algebra of $V$, then we clearly have a braided algebra structure on $F(B(V))$ with multiplication given by $m_{F(B(V))}=F(m_{B(V)})\circ \tau_{B(V),B(V)}$. Also, we have a braided algebra isomorphism
$S:B(F(V)) \longrightarrow F(B(V))$ induced by the map $D:= \oplus_{n\geq 0}\, D^n $, where $D^n:F(V)^{\otimes n} \to F(V^{\otimes n})$ is inductively defined by
\begin{align*}
D^n:=\tau_{V^{\otimes n-1},V}\circ (D^{n-1}\otimes \id_{F(V)}), \qquad D^0:=\id_{\C}.
\end{align*}

From now on, we assume that the set-theoretic section $s:\widehat{G}/T \to \widehat{G}$ of the projection $\pi:\widehat{G}\to \widehat{G}/T$ has the property $s([|b|])=|b|$ for homogeneous vectors $b \in B(V)$, where we use the notation $\pi(\chi)=[\chi]$ for the sake of readability.

\begin{lemma}
	Let $\kappa:\widehat{G}\times \widehat{G} \to \C^{\times}$ be such that $\pi^*\bar{\omega}=d\kappa^{-1}$ and $\pi^*\bar{\sigma}=\kappa/\kappa^T \sigma$. The following map is an algebra inclusion:
	\begin{align*}
	U:B(F(V))\#k^{\widehat{G}/T}_{\bar{\omega}} &\longrightarrow B(V) \# k^{\widehat{G}}\\
	b\#\delta_{\chi} &\longmapsto \sum_{\tau \in T}\,\kappa(|b|,s(\chi)\tau)\,S(b)\#\delta_{s(\chi)\tau}.
	\end{align*}	
	Here, $b=(\dots(F_{i_1}F_{i_2})\dots)F_{i_n}$ is a PBW-basis element of the Nichols algebra $B(F(V))$. Moreover, we used the assumption $s([|b|])=|b|$ implicitly.
\end{lemma}
\begin{proof}
	We know that $\kappa(1,\chi)=1$, $S(1)=1$ and 
	\begin{align*}
		\sum_{\chi \in \widehat{G}/T} \sum_{\tau \in T}\, \delta_{s(\chi)\tau} =\sum_{\psi \in \widehat{G}} \, \delta_{\psi} =1_{k^{\widehat{G}}}.
	\end{align*}
	Hence, $U$ preserves the unit. Moreover, we have
	\begin{align*}
		U((b\#\delta_\chi)(b'\#\delta_\psi)) &= \delta_{\chi,[|b'|]\psi} \,\bar{\omega}([|b|],[|b'|],\psi)^{-1}U(bb' \#\delta_{\psi}) \\
		&=\delta_{\chi,[|b'|]\psi}\,\sum_{\tau \in T}\, \bar{\omega}([|b|],[|b'|],\psi)^{-1}\kappa(|b||b'|,s(\psi)\tau)\, S(bb')\#\delta_{s(\psi)\tau} \\
		&=\delta_{\chi,[|b'|]\psi}\,\sum_{\tau \in T}\, \bar{\omega}([|b|],[|b'|],\psi)^{-1}\kappa(|b||b'|,s(\psi)\tau)\kappa(|b|,|b'|)\, S(b)S(b')\#\delta_{s(\psi)\tau},
	\end{align*}
	where we used $S(bb')=\kappa(|b|,|b'|)\, S(b)S(b')$ in the last line.
	On the other hand, we have
	\begin{align*}
		U(b\#\delta_\chi)U(b'\#\delta_\psi)&= \sum_{\tau,\tau' \in T}\, \kappa(|b|,s(\chi)\tau')\kappa(|b'|,s(\psi)\tau)\,(S(b)\#\delta_{s(\chi)\tau'})(S(b')\#\delta_{s(\psi)\tau})\\
		&= \sum_{\tau,\tau' \in T}\, \delta_{s(\chi)\tau',|b'|s(\psi)\tau}\, \kappa(|b|,s(\chi)\tau')\kappa(|b'|,s(\psi)\tau)\,S(b)S(b')\#\delta_{s(\psi)\tau}\\
		&=\delta_{\chi,[|b'|]\psi}\,\sum_{\tau\in T}\,  \kappa(|b|,|b'|s(\psi)\tau)\kappa(|b'|,s(\psi)\tau)\,S(b)S(b')\#\delta_{s(\psi)\tau}.
	\end{align*}
	Since $\psi=[s(\psi)\tau]$ and $\pi^*\bar{\omega}=(d\kappa)^{-1}$, we have an equality.
\end{proof}

\begin{remark}\label{rem: left inverse of algebra inclusion}
	Note that the above algebra homomorphism $U$ has a linear left inverse, given by
	\begin{align*}
		Q: B(V) \# k^{\widehat{G}} &\longrightarrow B(F(V))\#k^{\widehat{G}/T}_{\bar{\omega}}\\
		b\#\delta_\psi &\longmapsto \kappa(|b|,\psi)^{-1} \, S^{-1}(b)\#s^*\delta_{\psi}.
	\end{align*}
	It is easy to see that this map preserves the unit. Moreover we have 
	\begin{align*}
		Q((b\#\delta_{\chi})(b'\#\delta_{\psi}))=Q(b\#\delta_{\chi})Q(b'\#\delta_{\psi})
	\end{align*}
	if and only if $s([|b'|\psi])=|b'|\psi$.
\end{remark}

We now show that there is an algebra homomorphism between the corresponding small quantum groups of $\bar{u}^{\leq}:=B(F(V))\#k^{\widehat{G}/T}_{\bar{\omega}}$ and $u^{\leq }:=B(V) \# k^{\widehat{G}}$ extending $U$:

\begin{proposition}\label{prop: algebra inclusion M}
	The following map defines an algebra inclusion:
	\begin{align*}
	M:\bar{u} &\longrightarrow u \\
	\Gamma_{\bar{u}}((b\#\delta_1)^*) &\longmapsto \Gamma_u((S(b)\#\delta_1)^*)\left( \sum_{\chi \in \widehat{G}}\,\kappa(\chi|\bar{b}|,|b|)^{-1}\, \delta_{\chi} \right)\\
	\iota_{\bar{u}}(b\#\delta_\chi) &\longmapsto \iota_u(U(b\#\delta_\chi)) .
	\end{align*}
	Here, $\Gamma$ and $\iota$ are the maps introduced Thm. \ref{thm: maps iota and Gamma}.
\end{proposition}

\begin{proof}
	It is not hard to see that elements of the form $\Gamma_{\bar{u}}((b\#\delta_1)^*)$ and $\iota_{\bar{u}}(b\#\delta_\chi)$ generate the algebra $\bar{u}$, since we have for a general small quasi-quantum group $u(\omega,\sigma)$
	\begin{align*}
		\Gamma((b\#\delta_\chi)^*)=\Gamma((b\#\delta_1)^*)\left( \sum_{\psi}\, \frac{\omega(\psi|\bar{b}|,|b|,\chi)}{\sigma(\psi,\chi)}\delta_{\psi}\right) c_{\chi}
		=\Gamma((b\#\delta_1)^*)\left( \sum_{\psi}\, \frac{\omega(\psi|\bar{b}|,|b|,\chi)}{\sigma(\psi,\chi)}\delta_{\psi}\right),
	\end{align*}
	where we used that $c_\chi=1$ holds in the quotient. Moreover,  the elements $\Gamma((b\#\delta_1)^*)\iota(\tilde{b}\#\delta_\chi) \in u(\omega,\sigma)$ form a basis. Hence, we need to show that $M$ preserves products of the form
	\begin{align*}
		\left( \Gamma((b\#\delta_1)^*)\iota(\tilde{b}\#\delta_\chi)\right) \cdot\left( \Gamma((b'\#\delta_1)^*)\iota(\tilde{b}'\#\delta_\psi)\right) 
	\end{align*}
	Since we know that $M|_{\bar{u}^{\leq}}=U$ is a quasi-Hopf inclusion, for this it is sufficient to prove the following relations:
	\begin{itemize}
		\item $M\left( \Gamma((b\#\delta_1)^*)\Gamma((b'\#\delta_1)^*)\right) =M\left( \Gamma((b\#\delta_1)^*)\right) M\left( \Gamma((b'\#\delta_1)^*)\right) $,
		\item $M\left(\iota(\tilde{b}\#\delta_\chi) \Gamma((b\#\delta_1)^*) \right) =M\left(\iota(\tilde{b}\#\delta_\chi)\right) M\left(  \Gamma((b\#\delta_1)^*) \right) $.
	\end{itemize}
	We start with the first relation. We have a general formula for the product $\Gamma((b\#\delta_1)^*)\Gamma((b'\#\delta_1)^*)$ given by
	\begin{align*}
		\Gamma((b\#\delta_1)^*)\Gamma((b'\#\delta_1)^*)=\sum_{\xi \in {\widehat{G}/T}}\, \bar{\omega}(\xi[|\bar{b}'||\bar{b}|],[|b|],[|b'|])^{-1}\, \Gamma((b*b'\#\delta_1)^*)\delta_{\xi},
	\end{align*}
	where $b*b'\in B(F(V))$ denotes the quantum shuffle product as introduced in Appendix \ref{app: quantum shuffle product}. Hence,
	\begin{align*}
		M\left( \Gamma((b\#\delta_1)^*)\Gamma((b'\#\delta_1)^*)\right) &=M\left( \sum_{\xi \in {\widehat{G}/T}}\, \bar{\omega}(\xi[|\bar{b}'||\bar{b}|],[|b|],[|b'|])^{-1}\, \Gamma((b*b'\#\delta_1)^*)\iota(\delta_{\xi}) \right) \\
		&= \sum_{\xi \in {\widehat{G}/T}}\, \bar{\omega}(\xi[|\bar{b}'||\bar{b}|],[|b|],[|b'|])^{-1}\, M\left( \Gamma((b*b'\#\delta_1)^*)\right) M\left( \delta_{\xi}\right) \\
		&=\sum_{\chi \in \widehat{G}}\, \bar{\omega}([\chi|\bar{b}'||\bar{b}|],[|b|],[|b'|])^{-1}\kappa(\chi|\bar{b}||\bar{b}'|,|b||b'|)^{-1}\,  \Gamma((S(b*b')\#\delta_1)^*)\iota\left(  \delta_{\chi} \right) \\
		&=\sum_{\chi \in \widehat{G}}\, \bar{\omega}([\chi|\bar{b}'||\bar{b}|],[|b|],[|b'|])^{-1}\kappa(\chi|\bar{b}||\bar{b}'|,|b||b'|)^{-1}\kappa(|b|,|b'|)^{-1} \\
		&\times \,\Gamma((S(b)*S(b')\#\delta_1)^*)\iota\left(  \delta_{\chi} \right), 
	\end{align*}
	where we used $S(b*b')=\kappa(|b|,|b'|)^{-1}S(b)*S(b')$ in the last equality. On the other hand, we have
	\begin{align*}
		M\left( \Gamma((b\#\delta_1)^*)\right) M\left( \Gamma((b'\#\delta_1)^*)\right) &=\sum_{\xi,\psi \in \widehat{G}} \kappa(\xi|\bar{b}|,|b|)^{-1}\kappa(\psi|\bar{b'}|,|b'|)^{-1} \\
		&\times \,
		 \Gamma((S(b)\#\delta_1)^*)\delta_{\xi}\Gamma((S(b')\#\delta_1)^*)\delta_{\psi}\\
		&=\sum_{\psi \in \widehat{G}} \kappa(\psi|\bar{b}'||\bar{b}|,|b|)^{-1}\kappa(\psi|\bar{b'}|,|b'|)^{-1}  \, \Gamma((S(b)*S(b')\#\delta_1)^*)\delta_{\psi}.
	\end{align*}
	Since $\pi^*\bar{\omega}=d\kappa^{-1}$, we have an equality. For the second relation, it suffices to prove $M(\bar{F}_j\bar{E}_i)=M(\bar{F}_j)M(\bar{E}_i)$, where the bars indicate that the generators live in $\bar{u}$. By Remark \ref{rem: another commutator relation}, we have
	\begin{align*}
		M\left( \bar{F}_j\bar{E}_i\right) &=M\left( \left(\bar{E}_i\bar{F}_j-\delta_{ij}\left( K_{[\bar{\chi}_i]}-\bar{K}_{[\bar{\chi}_i]}^{-1}\right)  \right)\iota\left( \sum_{\xi \in \widehat{G}/T} \, \bar{\omega}([\bar{\chi}_j],\xi[\chi_i],[\bar{\chi}_i]) \delta_{\xi} \right)  \right) \\
		&=\left(M(\bar{E}_i)M(\bar{F}_j)-\delta_{ij}\left( M\left( K_{[\bar{\chi}_i]}\right) -M\left( \bar{K}_{[\bar{\chi}_i]}^{-1}\right) \right)  \right)\iota\left( \sum_{\xi \in \widehat{G}} \, \bar{\omega}([\bar{\chi}_j],[\xi\chi_i],[\bar{\chi}_i])\, \delta_{\xi} \right) \\
		&=\left(E_iF_j-\delta_{ij}\left(  K_{\bar{\chi}_i} -\bar{K}_{\bar{\chi}_i}^{-1} \right)  \right)\iota\left( \sum_{\xi \in \widehat{G}} \, \frac{\kappa(\bar{\chi}_j,\xi)}{\kappa(\xi\bar{\chi}_j\chi_i,\bar{\chi}_i)}\bar{\omega}([\bar{\chi}_j],[\xi\chi_i],[\bar{\chi}_i]) \,\delta_{\xi} \right) \\
		&=\left(E_iF_j-\delta_{ij}\left(  K_{\bar{\chi}_i} -\bar{K}_{\bar{\chi}_i}^{-1} \right)  \right)\iota\left( \sum_{\xi \in \widehat{G}} \, \frac{\kappa(\bar{\chi}_j,\xi\chi_i)}{\kappa(\xi\chi_i,\bar{\chi}_i)}\,\delta_{\xi} \right) \\
		&=F_jE_i\,\iota\left( \sum_{\xi \in \widehat{G}} \, \frac{\kappa(\bar{\chi}_j,\xi\chi_i)}{\kappa(\xi\chi_i,\bar{\chi}_i)}\,\delta_{\xi} \right) =M(\bar{F}_j)M(\bar{E}_i).
	\end{align*}
	Finally, it is easy to see that $M$ preserves the unit. This proves the claim.
\end{proof}

The following proposition is a tedious but straight forward calculation:

\begin{proposition}\label{rm: M is quasi-Hopf homomorphism to u^J }
	The algebra homomorphism $M$ from Prop. \ref{prop: algebra inclusion M} is a homomorphism of quasitriangular quasi-Hopf algebras if we replace $u$ by its twisted version $u^J$ (see Def. \ref{def: twist}), where 
	\begin{align*}
		J=\sum_{\chi \in \widehat{G}} \, \kappa(\chi,\psi)\,\delta_{\chi}\otimes \delta_{\psi}.
	\end{align*}
\end{proposition}

\begin{lemma}\label{lm: radical of hopf pairing for u(omega,sigma)}
	Let $\mathcal{C}$ denote the representation category of $u(1,\sigma)$. Then the M\"uger center $\mathcal{C}'$ is equivalent to $\mathsf{Vect}_T$, where $T=Rad(B)$ is the radical of the associated bihomomorphism $B=\sigma \sigma^T$.
\end{lemma}

\begin{proof}
	In the case $u(1,\sigma)=(u_q(\g),R_0)$ with simple Lie algebra $\g$ and $R_0$ coming from a symmetric non-degenerate bihomomorphism $f$, the lemma is a direct consequence of \cite[Cor. 5.20]{LentnerOhrmann2017}. The general case $u(1,\sigma)$ for an arbitrary non-degenerate bihomomorphism $f$ follows from Thm. 6.2. in \cite{Shimizu2016}. The theorem says that the M\"uger center $\mathcal{D}'$ of a finite braided category $\mathcal{D}$ is equivalent to the M\"uger center $({\mathcal{YD}(\mathcal{D})^A_A})'$ of the category $\mathcal{YD}(\mathcal{C})^A_A$ of Yetter-Drinfeld modules over a braided Hopf algebra $A$ in $\mathcal{D}$. In our case, we can set $\mathcal{D}=\mathsf{Vect}_{\hat{G}}^{(1,\sigma)}$, $A=B(V)$. As it is pointed out in Chapter 6.5 in \cite{Shimizu2016}, we have ${\mathcal{YD}(\mathcal{D})^A_A} \cong \mathsf{Rep}_u$ in this case. Moreover, by Prop. \ref{Modularizability of VectG} we have $\mathcal{D}'\cong \mathsf{Vect}_T$, which proves the claim.
\end{proof}

We now proof the main theorem:
\begin{proof}[Proof of Thm. \ref{thm:modularzation}]
	We first prove the second part of the theorem. Since the restriction functor is the identity on morphisms, it is additive, linear and even exact. In order to show that $\tau$ is a monoidal structure, we choose $u\otimes v \otimes w \in (\mathcal{F}(U)\otimes \mathcal{F}(V))\otimes \mathcal{F}(W)$. For $\bar{\alpha}_{X,Y,Z}$ and $\alpha_{U,V,W}$ denoting the associators in $\mathsf{Rep}_{\bar{u}}$ and $\mathsf{Rep}_u$, respectively, we obtain
	\begin{align*}
		&\tau_{U,V\otimes W}\circ (\id_{\mathcal{F}(U)}\otimes \tau_{V,W})\circ \bar{\alpha}_{\mathcal{F}(U),\mathcal{F}(V),\mathcal{F}(W)}(u\otimes v\otimes w)\\
		&=\sum_{\xi_k \in \widehat{G}/T}\sum_{\psi_l \in \widehat{G}}\, \kappa(\psi_1,\psi_2\psi_3)\kappa(\psi_2,\psi_3)\bar{\omega}(\xi_1,\xi_2,\xi_3)\,\delta_{\psi_1}M(\delta_{\xi_1}).u \otimes \delta_{\psi_2}M(\delta_{\xi_2}).v \otimes \delta_{\psi_3}M(\delta_{\xi_3}).w \\
		&=\sum_{\psi_l \in \widehat{G}}\, \kappa(\psi_1,\psi_2\psi_3)\kappa(\psi_2,\psi_3)\bar{\omega}([\psi_1],[\psi_2],[\psi_3])\,\delta_{\psi_1}.u \otimes \delta_{\psi_2}.v \otimes \delta_{\psi_3}.w \\
		&=\sum_{\psi_l \in \widehat{G}}\, \kappa(\psi_1\psi_2,\psi_3)\kappa(\psi_1,\psi_2)\,\delta_{\psi_1}.u \otimes \delta_{\psi_2}.v \otimes \delta_{\psi_3}.w \\
		&=\mathcal{F}(\alpha_{U,V,W})\circ \tau_{U\otimes V,W} \circ ( \tau_{U,V}\otimes \id_{\mathcal{F}(W)}).
	\end{align*} 
	This proves the associativity axiom. The unitality axiom follows from the fact that $\kappa$ is a normalized $2$-cochain. Hence, $\tau$ is a monoidal structure. We now show that $(\mathcal{F},\tau)$ preserves the braiding. For this, we first notice that
	\begin{align*}
		&(M\otimes M)(\bar{\Theta})\\
		&=\sum_{b \in \mathcal{B}}\, \left( M\left( \Gamma((b\#\delta_1)^*)\right) \otimes M\left( \iota(b)\right)\right) \cdot \left( \sum_{\chi,\psi \in \widehat{G}/T}\, \bar{\omega}(\chi[|\bar{b}|],[|b|],\psi)\,M(\delta_{\chi})\otimes M(\delta_{\psi})\right) \\
		&=\sum_{b \in \mathcal{B}}\, \left( \Gamma((S(b)\#\delta_1)^*) \otimes \iota(S(b))\right) \left( \sum_{\pi,\nu \in \widehat{G}}\, \frac{\kappa(|b|,\nu)}{\kappa(\pi|\bar{b}|,|b|)} \,  \bar{\omega}([\pi|\bar{b}|],[|b|],[\psi])\,\delta_\pi\otimes \delta_\nu \right) \\
		&=\sum_{b \in \mathcal{B}}\, \left( \Gamma((S(b)\#\delta_1)^*) \otimes \iota(S(b))\right) \left( \sum_{\pi,\nu \in \widehat{G}}\, \frac{\kappa(\pi,\nu)}{\kappa(\pi|\bar{b}|,\nu|b|)} \,\delta_\pi\otimes \delta_\nu \right) \\
		&=\left( \sum_{\chi,\psi \in \widehat{G}}\, \kappa(\chi,\psi)^{-1} \,\delta_\chi\otimes \delta_\psi \right)\left( \sum_{b \in \mathcal{B}}\, \left( \Gamma((S(b)\#\delta_1)^*) \otimes \iota(S(b))\right)\right)  \left( \sum_{\pi,\nu \in \widehat{G}}\, \kappa(\pi,\nu) \,\delta_\pi\otimes \delta_\nu \right) \\
		&=\left( \sum_{\chi,\psi \in \widehat{G}}\, \kappa(\chi,\psi)^{-1} \,\delta_\chi\otimes \delta_\psi \right)\Theta  \left( \sum_{\pi,\nu \in \widehat{G}}\, \kappa(\pi,\nu) \,\delta_\pi\otimes \delta_\nu \right).
	\end{align*}
	where $\bar{\Theta}=\bar{\Theta}^1\otimes \bar{\Theta}^2$ and $\Theta =\Theta^1 \otimes \Theta^2$ denote the quasi-$R$-matrices in $\bar{u}$ and $u$, as defined in Prop. \ref{prop: R-matrix}. The inverses are denoted by $\bar{\Theta}^{-1}=\bar{\Theta}^{(-1)}\otimes \bar{\Theta}^{(-2)}$ and $\Theta^{-1}=\Theta^{(-1)} \otimes \Theta^{(-2)}$, respectively.
	Now, let $\bar{c}_{X,Y}$ and $c_{V,W}$ denote the braiding in $\mathsf{Rep}_{\bar{u}}$ and $\mathsf{Rep}_u$, respectively. We set $\kappa:=\kappa^1\otimes \kappa^2=\sum_{\chi,\psi}\, \kappa(\chi,\psi)\, \delta_\chi \otimes \delta_\psi$. By $\bar{R}=\bar{R}^1\otimes \bar{R}^2=\bar{R}_0^1\bar{\Theta}^{(-1)} \otimes \bar{R}_0^2\bar{\Theta}^{(-2)}$ and $R=R^1\otimes R^2=R_0^1\Theta^{(-1)} \otimes R_0^2\Theta^{(-2)}$, we denote the $R$-matrices of $\bar{u}$ and $u$ as defined in Section \ref{sec:R-matrix}. Again, we choose an arbitrary element $v \otimes w \in \mathcal{F}(V) \otimes \mathcal{F}(W)$. We have
	\begin{align*}
		\tau_{W,V} \circ \bar{c}_{\mathcal{F}(V),\mathcal{F}(W)}(v\otimes w) &= \kappa^1M(\bar{R}^2).w\otimes \kappa^2M(\bar{R}^1).v \\
		&= \kappa^1M(\bar{R}_0^2)M(\bar{\Theta}^{(-2)}).w\otimes \kappa^2M(\bar{R}_0^1)M(\bar{\Theta}^{(-1)}).v \\
		&= \kappa^1M(\bar{R}_0^2)\tilde{\kappa}^{(-2)}\Theta^{(-2)}\tilde{\tilde{\kappa}}^{2}.w\otimes \kappa^2M(\bar{R}_0^1)\tilde{\kappa}^{(-1)}\Theta^{(-1)}\tilde{\tilde{\kappa}}^{1}.v \\
		&= R_0^2\Theta^{(-2)}\tilde{\tilde{\kappa}}^{2}.w\otimes R_0^1\Theta^{(-1)}\tilde{\tilde{\kappa}}^{1}.v \\
		&= R^2\tilde{\tilde{\kappa}}^{2}.w\otimes R^1\tilde{\tilde{\kappa}}^{1}.v \\
		&= \mathcal{F}(c_{V,W})\circ \tau_{V,W}(v\otimes w).
	\end{align*}
	Here, we used the notation $\kappa^1\otimes\kappa^2=\tilde{\kappa}^1\otimes\tilde{\kappa}^2=\tilde{\tilde{\kappa}}^1\otimes\tilde{\tilde{\kappa}}^2=\kappa$. The equality $\kappa^1M(\bar{R}_0^2)\tilde{\kappa}^{(-2)} \otimes \kappa^2M(\bar{R}_0^1)\tilde{\kappa}^{(-1)}=R_0^1\otimes R_0^2$ follows from $\pi^*\bar{\sigma}=\sigma\cdot \kappa/\kappa^T$. \\
	The fact that $\xi_V:\mathcal{F}(V^\vee) \to \mathcal{F}(V)^\vee$ is a natural $\bar{u}$-module isomorphism follows from the Remark \ref{rm: M is quasi-Hopf homomorphism to u^J } that $M$ is a homomorphism of quasi-Hopf algebras, if we replace $u$ by the twisted quasi-Hopf algebra $u^J$ for $J=\sum\,\kappa(\chi,\psi)\, \delta_{\chi}\otimes \delta_\psi$.\\
	We now show that the ribbon structure from $\mathsf{Rep}_u$ can be transported to a ribbon structure of $\mathsf{Rep}_{\bar{u}}$ so that $\mathcal{F}$ is a ribbon functor. \\
\newcommand{\fun}{\mathcal{F}}
\newcommand{\ribbon}{\boldsymbol{v}}
\newcommand{\pivot}{\boldsymbol{g}}
We now recall that $u$ is actually a ribbon Hopf algebra with the ribbon element $\ribbon=\pivot^{-1}\boldsymbol{u}$, where  $\boldsymbol{u}$ is the canonical Drinfeld element and $\pivot$ is the balancing group-like element, i.e. such that $S^2(a) = \pivot a \pivot^{-1}$, for $a\in u$. The element $\pivot$ can be constructed as a square root of the special group-like element, see details e.g.\ in~\cite{Radford2012}. 
Let $\theta$ denotes the corresponding twist in $\mathsf{Rep}_u$, in our convention it is given by the action with $\ribbon^{-1}$.
We now show that the ribbon structure from $\mathsf{Rep}_u$ can be transported to a ribbon structure of $\mathsf{Rep}_{\bar{u}}$ so that $\fun$ becomes a ribbon functor. 
By transporting the ribbon twist along $\fun\colon \mathsf{Rep}_u \to \mathsf{Rep}_{\bar{u}}$ we mean giving a natural family $\bar{\theta}_X$ such that for all $V \in \mathsf{Rep}_{u}$, 
\begin{equation}\label{eq:theta-def}
\bar{\theta}_{\fun(V)} = \fun(\theta_V)\ .
\end{equation}
We first recall that a twist, as a natural transformation of the identity functor, is uniquely determined by its values on projective objects. A proof of this fact is rather standard and is given e.g.\ in~\cite[Lem.\,B.1]{GainutdinovRunkel2017a}. Furthermore, as the twist respects direct sums it is enough to determine its values on a projective generator, e.g.\ on the regular representation $u$. 
We also note that $\fun(u)$  is projective and contains $\bar{u}$ as a direct summand and therefore it is  a projective generator in $\mathsf{Rep}_{\bar{u}}$. Therefore, a solution to~\eqref{eq:theta-def} is uniquely fixed by  $\bar{\theta}_{\fun(u)} = \fun(\theta_u)$.
Since $\fun$ is a restriction functor, it is faithful and such a family $\bar{\theta}$ indeed exists, while  uniqueness was shown above. Due to naturality of $\theta$, we also have that the family $\bar{\theta}_X$ is natural. 
By construction and the fact that $\fun$ is a braided functor, $\bar{\theta}$ satisfies the twist condition $\bar{\theta}_{X\otimes Y} = \bar{c}_{X,Y}\circ \bigl(\bar{\theta}_X \otimes \bar{\theta}_Y\bigr)$, and this automatically makes the functor $\fun$ ribbon. Therefore, the family $\bar{\theta}_X$ turns $\mathsf{Rep}_{\bar{u}}$ into a ribbon category (with braided monoidal structure as above).
Finally,  the corresponding  ribbon element $\bar{\ribbon}\in\bar{u}$ can be calculated as the inverse twist on the regular representation: $\bar{\ribbon}=(\bar{\theta}_{\bar{u}})^{-1}(1)$.
We however do not pursue an explicit calculation of this central element.

From Prop. \ref{rm: M is quasi-Hopf homomorphism to u^J }, we know that we have an exact sequence of quasi-Hopf algebras $\bar{u} \to u^J \to k^T$. By \cite[Prop.\,2.9]{BruguieresNatale2011}, this gives rise to an exact sequence of tensor categories $\mathsf{Rep}_{k^T} \to \mathsf{Rep}_{u^J} \to \mathsf{Rep}_{\bar{u}}$ in the sense of \cite{BruguieresNatale2011}. The composition of the second functor with the canonical equivalence $\mathsf{Rep}_{u} \cong \mathsf{Rep}_{u^J}$ is our functor $\mathcal{F}$. By Lemma \ref{lm: radical of hopf pairing for u(omega,sigma)}, we can identify $\mathsf{Rep}_{k^T}$ with the M\"uger center of $\mathsf{Rep}_{u}$
and hence  the radical of the pairing $\omega_{\mathsf{Rep}_u}$ is isomorphic to $kT$, i.e.\ a direct sum of all irreducible transparent $u$-modules. Since $\mathcal{F}$ preserves the Hopf pairings and thus maps 
the radical of $\omega_{\mathsf{Rep}_u}$ to the radical of $\omega_{\mathsf{Rep}_{\bar{u}}}$ and $\mathcal{F}(kT)$ is trivial, $\mathsf{Rep}_{\bar{u}}$ has the trivial radical and so a  modular category. As our functor $\mathcal{F}$ is a restriction functor it satisfies the conditions in Prop. \ref{prop: left adjoint with inverse allows for modularization} and hence it is a modularization in the sense of Def. \ref{def: non ssi modularization}.
\end{proof}

	\newcommand{\md}{\mathrm{-}}
\newcommand{\zem}{\mathfrak{Z}}
\renewcommand{\i}{\mathrm{i}}

\section{Application: Quasi-quantum groups and conformal field theory}\label{sec:CFT}
A main reason for us to study the quasi-quantum groups is their conjectured connection to logarithmic conformal field theory which we describe below.

 Let $\g$ be a semisimple finite-dimensional complex Lie algebra of rank $n$ with simple roots $\alpha_i,i=1,\ldots,n$ and Killing form $(\alpha_i,\alpha_j)$. Let $\ell=2p$ be an even natural number, which is  divisible by all $2d_i:=(\alpha_i,\alpha_i)$. Consider the following rescaled coroot lattice of $\g$ and its dual lattice, which is a rescaled weight lattice
$$\Lambda^\oplus:=\Lambda_R^\vee(\g)\sqrt{p}
\qquad\quad (\Lambda^\oplus)^*=\Lambda_W(\g)/\sqrt{p}\ .
$$ 
The divisibility condition ensures that $\Lambda^\oplus$ is an even integral lattice.\footnote{This divisibility condition is also present in other context, for example the unitarity of the modular tensor categories of affine Lie algebras at level $p$.}

To this algebraic data, one associates the so-called \textit{lattice vertex-operator algebra} $\V_{\Lambda^\oplus}$. Vertex algebras \cite{Kac1998,FrenkelBen-Zvi2004} are algebraic structures with an extra layer of analysis and appear in the description of $2$-dimensional conformal field theories. A vertex algebra has an associated category of representations, which under suitable conditions becomes a (modular) braided tensor category \cite{Huang2005}. For the lattice vertex algebra $\V_{\Lambda^\oplus}$, the representation category is a  semisimple modular tensor category with simple modules $\V_{[\lambda]}$ for each coset $[\lambda]\in (\Lambda^{\oplus})^*/\Lambda^{\oplus}$, in particular the unit object is $\V_{[0]}=\V_{\Lambda^\oplus}$ \cite{MooreSeiberg1989}. As a braided tensor category it is equivalent to the  tensor category $\Vect_{(\Lambda^{\oplus})^*/\Lambda^{\oplus}}$ of graded vector spaces with braiding and associator determined by the quadratic form $Q(\lambda)=q^{(\lambda,\lambda)}$ with $2p$-th root of unity $q=e^{\pi\i/p}$ (the ribbon twist being somewhat different, which we ignore in the following). 

Associated to the rescaled simple roots $-\alpha_i/\sqrt{p}$ there are so-called screening operators 
$$
\zem_i=\oint e^{-\alpha_i/\sqrt{p}}
$$
 acting on the space $\V_{(\Lambda^\oplus)^*}$.  The kernel of these operators forms a Vertex-Operator Algebra $\W(\g,p)\subset \V_{\Lambda^\oplus}$ of certain finite type\footnote{In what follows, we often omit the indication of $p$ when it is not necessary.}, so-called $C_2$-cofinite VOA. Its category $\Rep\,\W(\g)$ of representations is a finite braided tensor category~\cite{HuangLepowskyZhang}.  And conjecturally  it is a modular tensor category. There is a certain sequence of observations (discussed more below) suggesting  a relation of this category  to the small quantum group $u_q(\g)$. First of all,  it is proven in  general~\cite{Lentner2017} that the screening operators $\zem_i$ give a representation of the Borel part of the small quantum group $u_q(\g)$ on $\V_{(\Lambda^\oplus)^*}$, and this action suggests a certain equivalence of categories.
However, it is clear that   $\Rep\,\W(\g)$ and $\Rep \,u_q(\g)$  cannot be equivalent as modular tensor categories, because $\Rep \,u_q(\g)$ at an even root of unity $q$ does usually not posses a braiding, see e.g. \cite{KondoSaito2011} in the case of $\g=\sl_2$.  

We now briefly discuss the current state of the conjectural relation to $\Rep \,u_q(\g)$. Initially, \cite{Wakimoto1986,FeuiginFrenkelprime1988,Felder1989} have been able to realize affine Lie algebras as cohomology of screening operators. In \cite{FuchsHwangSemikhatovEtAl2004} for  the case $\g=\sl_2$ one considers  instead the kernel of screenings $\W$, which is a VOA and it gives a non-semisimple  tensor category. Here, $\W=\W(\sl_2)$ is the so-called \textit{triplet W-algebra} \cite{Kausch1991}; it is is proven in \cite{FeiginGainutdinovSemikhatovEtAl2006a,NagatomoTsuchiya2009,AdamovicMilas2008} that in this case the representation category $\Rep\,\W(\sl_2)$ as an abelian category is equivalent to the representation category of $u_q(\sl_2)$ (which is however not a modular tensor category, not even braided). To get the correct modular tensor category, \cite{GainutdinovRunkel2017,CreutzigGainutdinovRunkel2017} have constructed a quasi-Hopf algebra variant $\bar{u}_q(\sl_2)$ of the small quantum group, which has a \textsl{modular} tensor category of representations, and this category is highly expected to be equivalent to $\Rep\,\W(\sl_2)$. A second class of examples appearing in the literature is the case of root system $B_n$ for $\ell=4$, where the corresponding VOA is identified with the even part of the super VOA of $n$ pairs of symplectic fermions \cite{Kausch2000, Abe2007}. And here again, there is a quasi-Hopf algebra $u^{(\Phi)}$ at a fourth root of unity $q$ (at this small value there are certain degenerations) which is highly expected to give the correct modular tensor category \cite{GainutdinovRunkel2017,FarsadGainutdinovRunkel2017a,FlandoliLentner2017}. For arbitrary $\g$ the conjectures for abelian categories are formulated in \cite{FeiginTipunin2010,AdamovicMilas2014,FlandoliLentner2017}. 

One of the applications of our results in the previous sections  is a contruction, for every $(\g,p)$ with divisibility $d_i\mid p$, of a family of modular tensor categories as representations of quasi-quantum groups $\bar{u}_q(\g)$. And these categories are conjectured to be equivalent, as modular tensor categories, to the representation categories  of $\W(\g)$. 
A proof of the  equivalence is still not in sight, not even an equivalence of abelian categories. In the following, we discuss how our previous results on quasi-quantum groups give suitable modular tensor categories and justify our conjecture. We begin with

\begin{proposition}\label{prop_CFTcase}
	Let again $\g$ be a semisimple finite-dimensional complex Lie algebra of rank $n$ and $\ell=2p$ be an even natural number divisible by all $2d_i:=(\alpha_i,\alpha_i)$ and $q=e^{\pi\i/p}$ a primitive $\ell$-th root of unity. In this case the results of the previous section read as follows:
	\begin{enumerate}[a)]
		\item The quantum group $u_q(\g,\Lambda_W)$ associated in Sec.~\ref{Quantum groups uq(g,Lambda) and R-matrices} to the full weight lattice $\Lambda_W$, is a Hopf algebra with coradical $G=\Lambda_W/\Lambda'$ for $\Lambda'=2p\Lambda_{R}^\vee$, where $\Lambda_R^\vee$ denotes the co-root lattice, i.e. the $\Z$-span of the co-roots $\alpha_i^\vee:=\alpha_i/d_i$.
		
		\item This quantum group has an $R$-matrix associated to the trivial bimultiplicitive form $g(\bar{\lambda},\bar{\mu})=1$, so the $R_0$-matrix is defined by $f(\lambda,\mu)=q^{-(\lambda,\mu)}$.
		\item The braided tensor category $\mathcal{C}=\Rep\,u_q(\g,\Lambda_W)$ has the largest possible group of transparent objects $T\cong \Z_2^n$. More precisely the transparent objects are 
		$$\C_{q^{(\alpha,-)}}, \quad \alpha \in p\Lambda_{R}^\vee/\Lambda'\ .$$ 
		The respective dual sublattice $\bar{\Lambda}\subset\Lambda$ defined in App. \ref{sec: group theoretic part} is 
		$\bar{\Lambda}=2\Lambda_W$.
		\item $\Rep\, u_q(\g,\Lambda_W)$ is modularizable.
	\end{enumerate}
\end{proposition}

\begin{proof}~
	\begin{enumerate}[a)]
		\item The coradical $G=\Lambda_W/\Lambda'$ can be calculated from the centralizer formula \eqref{eq: centralizer formula}. The result is particularly simple due to the choice $\Lambda=\Lambda_W$ and the  divisibility condition in place, which ensures $2p\alpha_i^\vee\in\Lambda_R$:
		\begin{align*}
		\Lambda'&=\Cent_{\Lambda_R}(\Lambda_W)=\{\alpha\in \Lambda_R\mid  q^{(\alpha,\lambda)}=1\;\forall \lambda\in\Lambda_W\}
		=\left\langle 2p\alpha_i^\vee,\; i=1,\ldots, n  \right\rangle_\Z\ .
		\end{align*}
		\item If we set $g(\bar{\lambda},\bar{\mu})=1$ then the $R_0$-matrix is defined by $f(\lambda,\mu)=q^{-(\lambda,\mu)}$ as asserted. Then we can see directly from a) that this group pairing on $G=\Lambda/\Lambda'$ is nondegenerate, so we indeed get an $R$-matrix by Thm. \ref{thm: thm from LO}.
		\item By \cite[Cor.\,5.20]{LentnerOhrmann2017} all transparent objects are $1$-dimensional $\C_\chi$ with weights $\chi:\Lambda/\Lambda'\to\CC^\times$ the $f$-transform of the $2$-torsion subgroup of $\Lambda/\Lambda'$. The lattice $\Lambda_W$ contains the lattice $p\Lambda_{R^\vee}$ with $\Lambda/\Lambda'=\Z_2^n$. Since this is the maximal rank a subgroup of a rank $n$ abelian group can have, this is the $2$-torsion group of $G$ and the dual lattice is $\bar{\Lambda}=2\Lambda_W$ as asserted.
		\item The criterion in Prop. \ref{Modularizability of VectG} for modularizablility is that the quadratic form on $\widehat{G}$ defined by $R_0=\frac{1}{|G|}\sum_{\lambda,\mu\in G} q^{-(\lambda,\mu)}K_\lambda\otimes K_\mu$ is trivial on the group of transparent objects $\C_{q^{(p\alpha^\vee,-)}}$. This again follows from the divisibility condition $d_i|p_i$  
		$$q^{(p\alpha^\vee,p\alpha^\vee)}=q^{p^2\cdot 2d_i/d_i^2}=1\ .$$
	\end{enumerate}
\end{proof}

We can now formulate a precise conjecture (on the level of modular tensor categories) relating the small quasi-quantum groups to VOAs of logarithmic conformal field theories:

\begin{conjecture}
	For $(\g,p)$ as above, the modular tensor category $\bar{\CC}=\Rep\,\bar{u}_q(\g)$ is equivalent as a braided tensor category to the representation category of the VOA 
$\W(\g,p)$.
\end{conjecture}

We now give some indications why the conjecture might be true, that in particular explain why precisely this quasi-quantum group might appear:

\begin{example}
	For $\g=\sl_2$ our definition of $\bar{u}_q(\sl_2)$ is isomorphic (up to a twist) as a quasi-triangular quasi-Hopf algebra to $u^{(\Phi)}_q(\sl_2)$ defined in \cite{CreutzigGainutdinovRunkel2017}. In particular, in this case $\Rep\,\bar{u}_q(\sl_2)$ is equivalent as abelian category to $\Rep\,\W(\sl_2,p)$. It is highly expected that this is even an equivalence of modular tensor categories. For example, for $p=2$ and due to \cite[Conj.\,1.3]{DavydovRunkel2016} this agrees with $\Rep\,\W(\sl_2,2)$ which is the even subalgebra of the super VOA of one pair of symplectic fermions.
\end{example}
\begin{example}
	For $\g$ of type $B_n$ at $p=2$ it is shown in \cite[Sec.\,7.5]{FlandoliLentner2017} that the kernel of screenings  $\W(B_n,2)$ is as a vertex algebra isomorphic to the even part of the VOA of $n$ pairs of symplectic fermions. It contains $n$ copies of $\W(\sl_2,p=2)$ and has the derivation symmetry algebra given by the dual Lie algebra $\mathfrak{sp}_{2n}$. This degeneracy (a factorization on $n$ copies) and duality is also visible on the quantum group side, see~\cite[Sec.\,8]{FlandoliLentner2017}.  The algebra $\W(B_n,2)$ is known to be $C_2$-cofinite and thus its representation category $\Rep\,\W(B_n,2)$ is a finite braided tensor category.
	Following conjecture~\cite[Conj.\,1.3]{DavydovRunkel2016}, $\Rep\,\W(B_n,2)$ is expected to be equivalent to the braided tensor category $\mathcal{SF}_n$  constructed in~\cite{DR13,Runkel14}, which was also proven to be modular,
 and in~\cite{FarsadGainutdinovRunkel2017a} this modular tensor category is realized as representations of a ribbon quasi-Hopf algebra~$u^{(\Phi)}$.
This algebra contains $n$ copies of  $\bar{u}_q(\sl_2)$ from the previous example at $p=2$.\\

	We claim that our construction $\bar{u}_q(\g)$, where  $\g$ is for the  root system $B_n$ and  $\ell=2p=4$ in the setting of Proposition \ref{prop_CFTcase}, reproduces the family of quasi-triangular quasi-Hopf algebras~$u^{(\Phi)}$ from~\cite{FarsadGainutdinovRunkel2017a}. Hence, the modular tensor category $\Rep\,\bar{u}_q(\g)$ we construct in the present article is in this case equivalent to the modular tensor category $\mathcal{SF}_n$, and thus conjecturally to $\Rep\,\W(B_n,2)$.
	
	We start by applying the results of Proposition \ref{prop_CFTcase} to the case $\g=B_n,\;\ell=2p=4$.
	We first calculate the respective groups explicitly:
	\begin{align*}
	G
	=\Lambda_W(B_n)/4\Lambda_{R^\vee}(B_n)
	&=\langle \lambda_1,\ldots,\lambda_n \rangle_\Z / \langle 2\alpha_1,2\alpha_2\ldots,4\alpha_n \rangle_\Z \\
	\bar{G}
	= \Lambda_W(B_n)/2\Lambda_{R^\vee}(B_n) 
	&=\langle \lambda_1,\ldots,\lambda_n \rangle_\Z / \langle \alpha_1,\alpha_2\ldots,2\alpha_n \rangle_\Z \\	
	&=\begin{cases}
	   \langle \lambda_n \mid 4\lambda_n=0\rangle_\Z,\qquad & \text{for $n$ odd} \\
	   \langle \lambda_n,\alpha_n \mid 2\lambda_n=2\alpha_n=0\rangle_\Z,\qquad & \text{for $n$ even} \\ 
	  \end{cases}
	\end{align*}
which is isomorphic to the subgroup $\bar{\Lambda}/\Lambda'=2\Lambda_W(B_n)/4\Lambda_{R^\vee}(B_n)$ of $G$. Here, $\alpha_n$ is the short simple root of $B_n$ and $\lambda_n=\frac{1}{2}(\alpha_1+2\alpha_2+\cdots+n\alpha_n)$ is the respective fundamental weight. 

Proposition \ref{prop_CFTcase} ensures that  $f(\lambda,\mu):=q^{-(\lambda,\mu)}$ defines a non-degenerate pairing $G \times G \to \C^\times$, hence by Thm. \ref{thm: thm from LO} $u_q(\g)$ is quasi-triangular. Moreover, we can parametrize characters $\chi\in\hat{G}$ by $q^{(\lambda,-)}$ for $\lambda\in\Lambda_W/4\Lambda_R$ and thus we also parametrize the simple objects in $\Rep\,\C[G]\cong \Rep\,\C^{\hat{G}}$ by $\C_\psi=\C_{q^{(\lambda,-)}}$. Accordingly, we can parametrize characters $\chi\in \hat{\bar{G}}$ by $q^{(\lambda,-)}$ for $\lambda\in\Lambda_W/2\Lambda_{R^\vee}$ and thus we also parametrize simple objects in $\Rep\,\C[\bar{G}]\cong \Rep\,\C^{\hat{\bar{G}}}$ by $\C_\psi=\C_{q^{(\lambda,-)}}$. Here, $q$ is a primitive $4$th root of unity. In our case we denote the elements in $\hat{\bar{G}}$ by
\begin{align*}
 \psi^j,\quad j\in\{0,\ldots,3\},\qquad &\psi:=q^{-(\lambda_n,-)}\\
 \psi^j\chi^k,\quad j,k\in\{0,1\},\qquad &\psi:=q^{-(\lambda_n,-)},\;\chi:=q^{-(\alpha_n,-)}
\end{align*}
depending on the parity of the rank $n$ of $B_n$. In the first case we also define $\chi:=q^{-(\alpha_n,-)}=\psi^2$.
The quadratic form on $\hat{G}$ that factors to $\hat{\bar{G}}$ (in accordance with Proposition \ref{prop_CFTcase} is $Q(q^{\lambda,-})=q^{-(\lambda,\lambda)}$, or explicitly:
\begin{align*}
 Q(\psi^j)&=q^{-(j\lambda_n,j\lambda_n)}=(q^{n/2})^{(-j^{2})}\\
 Q(\psi^j\chi^k)&=q^{-((j\lambda_n+k\alpha_n,j\lambda_n+k\alpha_n))}
 =(q^{n/2})^{(-j^{2})}(q^{2})^{(-k^{2})}q^{(-jk)}
\end{align*}
depending on the parity of $n$, and the associated bilinear forms are as follows 
(in accordance with our results they are nondegenerate):
\begin{align*}
 B(\psi,\psi)&=q^{-n},\qquad \\
 B(\psi,\psi)&=q^{-n}=(-1)^{n/2},\quad B(\chi,\chi)=+1,\quad B(\psi,\chi)=q^{-2}=-1.
 \end{align*}
 Now we can choose any abelian $3$-cocycle $(\sigma,\omega)$ on $\hat{\bar{G}}$ associated to $Q$. For each choice, this defines the semisimple modular tensor category $\Rep\,\C^{\hat{\bar{G}},\sigma,\omega}$ with $4$ simple objects.
  Then results of Section 5 define a factorizable quasi-Hopf algebra $\bar{u}_q(\g)$ with coradical $\C^{\hat{\bar{G}}}$, and Section 6 states that $\bar{u}_q(\g)$ is a subalgebra of the quasi-triangular Hopf algebra ${u}_q(\g)$ with coradical $\C^{\hat{G}}$, which arises by modularization. Different choices of $(\sigma,\omega)$ define, of course, twist-equivalent quasi-Hopf algebras.\\
  
We claim that for a suitable choice of $(\sigma,\omega)$ the $\bar{u}_q(\g)$  is isomorphic to the factorizable quasi-Hopf algebra in \cite[Sec.3.1]{FarsadGainutdinovRunkel2017a}, and we from now on use their notation. Set $\beta=q^{-n/2}$ (technically, there are four choices), so $\beta^4=(-1)^n$, and let $q=\mathrm{i}$ be a primitive $4$th root of unity. Pick the following explicit $2$-cochain $\sigma$ on the group $\hat{\bar{G}}$ with explicit elements $1,\psi,\psi^2,\psi^3$ respectively $1,\psi,\chi,\psi\chi$:
$$\sigma:=\begin{pmatrix}
   1 & 1 & 1 & 1 \\
   1 & \beta & 1 & \beta \\
   1 & -1 & -1 & 1 \\
   1 & -\beta & -1 & \beta
  \end{pmatrix},
$$
then the assocate quadratic form takes the diagonal values $(1,\beta,-1,\beta)$, which coincides with the value of the quadratic form above, namely $(1,q^{-n/2},-1,q^{-n/2})$ for both parities of $n$. We remark that this is not a nice 2-cocycle in the sense of Def.~\ref{def: nice 3-cocycle}.

We claim that for this choice the quasi-triangular quasi-Hopf algebra $\C^{\hat{\bar{G}},\sigma,\omega}$ is isomorphic to the $4$-dimensional quasi-Hopf subalgebra in~\cite[Sec.3.1]{FarsadGainutdinovRunkel2017a} generated by $K$ with 
\begin{align*}
 K^4&=1\\
 \Delta(K)&=K\otimes K-(1+(-1)^n)(e_1\otimes e_1)(K\otimes K)
 \end{align*}
 and with the co-associator (recall that our conventions for coassociators are opposite, i.e.\ $\phi=\Phi^{-1}$ in notations of~\cite{FarsadGainutdinovRunkel2017a})
 \begin{equation*}
 \phi=1\otimes 1\otimes 1+e_1\otimes e_1\otimes \bigl((K^n-1)e_0+(\beta^2(-\mathrm{i}K)^n-1)e_1\bigr)
  \end{equation*}
  where we use the idempotents $e_0 = (1+K^2)/2$ and $e_1=(1-K^2)/2$.
 This algebra has also a quasi-triangular structure:
 \begin{equation*}
 R_0=\sum_{n,m\in\{0,1\}} \rho_{n,m}(e_n\otimes e_m)\ ,
  \end{equation*}
  with
  \begin{equation*}
\rho_{n,m} = \frac{1}{2} \sum_{i,j=0}^1 (-1)^{ij}\mathrm{i}^{-in+jm} K^i\otimes K^j \ ,
\qquad n,m \in \{0,1\}\ .
\end{equation*}
The isomorphism of quasi-triangular quasi-Hopf algebras, depending on the parity of $n$, is as follows:
 \begin{align*}
I\colon\;\delta_{\psi^j} 
&\mapsto \frac{1}{4}\sum_{l=0}^3 q^{-jl}K^l \ ,\\
I\colon\;\delta_{\psi^j\chi^k} 
&\mapsto \frac{1}{4}\sum_{k=0}^3 q^{-(j+2k)l} K^l\ .
\end{align*}
\begin{proof}
The respective images are the primitive idempotents of the algebra $\C[K]/(K^4)$, so this is clearly an algebra isomorphism. Moreover the simple $\C^{\hat{\bar{G}}}$-representations  
$\C_1,\C_\psi,\C_{\psi^2},\C_{\psi^3}$ respectively $\C_1,\C_{\psi},\C_{\chi},\C_{\psi\chi}$ are mapped to the simple $\C[K]/(K^4)$-representations $\C_{q^k}$ where $K$ acts by the scalar $q^k,\;k=0,\ldots 3$.

We check that this is a coalgebra morphism, where $\C^{\hat{\bar{G}}}$ has the familiar coproduct of the dual group ring. For $n$ odd this is clear, since $\hat{\bar{G}}=\Z_4$ and the element $K$ is grouplike. For $n$ even we have  $\hat{\bar{G}}=\Z_2\times\Z_2$ and we check that this agrees with the modified coproduct of $K$:
\begin{align*}
 &\Delta(I(\delta_{\psi^j\chi^k}))\\
 &=\frac{1}{4}\sum_{l=0}^3  q^{-(j+2k)l} K^l\otimes K^l \left(1-2 e_1 \otimes e_1\right)^l \\
 &=\frac{1}{4}\sum_{l=0}^3  q^{-(j+2k)l} K^l\otimes K^l \left(1+((-1)^l-1)(e_1 \otimes e_1)\right)\\
 &=\frac{1}{4}\sum_{l=0}^3  q^{-(j+2k)l} K^l\otimes K^l 
 -\frac{2}{4}\frac{1}{4}\sum_{l=1,3}  q^{-(j+2k)l} K^l\otimes K^l \left(1\otimes 1+K^2\otimes 1+1\otimes K^2+K^2\otimes K^2\right)\\
&=\frac{1}{4}\sum_{l=0}^3  q^{-(j+2k)l} K^l\otimes K^l 
 -\delta_{j=0}\frac{1}{4}\sum_{l'\neq l''\in\{1,3\}}  q^{-(j+2k)l''} K^{l'}\otimes K^{l''}\\
 &\sum_{j'+j''=j}\sum_{k'+k''=k}\Delta(I(\delta_{\psi^{j'}\chi^{k'}}))\otimes \Delta(I(\delta_{\psi^{j''}\chi^{k''}}))\\
 &=\frac{1}{4}\frac{1}{4} \sum_{l',l''=0}^3 K^{l'} \otimes  K^{l''} \left(\sum_{j'+j''=j}\sum_{k'+k''=k} q^{-(j'+2k')l'-(j''+2k'')l''} \right)\\
 &=\frac{1}{4}\frac{1}{4} \sum_{l',l''=0}^3 K^{l'} \otimes  K^{l''} \left(\sum_{j'}q^{-j'l'-(j-j'+2\delta_{j=0,j'=1})l''}\sum_{k'} q^{-2k'l'-2(k-k'+2\delta_{k=0,k'=1})l''} \right)\\
 &= \delta_{j=1} \frac{1}{4} \sum_{l=0}^3 q^{-(j+2k)l} K^{l} \otimes  K^{l} 
 +\delta_{j=0}\frac{1}{4}\sum_{l',l''=0}^3 K^{l'} \otimes  K^{l''}\left(q^{-(j+2k)l''}\delta_{2|l'+(l''-l')/2}\delta_{2|l'-l''} \right)\\
 \end{align*}
and the resp. last expressions agree. We also check that this isomorphism preserves the $R$-matrices
$$
I\bigl(\sum_{\psi',\psi''}\sigma(\psi',\psi'')\bigr)=R_0
$$
This is a short calculation (independent of the parity of $n$) of the action of $R_0$ on all $16$ pairs of simple modules of $\C_{k'}\otimes\C_{k''}$, and this is in fact how the table for $\sigma$ above was found. 
\end{proof}

We now consider the Nichols algebra generated by the \emph{short} root vectors $F_{\alpha_n+\cdots+\alpha_k}$ -- other root vectors are not in $\bar{u}_q(B_n)$ due to the degeneracy discussed above, see also~\cite[Sec.\,8]{FlandoliLentner2017}.
 The coaction of $\C^{\hat{\bar{G}}}$ is
$$F_{\alpha_n+\cdots+\alpha_k}\mapsto K_{q^{(\alpha_n+\cdots+\alpha_k,-)}}\otimes F_{\alpha_n+\cdots+\alpha_k}
=K_\chi \otimes F_{\alpha_n+\cdots+\alpha_k}\ .$$
The action is given by $\sigma$ and in particular the braiding is 
$$F_{\alpha_n+\cdots+\alpha_{k'}}\otimes F_{\alpha_n+\cdots+\alpha_{k''}}
\mapsto \sigma(\chi,\chi) \cdot F_{\alpha_n+\cdots+\alpha_{k''}}\otimes F_{\alpha_n+\cdots+\alpha_{k'}}\ .
\qquad  \sigma(\chi,\chi)=-1$$
Hence the Nichols algebra relations  match the relations in \cite[Sec.3.1]{FarsadGainutdinovRunkel2017a} as well:
$$F_{\alpha_n+\cdots+\alpha_k}^2=0\qquad
F_{\alpha_n+\cdots+\alpha_{k'}}F_{\alpha_n+\cdots+\alpha_{k''}}+F_{\alpha_n+\cdots+\alpha_{k''}}F_{\alpha_n+\cdots+\alpha_{k'}}=0\ .
$$
Thus we can extend the isomorphism to the Radford biproduct by sending $F_{\alpha_n+\cdots+\alpha_k}$ to an appropriate multiple of the fermionic generator $\mathsf{f}^+\in u^{(\Phi)}$, and finally to an isomorphism 
$I:\;\bar{u}_q(B_n)\to u^{(\Phi)}$. We omit the details on these calculations for brevity.
\end{example}

	\appendix	
	\section{Some tools in finite abelian groups}\label{sec: group theoretic part}

We start with a small quantum group $u:=u_q(\g,\Lambda)$ with $R$-matrix $R=R_0(f)\bar{\Theta}$ as in Section 4.
The Cartan part is given by $u_0=\C[\Lambda/\Lambda']$, where $\Lambda'=\Cent_{\Lambda_R}(\Lambda)$. The non-degenerate bihomomorphism $f:G_1 \times G_2 \to \C^\times$ defines a braiding on $\Vect_{\widehat{G}}$ given by:
\begin{align*}
	\sigma(\chi, \psi) := \chi|_{G_1}\left( f^{-1}\left( \bar{\psi}|_{G_2}\right) \right).
\end{align*}
From now on, we assume $G=G_1+G_2$. We set $Rad_0:=Rad\left( f\cdot f^T|_{G_1 \cap G_2} \right) \subseteq G_1\cap G_2 =:G_{12} $ and $T:=Rad\left( \Beta \right) \subseteq \widehat{G} $.

\begin{lemma}\label{lm: T= ad0}
	The following map is an isomorphism:
	\begin{align*}
		\Phi: Rad_0 & \longrightarrow T \\
		\mu & \longmapsto \left( \nu=\nu_1 + \nu_2 \mapsto \frac{f(\nu_1,\mu)}{f(\mu,\nu_2)} \right).
	\end{align*}
\end{lemma}

\begin{proof}
	By definition of $T$ the map $\Phi$ is well-defined. It is injective, since $\Phi(\mu)=1$ implies $f(\nu_1,\mu)=1$ for all $\nu_1 \in G_1$ and $f(\mu,\nu_2)=1$ for all $\nu_2 \in G_2$. By the non-degeneracy of $f$, we have $\mu=0$.\\
	Finally, we show that $\Phi$ is surjective. For $\chi \in T$, we have elements $\mu_1 \in G_1$, $\mu_2 \in G_2$, s.t. $\chi|_{G_2}=f(\mu_1, \, \_\,)$ and $\chi|_{G_1}=f( \, \_\,, \mu_2)$ by non-degeneracy of $f$.
	Since $\chi$ is in the radical of $\sigma \cdot \sigma^T$, we have
	\begin{align*}
		\sigma(\chi,\psi)\sigma(\psi,\chi) = f(f^{-1}(\psi|_{G_2}),\mu_2)\psi|_{G_1}(\mu_1) =  \psi|_{G_2}(\mu_2)\psi|_{G_1}(\mu_1)=1.
	\end{align*}
	Thus, $\psi(\mu_1)=\psi(-\mu_2)$ for all $\psi \in \widehat{G}$ and hence $\mu_1=-\mu_2=:\mu$. This implies
	\begin{align*}
		\chi = \chi|_{G_1}\chi|_{G_2}=f( \, \_\,, -\mu)(\mu, \, \_\,)=\Phi(\mu).
	\end{align*}
\end{proof}

We define two more important groups:
\begin{align*}
	\overline{G}:=Ann(T) \subseteq G \qquad \overline{G_1 \times G_2}:=\{\, (\mu_1,\mu_2) \in G_1 \times G_2 \, | \, f(\mu_1,\nu)=f(\nu,\mu_2) \quad \forall \nu \in G_{12} \,  \}.
\end{align*}

\begin{corollary} \label{cor: isomorphism of exact sequences}
	The isomorphism $\Phi:Rad_0 \to T$ induces an isomorphism of exact sequences:
	\begin{equation*}
		\begin{tikzcd}
		Rad_0 \arrow[r,"\iota",hook] \arrow[d, "\Phi"] & \overline{G_1 \times G_2} \arrow[r,"\pi", twoheadrightarrow] \arrow[d] & \overline{G} \arrow[d,"\Psi"] \\
		T \arrow[r,hook] &\widehat{G} \arrow[r, twoheadrightarrow] & \widehat{G}/T
		\end{tikzcd}
	\end{equation*}
	Here $\iota(\mu)=(\mu,-\mu)$ and $\pi(\mu_1,\mu_2)=\mu_1 + \mu_2$.
\end{corollary}

\begin{proof}
	The map $F:\overline{G_1 \times G_2} \to \widehat{G}$ is given by $(\mu_1,\mu_2) \mapsto (\nu=\nu_1 + \nu_2 \mapsto f(\nu_1,\mu_2)f(\mu_1,\nu_2))$ and does not depend on the splitting $\nu=\nu_1 + \nu_2$.\\
	We show that the map $\pi$ is surjective. Let $\mu = \mu_1 + \mu_2 \in \overline{G}$. We choose a set-theoretic section $\tilde{s}:G_{12}/Rad_0 \to G_{12}$. We can push $f\cdot f^T$ down to a non-degenerate symmetric bihomomorphism $\overline{f \cdot f^T}$ on $G_{12}/Rad_0$. Hence, there must be an element $x \in G_{12}/Rad_0$, s.t.
	\begin{align*}
		\frac{f(\mu_1, \, \_\,)}{f(\,\_\,,\mu_2)}=\overline{f\cdot f^T}(x, \, \_\,)= f\cdot f^T(\tilde{s}(x), \, \_\,)
	\end{align*}
	as characters on $G_{12}/Rad_0$. For $s(\mu)=(s(\mu)_1,s(\mu)_2):=(\mu_1-s(x),\mu_2+s(x)) \in \overline{G_1 \times G_2}$ we obtain $\pi(s(\mu))=\mu$. The map $\Psi$ is the well-defined isomorphism given by $\Psi(\mu)= [ F(s(\mu)) ]$.
\end{proof}

\begin{example}\label{exp:Gbar for sym f}
	Let $G_{12}=G$ and $f$ symmetric and $\mu \in \overline{G}$. We have
	\begin{align*}
		f(s(\mu)_1,\, \_ \,)=f(\, \_ \,,s(\mu)_2) = f(s(\mu)_2,\, \_ \,)
	\end{align*} 
	on $G$ and since $f$ is non-degenerate this implies $s(\mu)_2 = s(\mu)_1=:\tilde{\mu}$. Since $s: \overline{G} \to \overline{G_1 \times G_2}$ is a section, this implies $\mu = 2 \tilde{\mu}$. On the other hand, for $\nu = 2\tilde{\nu} \in 2G$, by Lemma \ref{lm: T= ad0} we have 
	\begin{align*}
		\chi(\nu)=\frac{f(\tilde{\nu},\mu)}{f(\mu,\tilde{\nu})}=\frac{f(\tilde{\nu},\mu)}{f(\tilde{\nu},\mu)}=1 \qquad \forall \chi \in T.
	\end{align*}
	Hence, $\overline{G} = 2G$ in this case.
\end{example}



\subsection*{A particular representative $(\bar{\omega},\bar{\sigma})\in Z^3(\widehat{G}/T)$}

Let $\sigma$ be a bihomomorphism on the dual of a finite abelian group $\widehat{G}$, such that the associated quadratic form $Q(\chi)=\sigma(\chi,\chi)$ vanishes on the radical $T=Rad(\Beta)$. Starting with an arbitrary set-theoretic section $s:\widehat{G}/T \to \widehat{G}$ for the quotient map $\pi: \widehat{G} \to \widehat{G}/T$, we want to define an abelian $3$-cocycle $(\bar{\omega},\bar{\sigma})$, such that 
\begin{align*}
	\pi^*\bar{\omega}=d\kappa^{-1}, \qquad \pi^*\bar{\sigma}=\kappa/\kappa^T.
\end{align*}

Before we define this abelian $3$-cocycle, we notice that $\sigma|_{T}$ is an alternating bihomomorphism and thus we have $\sigma|_{T}:=\eta/\eta^T$ for some $2$-cocycle $\eta:T \times T \to \C^\times$.

\begin{lemma}\label{lm: explicit 3 cocycle omegabar}
	Let $r(\chi,\psi):= s(\chi)s(\psi)s(\chi\psi)^{-1}$ denote the corresponding $2$-cocycle to the set-theoretic section $s:\widehat{G}/T \to \widehat{G}$. Moreover, for $\chi \in \widehat{G}$, we define $\tau_{\chi}:=\chi s(\pi(\chi))^{-1} \in T$. We set $\bar{\sigma}:=s^*\sigma$. Together with
	\begin{align*}
		\bar{\omega}(\chi,\psi,\xi):=\sigma(s(\xi),r(\chi,\psi))df(s(\chi),s(\psi),s(\xi)), \qquad f(\chi,\psi):=\eta(r(\pi(\chi),\pi(\psi)),\tau_{\chi}\tau_{\psi}),
	\end{align*}
	this defines an abelian $3$-cocycle. Explicitly, we have
	\begin{align*}
		df(s(\chi),s(\psi),s(\xi))= \frac{\eta(r(\chi,\psi\xi),r(\psi,\xi))}{\eta(r(\chi\psi,\xi),r(\chi,\psi))}.
	\end{align*}
	The $2$-cocochain $\kappa \in C^2(\widehat{G})$ satisfying $\pi^*(\bar{\omega},\bar{\sigma})=d_{ab}k\cdot (1,\sigma)$ is given by
	\begin{align*}
		\kappa(\chi,\psi)=\left( \sigma(\tau_\chi,\psi)\eta(\tau_\psi,\tau_\chi)f(\chi,\psi)\right) ^{-1}.
	\end{align*}
\end{lemma}

\begin{proof}
	Before we check that $(\bar{\omega},\bar{\sigma})$ is an abelian $3$-cocycle, we show that $\pi^*(\bar{\omega},\bar{\sigma})=d_{ab}k\cdot (1,\sigma)$ holds. We have
	\begin{align*}
		\frac{\kappa(\chi,\psi)}{\kappa(\psi,\chi)}\sigma(\chi,\psi) =\frac{\eta(\tau_\chi,\tau_\psi)}{\sigma(\tau_\chi,\psi)}\frac{\sigma(\tau_\psi,\chi)}{\eta(\tau_\psi,\tau_\chi)} \sigma(\chi,\psi) 
		=\frac{\sigma(\tau_\chi,\tau_\psi)\sigma(\chi,\psi)}{\sigma(\tau_\chi,\psi)\sigma(\chi,\tau_\psi)} 
		= \sigma(\chi\bar{\tau}_{\chi},\psi\bar{\tau}_{\psi})=\bar{\sigma}(\pi(\chi),\pi(\psi)).
	\end{align*}
	Here, we used that $\sigma$ is a bihomomorphism satisfying $\sigma|_T=\eta/\eta^T$ and $\sigma(\tau,\chi)=\sigma(\chi,\tau)^{-1}$ for $\tau \in T$. For $\pi^*\bar{\omega}$, the following relations are not very hard to check:
	\begin{align*}
		df(\chi,\psi,\xi)&=s^*df(\pi(\chi),\pi(\psi),\pi(\xi))\sigma(r(\chi,\psi),\tau_{\xi})dg(\chi,\psi,\xi)^{-1} \\
		\sigma(\xi,r(\pi(\chi),\pi(\psi)))&=\sigma(s([\xi]),r(\pi(\chi),\pi(\psi)))\sigma(\tau_\xi,r(\pi(\chi),\pi(\psi))),
	\end{align*}
	where $g(\chi,\psi)=\eta(\tau_{\psi},\tau_{\chi})$. Thus, we have
	\begin{align*}
		\pi^*\bar{\omega}(\chi,\psi,\xi)=df(\chi,\psi,\xi)\sigma(\xi,r(\pi(\chi),\pi(\psi)))dg(\chi,\psi,\xi).
	\end{align*}
	Hence, $\pi^*\bar{\omega}=d\kappa^{-1}$.\\
	We now want to show that $\bar{\omega}$ as defined above is a $3$-cocycle. For this, we compute $d(s^*df)$ and $dm$, where $m(a,b,c):=\sigma(s(c),r(a,b))$. We start with $d(s^*df)$:
	\begin{align*}
		d(s^*df)(a,b,c,d)&= \frac{s^*df(a,b,c)s^*df(a,bc,d)s^*df(b,c,d)}{s^*df(ab,c,d)s^*df(a,b,cd)} \\
		&=\frac{\eta(r(a ,bc ),r(b ,c ))}{\eta(r(ab ,c ),r(a ,b ))} \frac{\eta(r(a ,bcd ),r(bc ,d ))}{\eta(r( abc,d ),r(a ,bc ))}\frac{\eta(r(b ,cd ),r(c ,d ))}{\eta(r(bc ,d ),r(b ,c ))} \\
		&\times \frac{\eta(r(abc ,d ),r(ab ,c ))}{\eta(r(ab ,cd ),r(c ,d ))}\frac{\eta(r(ab ,cd ),r(a ,b ))}{\eta(r(a ,bcd ),r(b ,cd ))} \\
		&= \frac{\eta(r(ab ,cd ),r(a ,b ))}{\eta(r(ab ,cd ),r(c ,d ))}\frac{\eta(r(ab ,cd )r(a,b ),r(c ,d ))}{\eta(r(ab ,cd )r(c,d),r(a ,b ))} \\
		&=\frac{\eta(r(a ,b ),r(c ,d ))}{\eta(r(c,d ),r(a ,b ))} \\
		&=\sigma(r(a,b),r(c,d)).
	\end{align*}
	Here, we only used the fact that $\eta $ is a $2$-cocycle and that $\sigma|_T=\eta/\eta^T$. On the other hand,
	\begin{align*}
		dm(a,b,c,d) &= \frac{\sigma(s(c),r(a,b))\sigma(s(d),r(a,bc))\sigma(s(d),r(b,c))}{\sigma(s(d),r(ab,c))\sigma(s(cd),r(a,b))} \\
		&=\frac{\sigma(s(c),r(a,b))\sigma(s(d),r(a,bc))\sigma(s(d),r(b,c))}{\sigma(s(d),r(ab,c))\sigma(s(cd),r(a,b))} \frac{\sigma(s(d),r(a,b))}{\sigma(s(d),r(a,b))} \\
		&=\sigma(s(d),dr(a,b,c))\frac{\sigma(s(c),r(a,b)\sigma(s(d),r(a,b)}{\sigma(s(cd),r(a,b)} \\
		&=\sigma(r(c,d),r(a,b))\\
		&=\sigma(r(a,b),r(c,d))^{-1},
	\end{align*}
	where we used that $\sigma$ is a bihomomorphism and $\eta$ is a $2$-cocycle. Combining both results, we see that $\bar{\omega}$ is a $3$-cocycle. We now want to show that $(\bar{\omega},\bar{\sigma})$ satifies the hexagon equations. We have
	\begin{align*}
		\frac{\bar{\omega}(a,b,c)\bar{\omega}(c,a,b)}{\bar{\omega}(a,c,b)} &= \frac{\sigma(s(c),r(a,b))\sigma(s(b),r(c,a))}{\sigma(s(b),r(a,c))} \\
		&\times \frac{\eta(r(a,bc),r(b,c))\eta(r(c,ab),r(a,b))\eta(r(ac,b),r(a,c))}{\eta(r(ab,c),r(a,b))\eta(r(ca,b),r(c,a))\eta(r(a,cb),r(c,b))} \\
		&=\sigma(s(c),r(a,b)) =\sigma(r(a,b),s(c))^{-1}=\frac{\bar{\sigma}(ab,c)}{\bar{\sigma}(a,c)\bar{\sigma}(b,c)}.
	\end{align*}
	The second hexagon equation follows from the fact that $\bar{\sigma}\bar{\sigma}^T$ is a bihomomorphism.
\end{proof}

\begin{remark}\label{remark:Qs coincide for different section}
	Note that for different sections $s,s':\widehat{G}/T \to \widehat{G}$, the associated abelian $3$-cocycles constructed in Lemma \ref{lm: explicit 3 cocycle omegabar} are cohomologous. This can be seen by comparing the associated quadratic forms on $\widehat{G}/T$: Since the quadratic form $Q$ on $\widehat{G}$ vanishes on $T$, they must coincide.
\end{remark}
	\section{Nichols algebras in braided monoidal categories}\label{app: quantum shuffle product}

In this section, we briefly recall the notion of a Nichols algebra in an abelian rigid braided monoidal category (see \cite{BazlovBerenstein2013} for details). Moreover, we give a categorical definition of the quantum shuffle product.\\
Let $V \in \mathcal{C}$ be an object in an abelian braided monoidal category $\mathcal{C}$ with associator $\alpha$ and braiding $c$. We define $V^{n}:=V^{n-1}\otimes V$ with $V^0=\mathbb{I}$. 
The tensor algebra $T(V):=\bigoplus_{i\geq 0}\, V^i$ has a free algebra structure in $\mathcal{C}$ induced by the multiplications $m_{i,n-i}:V^i \otimes V^{n-i} \to V^n$, given by 
\begin{align*}
m_{i,n-i}:= (\alpha_{V^i,V,V}^{-1} \otimes \id_{V^{\otimes(n-(i+2))}}) \circ \dots \circ \alpha_{V^i,V^{(n-(i+1))},V}^{-1}.
\end{align*}	
Let $d_1,d_2:V \to T(V)\otimes T(V)$ be the canonical inclusions and set $\Delta_1:=d_1+d_2:V\to T(V)\otimes T(V)$. Then there is a unique extension $\Delta:T(V) \to T(V)\otimes T(V)$ of $\Delta_1$, s.t. $\Delta$ is an algebra homomorphism in $\mathcal{C}$. Moreover, we define a counit on $T(V)$ by $\epsilon|_\mathbb{I}=id_\mathbb{I}$ and  $\epsilon_{V^n}=0$ for $n \geq 1$. Similarly to the coproduct, we can uniquely extend the map $S_1:=-id_V:V \to V \subseteq T(V)$ to an anti-algebra homomorphism $S:T(V) \to T(V)$, which turns $T(V)$ into a Hopf algebra in $\mathcal{C}$.\\
\\
It is clear that the braid group $B_n$ with generators $\sigma_i$ acts as automorphisms of $V^n$ via
\begin{align*}
\sigma_{i}.(v_1\otimes \dots \otimes v_n)&=\left(A^{i,n}_V \right)^{-1} \circ ((\id_{V^{\otimes (i-1)}}\otimes c_{V,V}) \otimes \id_{V^{\otimes (n-(i+1))}}) \circ A^{i,n}_V(v_1\otimes \dots \otimes v_n), \text{ where} \\
A^{i,n}_V:&=(A^i_V\otimes \id_{V^{\otimes (n-(i+2))}})\circ \dots \circ (A^1_V\otimes \id_{V^{\otimes (n-3)}}) \\
A_V^i:&=(\id_V\otimes A^{i-1})\circ \alpha_{V,V^{\otimes n},V}, \qquad A_V^1:=\alpha_{V,V,V}.
\end{align*} 
Here, the paranthesis of the tensor powers $V^{\otimes n}$ in the indices is understood. Let $\rho:S_n \to B_n$ be the Matsumoto section of the canonical epimorphism $B_n \twoheadrightarrow S_n$. We define the so-called Woronowicz symmetrizer:
\begin{align*}
	Wor(c)_n:=\sum_{\sigma \in S_n}\, \rho(\sigma) \in End(V^n) \qquad Wor(c):= \bigoplus_{n=1}^\infty Wor(c)_n \in End(T(V)).
\end{align*}

\begin{definition}\label{def: Nichols algebra}
	The Nichols algebra $B(V)$ of $V$ in $\mathcal{C}$ is defined as the quotient Hopf algebra $T(V)/ker(Wor(c))$.
\end{definition}

\begin{remark}\label{rem: equiv defs of Nichols algebras}
	The Nichols algebra $B(V)$ has two important equivalent characterizations:
	\begin{itemize}
		\item We can extend the evaluation map $ev_V:V^\vee \otimes V \to \mathbb{I}$ to a unique Hopf pairing $T(V^\vee) \otimes T(V) \to \mathbb{I}$. This Hopf pairing factors through a non-degenerate Hopf pairing $B(V^\vee) \otimes B(V) \to \mathbb{I}$.
		\item It can be shown that $B(V)$ is the unique quotient Hopf algebra of $T(V)$, s.t. $V \subseteq B(V)$ and $ker(\Delta_n-(1\otimes \id_{V^n}+\id_{V^n} \otimes 1))=0$ for $n>1$. 
	\end{itemize}
\end{remark}

In order to define the quantum shuffle-algebra, we define a different Hopf structure on $T(V)$. Similar to the free algebra structure from above, we can endow $T(V)$ with the cofree coalgebra structure. Moreover, we define a multiplication on $T(V)$ as follows:

\begin{definition}[Quasi-quantum shuffle product]	
	 A permutation $\sigma \in S_n$ is called an $i$-shuffle if
	\begin{align*}
		\sigma(1)< \dots < \sigma(i), \qquad \sigma(i+1)< \dots < \sigma(n).
	\end{align*}
	We define a multiplication $\mu_{i,n-i}:V^i \otimes V^{n-i} \to V^n $ by
		\begin{align}
			\mu_{i,n-i}:= \sum_{\sigma:\,i-\text{shuffle}}\, \rho(\sigma).m_{i,n-i}.
		\end{align}
	The induced product on $T(V)$ is denoted by $*:T(V)\otimes T(V) \to T(V)$. We call this the braided shuffle product. In the case $\mathcal{C}={}_H^H\mathcal{YD}$, where $H$ is a quasi-Hopf algebra, we call this the quasi-quantum shuffle product.
\end{definition}

Again, we can define a corresponding unit and antipode uniquely in order to turn $T(V)$ into a Hopf algebra in $\mathcal{C}$, which we now denote by $t(V)$. The Hopf algebras $B(V)$ and $t(V)$ are related as follows:

\begin{proposition}\label{prop: Nichols alg and shuffle algebra}
	The Woronowicz symmetrizer $Wor(c):B(V) \to t(V)$ is a monomorphism of Hopf algebras in $\mathcal{C}$. The image of $Wor(c)$ is simply the Hopf subalgebra of $t(V)$ generated by $V$ as an algebra.
\end{proposition}
	\section{Factorizable quasi-Hopf algebras}\label{app:Factorizability}

In \cite{BulacuTorrecillas2004}, the authors define the notion of factorizability for quasi-Hopf algebras:

\begin{definition}
	Let $(H,R)$ be a quasi-triangular quasi-Hopf algebra and $H^*$ the corresponding coquasi-Hopf algebra. We will say that $H$ is factorizable if the following linear map is bijective:
	\begin{align*}
		Q:H^* &\longrightarrow H \\
		f& \longmapsto f\left( S(X^2_{(2)}\tilde{p}^2)f^1R^2r^1g^1S(q^2)X^3\right) X^1S(X^2_{(1)}\tilde{p}^1)f^2R^1r^2g^2S(q^1).
	\end{align*}
	Here $R=R^1\otimes R^2=r^1\otimes r^2$, $f=f^1\otimes f^2$ and $f^{-1}=g^1\otimes g^2$ (see Sec. \ref{sec: Quasi-Hopf algebras}).
\end{definition}

They also defined braided Hopf structures on $H$ and $H^*$ in the braided monoidal category $\mathsf{Mod}_H$ (see \cite{BulacuTorrecillas2004}, Sec.4). To avoid confusion, they denoted these braided Hopf algebras by $\underline{H}$ and $\underline{H}^*$. They furthermore showed that $Q$ is a braided Hopf algebra homomorphism.\\
In \cite{FarsadGainutdinovRunkel2017}, the authors gave an alternative interpretation for $\underline{H}$, $\underline{H}^*$ and $Q$ by proving the following statements:

\begin{proposition}\label{prop: factorizability in terms of ends/coends}
	Let $H$ be a finite dimensional quasi-triangular quasi-Hopf algebra. Furthermore, let  $\underline{H}, \underline{H}^* \in \mathsf{Mod}_H$ be the braided Hopf algebras as defined in \cite{BulacuTorrecillas2004}, Sec.4. Then we have
	\begin{enumerate}
		\item $\underline{H}$ is the end over the functor $(\_)\otimes(\_)^\vee$ with dinatural transformation given by $\pi_X(h)=h.e_i \otimes e^i$.
		\item $\underline{H}^*$ is the coend over the functor $(\_)^\vee\otimes(\_)$ with dinatural transformation given by $\iota_X(f\otimes x):=f(\_.x)$.
		\item The morphism $Q:\underline{H}^* \to \underline{H}$ is uniquely determined by 
		\begin{align*}
			\pi_Y \circ Q \circ \iota_X = (ev_X \otimes \id_{Y \otimes Y^\vee})\circ (\id \otimes c^2_{X,Y} \otimes \id) \circ (\id_{X^\vee \otimes X} \otimes coev_Y),
		\end{align*}
		where we omitted the associators in $\mathsf{Mod}_H$.
	\end{enumerate}
\end{proposition}

\begin{remark}
	Note that in \cite{BulacuTorrecillas2004} the authors did not assume a ribbon structure on $H$ in order to define the braided Hopf structure on the coend $\underline{H}^*$. However, up to the antipode it is exactly the same Hopf structure as for example defined in \cite{KerlerLyubashenko2001}. One can show that the antipode in \cite{BulacuTorrecillas2004} is uniquely determined by
	\begin{align*}
		S\circ \iota_X = (\iota_{X^\vee}\otimes ev_X)\circ (c^{-1}_{X^{\vee},X^{\vee\vee}\otimes X^\vee}\otimes \id_{X})\circ(coev_{X^\vee}\otimes \id_{X^\vee \otimes X}),
	\end{align*}
	where we again omitted the associators.
\end{remark}

In \cite{BulacuPanaiteVanOystaeyen2000}, the authors give a more general interpretation for the underline $(\underline{\quad})$:

\begin{proposition}\label{prop: underline functor}
	Let $H$ be a quasi-Hopf algebra, $A$ an associative algebra and $f:H \to A$ an algebra homomorphism. Then we can define a new multiplication on $A$ via
	\begin{align*}
		a\underline{\cdot} b:= f(X^1)af(S(x^1X^2)\alpha x^2 X^3_1)bf(S(x^3X^3_2)).
	\end{align*}
	With this multiplication, unit given by $\beta$ and left $H$-action given by $h.a:=f(h_1)af(S(h_2))$, $A$ becomes a left $H$-module algebra, i.e. an algebra in $\mathsf{Rep}_H$, which we denote by $\underline{A}$.
\end{proposition}

It is easy to see that $f:\underline{H} \to \underline{A}$ then becomes a left $H$-module algebra homomorphism. Note that this is exactly our situation in Section \ref{sec: Modularization}, where we have an algebra homomorphism $M:\bar{u} \to u$.\\
Similarly, if $H$ and $A$ are quasitriangular quasi-Hopf algebras and $f:H \to A$ is a homomorphism of such, we can endow $A$ with a braided Hopf algebra structure in $\mathsf{Rep}_H$ with $H$-module algebra structure as above and comultiplication, counit and antipode given by
\begin{align*}
	\underline{\Delta}(a):&= f(x^1X^1)a_1f(g^1S(x^2R^2y^3X^3_2))\otimes x^3R^1.(f(y^1F^2)a_2f(g^2S(y^2X^3_1)))\\
	\underline{\epsilon}(a):&=\epsilon(a)\\
	\underline{S}(a):&=f(X^1R^2p^2)S(f(q^1)(X^2R^1p^1.a)S(f(q^2))f(X^3)).
\end{align*}

Conversely, if $f:A\to H^*$ is a homomorphism of coquasi-Hopf algebras, we can endow $A$ with the structure of a braided Hopf algebra, denoted by $\underline{A}$, in the category of right $H^*$-comodules (see Thm. 3.5 in \cite{BulacuTorrecillas2004}) which can be identified with the category of left $H$-modules $\mathsf{Rep}_H$.

\bibliography{C:/Users/bipol_000/OneDrive/Documents/PhD/Jabref/bibliography}

\begin{thebibliography}{DGNO10}
\expandafter\ifx\csname url\endcsname\relax
  \def\url#1{\texttt{#1}}\fi
\expandafter\ifx\csname doi\endcsname\relax
  \def\doi#1{\burlalt{doi:#1}{http://dx.doi.org/#1}}\fi
\expandafter\ifx\csname urlprefix\endcsname\relax\def\urlprefix{URL }\fi
\expandafter\ifx\csname href\endcsname\relax
  \def\href#1#2{#2}\fi
\expandafter\ifx\csname burlalt\endcsname\relax
  \def\burlalt#1#2{\href{#2}{#1}}\fi

\bibitem[Abe07]{Abe2007}
Toshiyuki Abe.
\newblock A {$\mathbb{Z}_2$}-orbifold model of the symplectic fermionic vertex
  operator superalgebra.
\newblock {\em Math. Z.}, 255(4):755--792, 2007.
\newblock \doi{10.1007/s00209-006-0048-5}.

\bibitem[AC92]{AltschuelerCoste1992}
Daniel Altsch\"uler and Antoine Coste.
\newblock Quasi-quantum groups, knots, three-manifolds, and topological field
  theory.
\newblock {\em Comm. Math. Phys.}, 150(1):83--107, 1992.
\newblock \urlprefix\url{http://projecteuclid.org/euclid.cmp/1104251784}.

\bibitem[AGP14]{AngionoGalindoPereira2014}
Iv\'an Angiono, C\'esar Galindo, and Mariana Pereira.
\newblock De-equivariantization of {H}opf algebras.
\newblock {\em Algebr. Represent. Theory}, 17(1):161--180, 2014.
\newblock \urlprefix\url{https://doi.org/10.1007/s10468-012-9392-9}.

\bibitem[AM08]{AdamovicMilas2008}
Dra\v{z}en Adamovi\'c and Antun Milas.
\newblock On the triplet vertex algebra {$\mathcal{W}(p)$}.
\newblock {\em Adv. Math.}, 217(6):2664--2699, 2008.
\newblock \doi{10.1016/j.aim.2007.11.012}.

\bibitem[AM14]{AdamovicMilas2014}
Dra\v{z}en Adamovi\'c and Antun Milas.
\newblock {$C_2$}-cofinite {$\mathcal{W}$}-algebras and their logarithmic
  representations.
\newblock In {\em Conformal field theories and tensor categories}, Math. Lect.
  Peking Univ., pages 249--270. Springer, Heidelberg, 2014.

\bibitem[AS02]{AndruskiewitschSchneider2002}
Nicol\'as Andruskiewitsch and Hans-J\"urgen Schneider.
\newblock Pointed {H}opf algebras.
\newblock In {\em New directions in {H}opf algebras}, volume~43 of {\em Math.
  Sci. Res. Inst. Publ.}, pages 1--68. Cambridge Univ. Press, Cambridge, 2002.
\newblock \doi{10.2977/prims/1199403805}.

\bibitem[AS10]{AndruskiewitschSchneider2010}
Nicol\'as Andruskiewitsch and Hans-J\"urgen Schneider.
\newblock On the classification of finite-dimensional pointed {H}opf algebras.
\newblock {\em Ann. of Math. (2)}, 171(1):375--417, 2010.
\newblock \urlprefix\url{https://doi.org/10.4007/annals.2010.171.375}.

\bibitem[BB13]{BazlovBerenstein2013}
Yuri Bazlov and Arkady Berenstein.
\newblock Cocycle twists and extensions of braided doubles.
\newblock In {\em Noncommutative birational geometry, representations and
  combinatorics}, volume 592 of {\em Contemp. Math.}, pages 19--70. Amer. Math.
  Soc., Providence, RI, 2013.
\newblock \doi{10.1090/conm/592/11772}.

\bibitem[BBG17]{BeliakovaBlanchetGeer2017}
Anna {Beliakova}, Christian {Blanchet}, and Nathan {Geer}.
\newblock {Logarithmic Hennings invariants for restricted quantum sl(2)}.
\newblock {\em ArXiv e-prints}, May 2017,
  \burlalt{1705.03083}{http://arxiv.org/abs/1705.03083}.

\bibitem[BN02]{BulacuNauwelaerts2002}
Daniel Bulacu and Erna Nauwelaerts.
\newblock Radford's biproduct for quasi-{H}opf algebras and bosonization.
\newblock {\em J. Pure Appl. Algebra}, 174(1):1--42, 2002.
\newblock \doi{10.1016/S0022-4049(02)00014-2}.

\bibitem[BN03]{BulacuNauwelaerts2003}
Daniel Bulacu and Erna Nauwelaerts.
\newblock Quasitriangular and ribbon quasi-{H}opf algebras.
\newblock {\em Comm. Algebra}, 31(2):657--672, 2003.
\newblock \doi{10.1081/AGB-120017337}.

\bibitem[BN11]{BruguieresNatale2011}
Alain Brugui\`eres and Sonia Natale.
\newblock Exact sequences of tensor categories.
\newblock {\em Int. Math. Res. Not. IMRN}, (24):5644--5705, 2011.
\newblock \doi{10.1093/imrn/rnq294}.

\bibitem[BPVO00]{BulacuPanaiteVanOystaeyen2000}
Daniel Bulacu, Florin Panaite, and Freddy Van~Oystaeyen.
\newblock Quasi-{H}opf algebra actions and smash products.
\newblock {\em Comm. Algebra}, 28(2):631--651, 2000.
\newblock \doi{10.1080/00927870008826849}.

\bibitem[Bru00]{Bruguieres2000}
Alain Brugui\`eres.
\newblock Cat\'egories pr\'emodulaires, modularisations et invariants des
  vari\'et\'es de dimension 3.
\newblock {\em Math. Ann.}, 316(2):215--236, 2000.
\newblock \doi{10.1007/s002080050011}.

\bibitem[BT04]{BulacuTorrecillas2004}
Daniel Bulacu and Blas Torrecillas.
\newblock Factorizable quasi-{H}opf algebras---applications.
\newblock {\em J. Pure Appl. Algebra}, 194(1-2):39--84, 2004.
\newblock \doi{10.1016/j.jpaa.2004.04.010}.

\bibitem[CF06]{CarquevilleFlohr2006}
Nils Carqueville and Michael Flohr.
\newblock Nonmeromorphic operator product expansion and {$C_2$}-cofiniteness
  for a family of {$W$}-algebras.
\newblock {\em J. Phys. A}, 39(4):951--966, 2006.
\newblock \doi{10.1088/0305-4470/39/4/015}.

\bibitem[CGR17]{CreutzigGainutdinovRunkel2017}
Thomas {Creutzig}, Azat~M. {Gainutdinov}, and Ingo {Runkel}.
\newblock {A quasi-Hopf algebra for the triplet vertex operator algebra}.
\newblock {\em ArXiv e-prints}, December 2017,
  \burlalt{1712.07260}{http://arxiv.org/abs/1712.07260}.

\bibitem[DGNO10]{DrinfeldGelakiNikshychEtAl2010}
Vladimir~G. Drinfeld, Shlomo Gelaki, Dmitri Nikshych, and Victor Ostrik.
\newblock On braided fusion categories. {I}.
\newblock {\em Selecta Math. (N.S.)}, 16(1):1--119, 2010.
\newblock \doi{10.1007/s00029-010-0017-z}.

\bibitem[DPR92]{DijkgraafPasquierRoche1992}
Robbert Dijkgraaf, Vincent Pasquier, and Philippe Roche.
\newblock Quasi-{H}opf algebras, group cohomology and orbifold models.
\newblock In {\em Integrable systems and quantum groups ({P}avia, 1990)}, pages
  75--98. World Sci. Publ., River Edge, NJ, 1992.

\bibitem[DR13]{DR13}
Alexei Davydov and Ingo Runkel, 
\newblock{$\mathbb{Z}/2\mathbb{Z}$-extensions of Hopf algebra module categories by their base categories},
\newblock{Adv. Math. 247 (2013) 192--265}.

\bibitem[DR16]{DavydovRunkel2016}
Alexei Davydov and Ingo Runkel,
\newblock {Holomorphic Symplectic Fermions}.
\newblock {\em ArXiv e-prints}, 2016,
  \burlalt{1601.06451}{http://arxiv.org/abs/1601.06451}.

\bibitem[Dri87]{Drinfeld1987}
Vladimir~G. Drinfeld.
\newblock Quantum groups.
\newblock In {\em Proceedings of the {I}nternational {C}ongress of
  {M}athematicians, {V}ol. 1, 2 ({B}erkeley, {C}alif., 1986)}, pages 798--820.
  Amer. Math. Soc., Providence, RI, 1987.

\bibitem[Dri89]{Drinfeld1989}
Vladimir~G. Drinfeld.
\newblock Quasi-{H}opf algebras.
\newblock {\em Algebra i Analiz}, 1(6):114--148, 1989.

\bibitem[EG05]{EtingofGelaki2005}
Pavel Etingof and Shlomo Gelaki.
\newblock On radically graded finite-dimensional quasi-{H}opf algebras.
\newblock {\em Mosc. Math. J.}, 5(2):371--378, 494, 2005.

\bibitem[EGNO15]{EtingofGelakiNikshychEtAl2015}
Pavel Etingof, Shlomo Gelaki, Dmitri Nikshych, and Victor Ostrik.
\newblock {\em Tensor categories}, volume 205 of {\em Mathematical Surveys and
  Monographs}.
\newblock American Mathematical Society, Providence, RI, 2015.
\newblock \doi{10.1090/surv/205}.

\bibitem[FBZ04]{FrenkelBen-Zvi2004}
Edward Frenkel and David Ben-Zvi.
\newblock {\em Vertex algebras and algebraic curves}, volume~88 of {\em
  Mathematical Surveys and Monographs}.
\newblock American Mathematical Society, Providence, RI, second edition, 2004.
\newblock \urlprefix\url{https://doi.org/10.1090/surv/088}.

\bibitem[Fel89]{Felder1989}
Giovanni Felder.
\newblock B{RST} approach to minimal models.
\newblock {\em Nuclear Phys. B}, 317(1):215--236, 1989.
\newblock \doi{10.1016/0550-3213(89)90568-3}.

\bibitem[FF88]{FeuiginFrenkelprime1988}
Boris~L. Feigin and Edward~V. Frenkel.
\newblock A family of representations of affine {L}ie algebras.
\newblock {\em Uspekhi Mat. Nauk}, 43(5(263)):227--228, 1988.
\newblock \doi{10.1070/RM1988v043n05ABEH001935}.

\bibitem[FGR17a]{FarsadGainutdinovRunkel2017}
Vanda {Farsad}, Azat~M. {Gainutdinov}, and Ingo {Runkel}.
\newblock {SL(2,Z)-action for ribbon quasi-Hopf algebras}.
\newblock {\em ArXiv e-prints}, February 2017,
  \burlalt{1702.01086}{http://arxiv.org/abs/1702.01086}.

\bibitem[FGR17b]{FarsadGainutdinovRunkel2017a}
Vanda {Farsad}, Azat~M. {Gainutdinov}, and Ingo {Runkel}.
\newblock {The symplectic fermion ribbon quasi-Hopf algebra and the
  SL(2,Z)-action on its centre}.
\newblock {\em ArXiv e-prints}, 2017,
  \burlalt{1706.08164}{http://arxiv.org/abs/1706.08164}.

\bibitem[FGS]{FuchsGainutdinovSchweigert}
J\"urgen Fuchs, Azat~M. Gainutdinov, and Christoph Schweigert.
\newblock Modularization of finite ribbon categories.
\newblock {\em in preparation}.

\bibitem[FGST06]{FeiginGainutdinovSemikhatovEtAl2006a}
Boris~L. Feigin, Azat~M. Gainutdinov, Alexey~M. Semikhatov, and Ilya~Yu.
  Tipunin.
\newblock The {K}azhdan-{L}usztig correspondence for the representation
  category of the triplet {$W$}-algebra in logorithmic conformal field
  theories.
\newblock {\em Teoret. Mat. Fiz.}, 148(3):398--427, 2006.
\newblock \doi{10.1007/s11232-006-0113-6}.

\bibitem[FHST04]{FuchsHwangSemikhatovEtAl2004}
J\"urgen Fuchs, Stephen Hwang, Alexey~M. Semikhatov, and Ilya~Yu. Tipunin.
\newblock Nonsemisimple fusion algebras and the {V}erlinde formula.
\newblock {\em Comm. Math. Phys.}, 247(3):713--742, 2004.
\newblock \doi{10.1007/s00220-004-1058-y}.

\bibitem[FL17]{FlandoliLentner2017}
Ilaria Flandoli and Simon Lentner.
\newblock Logarithmic conformal field theories of type {$B_n$}, {$\ell=4$} and
  symplectic fermions.
\newblock 2017, \burlalt{1706.07994}{http://arxiv.org/abs/1706.07994}.

\bibitem[FT10]{FeiginTipunin2010}
Boris~L. Feigin and Ilya~Yu. Tipunin.
\newblock Logarithmic cfts connected with simple lie algebras.
\newblock 2010, \burlalt{1002.5047}{http://arxiv.org/abs/1002.5047}.

\bibitem[GPM13]{GeerPatureau-Mirand2013}
Nathan Geer and Bertrand Patureau-Mirand.
\newblock Topological invariants from nonrestricted quantum groups.
\newblock {\em Algebr. Geom. Topol.}, 13(6):3305--3363, 2013.
\newblock \urlprefix\url{https://doi.org/10.2140/agt.2013.13.3305}.

\bibitem[GR17a]{GainutdinovRunkel2017a}
Azat~M. Gainutdinov and Ingo Runkel.
\newblock {Projective objects and the modified trace in factorisable finite
  tensor categories}.
\newblock 2017, \burlalt{1703.00150}{http://arxiv.org/abs/1703.00150}.

\bibitem[GR17b]{GainutdinovRunkel2017}
Azat~M. Gainutdinov and Ingo Runkel.
\newblock Symplectic fermions and a quasi-{H}opf algebra structure on
  {$\overline{U}_{\rm i}s\ell(2)$}.
\newblock {\em J. Algebra}, 476:415--458, 2017.
\newblock \urlprefix\url{https://doi.org/10.1016/j.jalgebra.2016.11.026}.

\bibitem[Hec06]{Heckenberger2006}
Istvan Heckenberger.
\newblock The {W}eyl groupoid of a {N}ichols algebra of diagonal type.
\newblock {\em Invent. Math.}, 164(1):175--188, 2006.
\newblock \urlprefix\url{https://doi.org/10.1007/s00222-005-0474-8}.

\bibitem[Hec08]{Heckenberger2008}
Istvan Heckenberger.
\newblock Lecture notes in nichols algebras, July 2008.

\bibitem[Hec09]{Heckenberger2009}
Istvan Heckenberger.
\newblock Classification of arithmetic root systems.
\newblock {\em Advances in Mathematics}, 220(1):59 -- 124, 2009.
\newblock \doi{https://doi.org/10.1016/j.aim.2008.08.005}.

\bibitem[HLYY15]{HuangLiuYuYe2015}
Hua-Lin Huang, Gongxiang Liu, and Yuping~Yang Yu~Ye.
\newblock Finite quasi-quantum groups of rank two.
\newblock 2015, \burlalt{1508.04248}{http://arxiv.org/abs/1508.04248}.

\bibitem[HLZ]{HuangLepowskyZhang}
Yi-Zhi Huang, James Lepowsky, and Lin Zhang.
\newblock Logarithmic tensor category theory i-viii.
\newblock
  \burlalt{1012.4193,1012.4196,1012.4197,1012.4198,1012.4199,1012.4202,1110.1929,1110.1931}{http://arxiv.org/abs/1012.4193,1012.4196,1012.4197,1012.4198,1012.4199,1012.4202,1110.1929,1110.1931}.

\bibitem[HN99a]{HausserNill1999}
Frank Hausser and Florian Nill.
\newblock Diagonal crossed products by duals of quasi-quantum groups.
\newblock {\em Rev. Math. Phys.}, 11(5):553--629, 1999.
\newblock \urlprefix\url{https://doi.org/10.1142/S0129055X99000210}.

\bibitem[HN99b]{HausserNill1999a}
Frank Hausser and Florian Nill.
\newblock Doubles of quasi-quantum groups.
\newblock {\em Comm. Math. Phys.}, 199(3):547--589, 1999.
\newblock \urlprefix\url{https://doi.org/10.1007/s002200050512}.

\bibitem[Hua05]{Huang2005}
Yi-Zhi Huang.
\newblock Vertex operator algebras, the {V}erlinde conjecture, and modular
  tensor categories.
\newblock {\em Proc. Natl. Acad. Sci. USA}, 102(15):5352--5356, 2005.
\newblock \urlprefix\url{https://doi.org/10.1073/pnas.0409901102}.

\bibitem[Jim85]{Jimbo1985}
Michio Jimbo.
\newblock A {$q$}-difference analogue of {$U({\mathfrak{g}})$} and the
  {Y}ang-{B}axter equation.
\newblock {\em Lett. Math. Phys.}, 10(1):63--69, 1985.
\newblock \urlprefix\url{https://doi.org/10.1007/BF00704588}.

\bibitem[JS93]{JoyalStreet1993}
Andr\`e Joyal and Ross Street.
\newblock Braided tensor categories.
\newblock {\em Advances in Mathematics}, 102(1):20 -- 78, 1993.
\newblock \doi{http://dx.doi.org/10.1006/aima.1993.1055}.

\bibitem[Kac98]{Kac1998}
Victor Kac.
\newblock {\em Vertex algebras for beginners}, volume~10 of {\em University
  Lecture Series}.
\newblock American Mathematical Society, Providence, RI, second edition, 1998.
\newblock \doi{10.1090/ulect/010}.

\bibitem[Kau91]{Kausch1991}
Horst~G. Kausch.
\newblock Extended conformal algebras generated by a multiplet of primary
  fields.
\newblock {\em Phys. Lett. B}, 259(4):448--455, 1991.
\newblock \doi{10.1016/0370-2693(91)91655-F}.

\bibitem[Kau00]{Kausch2000}
Horst~G. Kausch.
\newblock Symplectic fermions.
\newblock {\em Nuclear Phys. B}, 583(3):513--541, 2000.
\newblock \doi{10.1016/S0550-3213(00)00295-9}.

\bibitem[KL01]{KerlerLyubashenko2001}
Thomas Kerler and Volodymyr~V. Lyubashenko.
\newblock {\em Non-semisimple topological quantum field theories for
  3-manifolds with corners}, volume 1765 of {\em Lecture Notes in Mathematics}.
\newblock Springer-Verlag, Berlin, 2001.

\bibitem[KS11]{KondoSaito2011}
Hiroki Kondo and Yoshihisa Saito.
\newblock Indecomposable decomposition of tensor products of modules over the
  restricted quantum universal enveloping algebra associated to
  {${\mathfrak{sl}}_2$}.
\newblock {\em J. Algebra}, 330:103--129, 2011.
\newblock \urlprefix\url{https://doi.org/10.1016/j.jalgebra.2011.01.010}.

\bibitem[Len16]{Lentner2016}
Simon~D. Lentner.
\newblock A {F}robenius homomorphism for {L}usztig's quantum groups for
  arbitrary roots of unity.
\newblock {\em Commun. Contemp. Math.}, 18(3):1550040, 42, 2016.
\newblock \doi{10.1142/S0219199715500406}.

\bibitem[Len17]{Lentner2017}
Simon~D. Lentner.
\newblock {Quantum groups and Nichols algebras acting on conformal field
  theories}.
\newblock 2017, \burlalt{1702.06431}{http://arxiv.org/abs/1702.06431}.

\bibitem[LN15]{LentnerNett2015}
Simon~D. Lentner and Daniel Nett.
\newblock New {$R$}-matrices for small quantum groups.
\newblock {\em Algebr. Represent. Theory}, 18(6):1649--1673, 2015.
\newblock \urlprefix\url{https://doi.org/10.1007/s10468-015-9555-6}.

\bibitem[LO17]{LentnerOhrmann2017}
Simon~D. Lentner and Tobias Ohrmann.
\newblock Factorizable {$R$}-matrices for small quantum groups.
\newblock {\em SIGMA Symmetry Integrability Geom. Methods Appl.}, 13:076, 25
  pages, 2017.
\newblock \urlprefix\url{https://doi.org/10.3842/SIGMA.2017.076}.

\bibitem[Lus90]{Lusztig1990}
George Lusztig.
\newblock Quantum groups at roots of {$1$}.
\newblock {\em Geom. Dedicata}, 35(1-3):89--113, 1990.
\newblock \urlprefix\url{https://doi.org/10.1007/BF00147341}.

\bibitem[Lus93]{Lusztig1993}
George Lusztig.
\newblock {\em Introduction to quantum groups}, volume 110 of {\em Progress in
  Mathematics}.
\newblock Birkh\"auser Boston, Inc., Boston, MA, 1993.

\bibitem[Mac52]{MacLane1952}
Saunders MacLane.
\newblock Cohomology theory of {A}belian groups.
\newblock In {\em Proceedings of the {I}nternational {C}ongress of
  {M}athematicians, {C}ambridge, {M}ass., 1950, vol. 2}, pages 8--14. Amer.
  Math. Soc., Providence, R. I., 1952.

\bibitem[Maj98]{Majid1998}
Shawn Majid.
\newblock Quantum double for quasi-{H}opf algebras.
\newblock {\em Lett. Math. Phys.}, 45(1):1--9, 1998.
\newblock \doi{10.1023/A:1007450123281}.

\bibitem[MS89]{MooreSeiberg1989}
Gregory Moore and Nathan Seiberg.
\newblock Classical and quantum conformal field theory.
\newblock {\em Comm. Math. Phys.}, 123(2):177--254, 1989.
\newblock \urlprefix\url{http://projecteuclid.org/euclid.cmp/1104178762}.

\bibitem[NT09]{NagatomoTsuchiya2009}
Kiyokazu {Nagatomo} and Akihiro {Tsuchiya}.
\newblock {The Triplet Vertex Operator Algebra W(p) and the Restricted Quantum
  Group at Root of Unity}.
\newblock {\em ArXiv e-prints}, February 2009,
  \burlalt{0902.4607}{http://arxiv.org/abs/0902.4607}.

\bibitem[Rad12]{Radford2012}
David~E. Radford.
\newblock {\em Hopf algebras}, volume~49 of {\em Series on Knots and
  Everything}.
\newblock World Scientific Publishing Co. Pte. Ltd., Hackensack, NJ, 2012.

\bibitem[Ros98]{Rosso1998}
Marc Rosso.
\newblock Quantum groups and quantum shuffles.
\newblock {\em Invent. Math.}, 133(2):399--416, 1998.
\newblock \doi{10.1007/s002220050249}.

\bibitem[RSTS88]{ReshetikhinSemenov-Tian-Shansky1988}
N.~Yu. Reshetikhin and M.~A. Semenov-Tian-Shansky.
\newblock Quantum {$R$}-matrices and factorization problems.
\newblock {\em J. Geom. Phys.}, 5(4):533--550 (1989), 1988.
\newblock \doi{10.1016/0393-0440(88)90018-6}.

\bibitem[R14]{Runkel14} Ingo Runkel, A braided monoidal category for free super-bosons, J. Math. Phys. 55 (2014) 041702.


\bibitem[Sch02]{Schauenburg2002}
Peter Schauenburg.
\newblock Hopf modules and the double of a quasi-{H}opf algebra.
\newblock {\em Trans. Amer. Math. Soc.}, 354(8):3349--3378, 2002.
\newblock \doi{10.1090/S0002-9947-02-02980-X}.

\bibitem[Sch04]{Schauenburg2004}
Peter Schauenburg.
\newblock A quasi-{H}opf algebra freeness theorem.
\newblock {\em Proc. Amer. Math. Soc.}, 132(4):965--972, 2004.
\newblock \doi{10.1090/S0002-9939-03-07181-8}.

\bibitem[Sch05]{Schauenburg2005}
Peter Schauenburg.
\newblock Quotients of finite quasi-{H}opf algebras.
\newblock In {\em Hopf algebras in noncommutative geometry and physics}, volume
  239 of {\em Lecture Notes in Pure and Appl. Math.}, pages 281--290. Dekker,
  New York, 2005.

\bibitem[Shi16]{Shimizu2016}
Kenichi Shimizu.
\newblock Non-degeneracy conditions for braided finite tensor categories.
\newblock 2016, \burlalt{1602.06534}{http://arxiv.org/abs/1602.06534}.

\bibitem[Voc10]{Vocke2010}
Karolina Vocke.
\newblock Konstruktion von anyonenmodellen aus projektiven
  yetter-drinfeld-moduln.
\newblock Master's thesis, Ludwigs-Maximilians-Universit{\"a}t M{\"u}nchen,
  2010.

\bibitem[Wak86]{Wakimoto1986}
Minoru Wakimoto.
\newblock Fock representations of the affine {L}ie algebra {$A^{(1)}_1$}.
\newblock {\em Comm. Math. Phys.}, 104(4):605--609, 1986.
\newblock \urlprefix\url{http://projecteuclid.org/euclid.cmp/1104115171}.

\end{thebibliography}
\bibliographystyle{halpha}

\end{document}